\documentclass[10pt,a4paper,reqno]{amsart}
\usepackage{bm}

\usepackage{setspace}
\usepackage{indentfirst}

\usepackage{tabls}
\usepackage{url}
\usepackage{color}
\definecolor{refkey}{rgb}{1,0,0}
\definecolor{labelkey}{rgb}{0,0,1}
\definecolor{labelkey}{rgb}{1,1,1}
\usepackage[linktocpage = true]{hyperref}
\hypersetup{
 colorlinks = true,
 citecolor = blue,
}

\usepackage{times}
\usepackage{amsthm}
\usepackage{amsmath}
\usepackage{amsfonts}
\usepackage{amssymb}
\usepackage{amsrefs}
\usepackage{mathrsfs}
\usepackage{xypic}
\usepackage{enumerate}
\usepackage{graphicx}
\usepackage{subfigure}
\usepackage[toc,page]{appendix}
\usepackage[all,cmtip]{xy}

\usepackage{tikz-cd}

\allowdisplaybreaks

\hypersetup{pdfpagemode=UseOutlines,colorlinks=false,pdfpagelayout=SinglePage,pdfstartview=FitH,bookmarksopen=true}

\theoremstyle{plain}
\newtheorem{theorem}{Theorem}[section]
\newtheorem{lemma}[theorem]{Lemma}
\newtheorem{corollary}[theorem]{Corollary}
\newtheorem{definition}[theorem]{Definition}
\newtheorem{proposition}[theorem]{Proposition}

\newtheorem{example}[theorem]{Example}
\newtheorem{remark}[theorem]{Remark}

\theoremstyle{definition}

\newtheorem*{note*}{Note}

\theoremstyle{remark}

\DeclareMathOperator{\ad}{ad} \DeclareMathOperator{\Ad}{Ad}

\DeclareMathOperator{\Hom}{Hom}

\DeclareMathOperator{\Lie}{Lie}

\newcommand{\ZZ}{\mathbb{Z}}

\newcommand{\XX}{\mathfrak{X}} 

\newcommand{\toto}{\rightrightarrows}

\newcommand{\ba}[2]{[#1,#2]}

\newcommand{\abs}[1]{\left\vert#1\right\vert}

\newcommand{\dA}{d_A }

\newcommand{\CIM}{C^{\infty}(M)}
\newcommand{\inserts}{\iota}

\newcommand{\set}[1]{\left\{#1\right\}}

\newcommand{\Img}{\mathrm{img} }

\newcommand{\multivectorfields}[2]{\mathfrak{X}^{#1}(#2)}

\newcommand{\rhopush}{D_{\rho}}

\newcommand{\minuspower}[1]{(-1)^{#1}}

\newcommand{\jet}{\mathfrak{J}}
\newcommand{\heat}{\mathfrak{H}}
\newcommand{\core}{\mathfrak{h}}
\newcommand{\liftingd}{\jmath^1}

\newcommand{\Gpd}{\mathcal{G}}
\newcommand{\Smap}{\mathfrak{e}}
\newcommand{\SSmap}{\mathfrak{E}}
\newcommand{\invstar}{\mathrm{inv}_*}
\newcommand{\gaugeequiv}{\Bumpeq}
\newcommand{\rank}{\mathrm{rank}}

\newcommand{\Top}{\mathrm{top}}
\newcommand{\CinfM}{C^\infty(M)}
\newcommand{\Der}{\mathrm{Der}}

\newcommand{\Apolysection}[1]{\Gamma( \wedge^{#1} A)}
\newcommand{\littlecirc}{}
\newcommand{\differentials}{\mathfrak{D}\mathrm{iff}}
\newcommand{\Reduced}{{R}}
\newcommand{\ReducedProper}{R_{\mathrm{\rho}}}

\newcommand{\coHg}{\mathrm{H}}

\newcommand{\deltazero}{\delta^{(0)}}
\newcommand{\deltaone}{\delta^{(1)}}

\newcommand{\basicconnection}{\nabla^{\mathrm{bas}}}
\newcommand{\basiccurvature}{R_\nabla^{\mathrm{bas}}}
\newcommand{\Dnabla}{D_{\nabla}}
\newcommand{\adnabla}{\mathrm{ad}_{\nabla}}
\newcommand{\adjUTH}{\mathrm{ad}}

\newcommand{\Defcomplex}{\mathrm{C}_{\mathrm{Def}}}
\newcommand{\OmegaA}{\Omega(A)}
\newcommand{\dCE}{d_{\mathrm{CE}}}

\begin{document}

\title{On the reduced space of multiplicative multivectors}

\author{Zhuo Chen}
\address{Department of Mathematics, Tsinghua University, Beijing 100084, China}
\email{\href{mailto:chenzhuo@tsinghua.edu.cn}{chenzhuo@tsinghua.edu.cn}}

\author{Honglei Lang}
\address{College of Science, China Agricultural University, Beijing 100083, China}
\email{\href{mailto:hllang@cau.edu.cn}{hllang@cau.edu.cn}(corresponding author)}

\author{Zhangju Liu}
\address{Department of  Mathematics, Peking University, Beijing 100871, China}
\email{\href{mailto:liuzj@pku.edu.cn}{liuzj@pku.edu.cn}}



























\begin{center}
	\emph{To the memory of Professor Kirill C. H. Mackenzie (1951-2020)}
\end{center}

\makeatother

\begin{abstract}
	A strict Lie $2$-algebra $\Gamma(\wedge^\bullet A) \stackrel{T}{\rightarrow} \mathfrak{X}_{\mathrm{mult}}^\bullet(\Gpd )$ is associated with any Lie groupoid $\Gpd$. Here, $\Gamma(\wedge^\bullet A)$ is the Schouten algebra of the tangent Lie algebroid $A$ of $\Gpd$ and $\mathfrak{X}_{\mathrm{mult}}^\bullet(\Gpd )$ is the space of multiplicative multivectors on $\Gpd$. The quotient ${\Reduced}_{\mathrm{mult}}^\bullet:=\mathfrak{X}_{\mathrm{mult}}^\bullet(\Gpd )/\Img T$,  a Morita invariant of $\Gpd$, is called the reduced space of multiplicative multivectors. We prove a canonical decomposition formula of elements in $\mathfrak{X}_{\mathrm{mult}}^\bullet(\Gpd )$ and establish a key relation between ${\Reduced}_{\mathrm{mult}}^k$ and the cohomology $\coHg ^1(\jet \Gpd,\wedge^k A)$ where $\jet \Gpd$ is the jet groupoid of $\Gpd$ and $1\leqslant k\leqslant \rank A$.  We also study ${\Reduced}_{\mathrm{diff}}^\bullet $, the reduced space of Lie algebroid differentials on $A$. By taking infinitesimals, $\bar{\delta}: $ ${\Reduced}_{\mathrm{mult}}^\bullet $ $\to $ ${\Reduced}_{\mathrm{diff}}^\bullet $, the two reduced spaces are related. We find that the kernel of $\bar{\delta}$ is isomorphic to the kernel of the Van Est map $\coHg ^1(\Gpd,\wedge^k \ker\rho)\to \coHg ^1(A,\wedge^k \ker\rho)$, where $\rho$ is the anchor of $A$.



\end{abstract}

\maketitle

\tableofcontents

\section{Introduction}

The concept of multiplicativity on Lie groupoids, which includes multiplicative multivectors and multiplicative forms\footnote{In this paper, we abbreviate `multivector field' to `multivector' and `$k$-vector field' to `$k$-vector'.}, has been widely studied. Its origins can be traced back to Drinfeld's research on Poisson Lie groups \cite{Drinfeld}, which are Lie groups with a compatible Poisson structure. Subsequent to this initial work, symplectic groupoids \cites{Karasev,W2} and Poisson groupoids \cite{W1} have been defined and studied. Multiplicativity has also been generalized in recent times to include multivectors \cites{ILX,LuW,MX1,MX2}, forms \cites{VE,BC,BCO,C}, and tensors \cite{BD}. Moreover, there are multiplicative structures that are compatible with contact and Jacobi structures \cites{CZ,IM}, holomorphic structures \cites{LSX1,XuCamile}, foliations \cites{DJO,H,JO}, Dirac structures \cites{BCWZ,JL,O}, Manin pairs \cite{LS}, generalized complex structures \cite{JSX}, and other structures \cite{ETV}. A thorough survey of multiplicativity, from its inception to its current generalizations, is provided by \cite{Kosmann}. For a comprehensive discussion of the theory of Lie groupoids and Lie algebroids, the reader is referred to the standard text \cite{Mackenzie}. Basic notions and notations that will be used throughout the paper can be found in Appendix \ref{Sec:LALGPbasic}.

\subsection{Reduced space of multiplicative multivectors}  
In this paper, we will define the \textit{reduced space of multiplicative multivectors} on Lie groupoids. The main reason for studying such a space is due to a certain homotopy invariant that is associated with the Morita equivalence class of Lie groupoids representing a given differentiable stack. A recent paper by Bonechi, Ciccoli, Laurent-Gengoux, and Xu \cite{BCLX} proved that the homotopy class of a certain graded Lie 2-algebra associated with a Lie groupoid is invariant under Morita equivalences between Lie groupoids. This result led the authors to define the reduced space of multiplicative multivectors on a differentiable stack as such a homotopy class modulo a well-known equivalence relation. Partial results in this direction are also achieved by Berwick-Evans-Lerman \cite{BLlie2algebra} and Ortiz-Waldron \cite{OrtizWaldron17}.

    Let $\Gpd $ be  a Lie groupoid over a smooth manifold $M$ whose   source and target maps
   are denoted by $s$ and  $t$, respectively. Its tangent Lie algebroid is denoted by $(A, [~,~], \rho)$. 
      We denote by $\multivectorfields{k}{\Gpd }=\Gamma(\wedge^k T\Gpd )$ the space of multivectors of order $k$, or $k$-vectors for short.
   A multiplicative $0$-vector on $\Gpd$ is a function $F\in C^\infty(\Gpd)$ such that $F(gr)=F(g)+F(r)$ for all $(g,r)\in \Gpd^{(2)}$.  Here $\Gpd ^{(2)}$ denotes
   the set of composable pairs, i.e. pairs $(g,{r})$ with $s(g)=t({r})$.
   For $ k\geqslant  1$, recall that a $k$-vector $\Pi\in
   \multivectorfields{k}{\Gpd }$ is said to be {\bf multiplicative} if the
   graph of  groupoid multiplication
   $$\Lambda=\set{(g,{r},g{r})|t({r})=s(g)}\subset \Gpd \times\Gpd \times\Gpd $$
   is a coisotropic submanifold with respect to
   $\Pi\times\Pi\times\minuspower{k-1}\Pi$ (see   \cite{ILX} for more explanations).

   Denote by $ \mathfrak{X}_{\mathrm{mult}}^k(\Gpd )$   the space   of multiplicative $k$-vectors on $\Gpd $.
   Due to dimension reasons, it suffices to consider those $k$ such that $0\leqslant  k\leqslant  \Top+1$. Here and throughout the paper, $\Top$ denotes the integer $\rank A$.
    We will study multiplicative $k$-vectors  (and $k$-differentials on Lie algebroids) in three situations:
    \begin{itemize}
    	\item[(1)]$0\leqslant  k\leqslant  \Top +1$, which we call the \textit{generic} case;
    	\item[(2)]$1\leqslant  k\leqslant  \Top $, which we call the \textit{ordinary} case;
    	\item[(3)]$k=0$ or $k=\Top +1$, which we call the \textit{exceptional} cases.
    \end{itemize}

 A particular type of multiplicative multivectors, called \textbf{exact} ones, is constructed as follows.
   Let $\tau\in \Gamma(\wedge^k A)$ be given. The $k$-vector on $\Gpd$ $$ T(\tau):=\overrightarrow{\tau }-\overleftarrow{\tau }$$
   is multiplicative. Here $\overrightarrow{\tau}$ (resp. $\overleftarrow{\tau}$) denotes  the right-invariant (resp. left-invariant) $k$-vector on $\Gpd $ corresponding to $\tau$, i.e.,
   $$
   \overrightarrow{\tau}|_{g}= R_{g*} ( \tau|_{t(g)}),\qquad \forall g\in \Gpd
   $$
   and
   $$
   \overleftarrow{\tau}= (-1)^k \invstar  (\overrightarrow{\tau}),
   $$
   where    $R_g$ stands for the right multiplication by $g$ and $\mathrm{inv}$   the groupoid inversion     (see Appendix \ref{Sec:LALGPbasic}).

    We denote by $\mathfrak{T}^k(\Gpd)$ the space of such exact multiplicative $k$-vectors.
    One can prove that the subspace $\mathfrak{T}^\bullet(\Gpd)=\oplus_{k= 0}^\Top \mathfrak{T}^k(\Gpd)$ formed by exact multiplicative multivectors is an ideal in the graded Lie algebra $\mathfrak{X}_{\mathrm{mult}}^\bullet(\Gpd )$ $=$ $\oplus_{k= 0}^{\Top+1}$ $\mathfrak{X}_{\mathrm{mult}}^k(\Gpd )$ (equipped with the Schouten bracket).

Recall a key point of view in the study of differentiable stacks  \cites{BLlie2algebra,BCLX} --- Given a  representative of a differentiable stack $\mathbf{G}$ by a Lie groupoid $\Gpd$, the space of multiplicative multivectors   on $\Gpd$  is part of a (strict and $\ZZ$-graded) Lie $2$-algebra:
\begin{equation}
\label{Eqt:Lie2algebra}
\Gamma(\wedge^\bullet A)  \stackrel{T}{\longrightarrow} \mathfrak{X}_{\mathrm{mult}}^\bullet(\Gpd ).
\end{equation}

Here by saying a (strict and $\ZZ$-graded) Lie $2$-algebra (also known as a crossed module of graded Lie algebras) we mean a pair of  ($\ZZ$-graded) Lie algebras $\mathfrak{A}$ and $\mathfrak{B}$ together with
\begin{itemize}
	\item a morphism of   graded Lie algebras $t: \mathfrak{A}\to \mathfrak{B}$, and
	\item a   graded Lie algebra action of $\mathfrak{B}$ on the graded vector space $\mathfrak{A}$
	\[\mathfrak{B}\times \mathfrak{A}\to \mathfrak{A},\qquad (u,a)\mapsto u\cdot a,\]
\end{itemize} such that the following two identities hold true: for all $a,a_1,a_2\in \mathfrak{A}$ and $u\in \mathfrak{B}$,
\begin{eqnarray*}
	t(u\cdot a) &=&[u,t a ],\\
	(t a_1) \cdot a_2&=&[a_1,a_2].
\end{eqnarray*}

For the particular Lie  $2$-algebra \eqref{Eqt:Lie2algebra}, we take $\mathfrak{A}=\Gamma(\wedge^\bullet A)[1]$, $\mathfrak{B}=\mathfrak{X}_{\mathrm{mult}}^\bullet(\Gpd )[1]$, $t=T$ and $\cdot$ defined in the obvious manner
$$
\Pi\cdot \tau=[\Pi,\overrightarrow{\tau }]|_M,\qquad \forall ~\Pi\in \mathfrak{X}_{\mathrm{mult}}^\bullet (\Gpd ),\tau\in \Gamma(\wedge^\bullet A).
$$

It is  proved in \cites{BLlie2algebra,OrtizWaldron17} for the $\bullet=1$ case, and then extended to the general case in \cite{BCLX}, that the   Lie $2$-algebras associated in this way to Morita equivalent Lie groupoids are necessarily homotopy equivalent. Therefore, in  \cite{BCLX},   the space of multivector fields on a
differentiable stack $\mathbf{G}$ is defined to be the homotopy
equivalence class of the Lie $2$-algebras \eqref{Eqt:Lie2algebra} associated with the Lie groupoids representing
the differentiable stack $\mathbf{G}$.

Two elements $\Pi$ and $\Pi'$ in $\mathfrak{X}_{\mathrm{mult}}^k(\Gpd )$ are said to be  \textbf{equivalent} (notation: $\Pi\gaugeequiv \Pi'$)  if their subtraction is an exact multiplicative $k$-vector. This equivalence relation appears as early as in \cites{LuThesis,Waffine}.   The quotient space
$${\Reduced}_{\mathrm{mult}}^\bullet:=\mathrm{coker}(T)=\mathfrak{X}_{\mathrm{mult}}^\bullet(\Gpd )/\mathfrak{T}^\bullet(\Gpd )=\mathfrak{X}_{\mathrm{mult}}^\bullet(\Gpd )/\Img T , $$which we  call  the \textbf{reduced  space of multiplicative multivectors} on $\Gpd$, classifies multiplicative multivectors up to  equivalence. By results in \cites{BLlie2algebra,OrtizWaldron17,BCLX} mentioned as above,  the space ${\Reduced}_{\mathrm{mult}}^\bullet$ is a Morita invariant: for any two Morita equivalent groupoids $\Gpd$ and $\Gpd'$, their respective ${\Reduced}_{\mathrm{mult}}^\bullet$ are isomorphic  Lie algebras\footnote{We remark that, besides ${\Reduced}_{\mathrm{mult}}^\bullet$, there is another obvious Morita invariant  $\ker T$ which can be  identified with the  Lie groupoid cohomology $  \coHg^0(\Gpd,\Gamma(\wedge^\bullet \ker \rho))$,
	i.e.  $\Gpd$-invariant sections in the bundle $\wedge^\bullet \ker \rho$ (which might be singular).}. One of our purposes   is to investigate ${\Reduced}_{\mathrm{mult}}^\bullet$, eventually providing a 
new description for it.


 Any Lie group can be thought of as a Lie groupoid over a single point.  It is well-known that the associated reduced space is isomorphic to a degree $1$ group cohomology. In fact,  if $\Gpd$ is a Lie group (i.e., $M$ is a single point),  then   each   $\Pi$  $\in$ $\mathfrak{X}_{\mathrm{mult}}^k(\Gpd )$ can be written in the form
\begin{equation}\label{Eqt:PigLieGroupCase}
\Pi_g=R_{g*}c(g ),\qquad \forall~ g\in\Gpd.\end{equation}
Here $c\colon \Gpd\to \wedge^k A$ is a Lie group $1$-cocycle valued in $\wedge^k A$ which admits the standard adjoint action by the Lie group $\Gpd$ \cite{LuThesis}. Moreover, $\Pi$ is exact if and only if $c$ is a coboundary. So  the space     ${\Reduced}_{\mathrm{mult}}^k=\mathfrak{X}_{\mathrm{mult}}^k(\Gpd )/\mathfrak{T}^k(\Gpd)$ is isomorphic to $\coHg^1(\Gpd, \wedge^k A)$.

Thus, it is natural to investigate, for general groupoids, the reduced space ${\Reduced}_{\mathrm{mult}}^\bullet$ and its relation to the cohomologies of specific \textit{Lie groupoid modules}. 
    However, the  approach for a Lie group mentioned above fails for a Lie groupoid $\Gpd$ over a nontrivial base manifold $M$. In fact, the reduced space ${\Reduced}_{\mathrm{mult}}^k$ does \textit{not} admit a characterization $\coHg^1(\Gpd, \wedge^k A)$ as in the group case and not even an adjoint action exists on $\wedge^k A$ by $\Gpd$. Nevertheless, an interesting construction of ${\Reduced}_{\mathrm{mult}}^\bullet$ as a pullback is provided (Theorem \ref{Thm:Rmultaspullback}) by involving cohomologies of both the jet groupoid $\jet \Gpd$ (with respect to its adjoint action on $\wedge^\bullet A$) and the bundle of isotropy jet Lie groups of $\jet \Gpd$ itself.


\subsection{Reduced space of Lie algebroid differentials} The infinitesimal of a multiplicative multivector on $\Gpd$ is {a differential  on $A$.} 
Recall that a {\bf $k$-differential} on a Lie algebroid $A$ is
a derivation of degree $(k-1)$ on the exterior algebra $\Gamma(\wedge^\bullet A)$: $$\delta:~\Gamma(\wedge^\bullet A)\to \Gamma(\wedge^{\bullet+k-1} A),$$
which is compatible with the Schouten bracket:
\begin{equation*}\label{Eqt:kdifferentialcocyclecondition}
\delta [X,Y]=[\delta (X),Y]+(-1)^{(k-1)(\abs{X}-1)}[X,\delta (Y)],\qquad X, Y\in \Gamma(\wedge^\bullet A).\end{equation*}This notion was introduced in
\cite{ILX} (also called a morphic $k$-vector in \cites{BCO,MX2}). Again, it suffices to consider those $0\leqslant  k\leqslant  \Top+1$, where $\Top=\rank A$.
We denote by   $ \differentials^k(A)$  the space  of   $k$-differentials on $A$, and by $\differentials^\bullet(A)=\bigoplus_{k= 0}^{\Top  +1}\differentials^k(A)$ the total space. Equipped with the standard commutator
$$
[\delta,\delta']:=\delta{\littlecirc} \delta'-(-1)^{(k-1)(k'-1)}\delta'{\littlecirc} \delta,\quad \delta\in \differentials^k(A),\delta'\in \differentials^{k'}(A),
$$
the space $\differentials^\bullet(A)$ is a degree $(-1)$ graded Lie algebra. It is clear that $\differentials^\bullet(A)$ coincides with the Lie algebra of derivations $\mathrm{Der}^{\bullet -1}(\Gamma(\wedge^\bullet A)[1])$.

Given $\tau\in \Gamma(\wedge^k A)$, the operation  $[\tau,\cdot~]$ defines a $k$-differential. We call $[\tau,\cdot~]$ an \textbf{exact} $k$-differential and denote by $\mathfrak{A}^k(A) $ the collection of such exact $k$-differentials.
The subspace   $\mathfrak{A}^\bullet(A)= \oplus_{k=0}^{\Top }\mathfrak{A}^k(A)$   is an ideal in the graded Lie algebra $\differentials^\bullet(A)$.

 Two $k$-differentials $\delta $ and $\delta'$  are said to be {\bf  equivalent} (notation: $\delta\gaugeequiv \delta'$)
	if they are differed by an exact $k$-differential.   The  problem of classifying  differentials on $A$ up to equivalence, i.e., to characterize the space
$$
{\Reduced}^k_{\mathrm{diff}}:=\differentials^k(A)/\mathfrak{A}^k(A)\,,
$$
is also one of the purposes of this paper. We call ${\Reduced}^k_{\mathrm{diff}}$ the \textbf{reduced  space of $k$-differentials}, which is the   infinitesimal version of  the reduced  space ${\Reduced}_{\mathrm{mult}}^k$ introduced as earlier. Note that ${\Reduced}^\bullet_{\mathrm{diff}}$ inherits a standard degree $(-1)$ graded Lie algebra structure from   $\differentials^\bullet(A)$.

The result of Arias Abad and Crainic \cite[Proposition 4.6]{AC2} inspired us to further analyze the relationship between ${\Reduced}^k_{\mathrm{diff}}$ and Lie algebroid cohomologies. According to the said proposition, ${\Reduced}^k_{\mathrm{diff}}$ is isomorphic to $\coHg^1 (A,\wedge^k \adjUTH)$, which refers to the degree $1$ cohomology of the $k$-th exterior power of the adjoint representation up to homotopy of $A$.



\subsection{Outline of the paper} Below is a brief account of the main results and structure of this paper.

In the preliminary Section \ref{Sec:pre},  {we  review   a characterization of the reduced space of differentials for the generic case  $0\leqslant  k \leqslant  \Top +1$, i.e. ${\Reduced}_{\mathrm{diff}}^k= \coHg _{\mathrm{Def}}^{1,k-1}(A)$,   the Lie algebroid deformation cohomology due to Arias Abad and Crainic \cite{AC2} }
(see Theorem \ref{Thm:gaugespacekdifferential1}). We also recall a result characterizing ${\Reduced}_{\mathrm{diff}}^k$ in terms of representation up to homotopy on $\ker\rho\oplus  TM/\mathrm{img}\rho$ when the Lie algebroid $A$ is regular (see Corollary \ref{Cor:regularReducedDiff}).

 We then introduce a particular type of tensor fields $\pi\in \Gamma(TM\otimes (\wedge^{k-1} A))$ which we call $\rho$-compatible $k$-tensors (see Definition \ref{Defn:properktensor}) and will be needed to build multiplicative $k$-vectors and $k$-differentials.

In Section \ref{Sec:transtivedifferentials}, we study Lie algebroid differentials and groupoid multiplicative multivectors under the   transitive assumption, i.e. $\rho: A\to TM$ is surjective. In this case, any Lie algebroid differential $\delta\in \differentials^k(A)$ admits a nice expression (see Theorem \ref{Thm:s1-1}):
$$\delta= [\Lambda,~\cdot~]+\Omega$$
where $\Lambda\in \Gamma(\wedge^k A)$, $\Omega\in Z^1(A,\wedge^k \ker  \rho)$. The phenomenon parallel to this on the Lie groupoid side is that any multiplicative $k$-vector $\Pi\in \mathfrak{X}_{\mathrm{mult}}^k(\Gpd )$ can be written in the form (see Theorem \ref{Thm:S1-1}):
$$\Pi=\overrightarrow{\Lambda}-\overleftarrow{\Lambda}+\Pi_{\mathcal{F}}$$
where $\Lambda$ is as before, $\mathcal{F}$ $\in Z^1(\Gpd,\wedge^k \ker  \rho)$, and $\Pi_{\mathcal{F}}$ denotes the corresponding multiplicative multivector arising from $\mathcal{F}$. Although these two decompositions are not canonical, we can use them to   find how the reduced spaces ${\Reduced}_{\mathrm{diff}}^k$ and ${\Reduced}_{\mathrm{mult}}^k$ are related to the usual   cohomology of (transitive) Lie algebroids and groupoids.   Indeed, we can recover that
${\Reduced}_{\mathrm{diff}}^k \cong  \coHg  ^1(A,\wedge^k \ker\rho)$, which first appeared in \cite{AC2} (see Corollary        \ref{Cor:transitivecasetorecover}), and
similarly, we have   ${\Reduced}_{\mathrm{mult}}^k \cong  \coHg ^1(\Gpd,\wedge^k \ker\rho) $
(see     Theorem \ref{Thm:transitiveRmult}).
Note that, for a transitive groupoid $\Gpd$, one has $\coHg ^1(\Gpd,\wedge^k \ker\rho) \cong \coHg ^1(G_x,\wedge^k \mathfrak{g}_x) $, where $G_x$ is the {isotropy} group at $x\in M$ and $\mathfrak{g}_x$ is the tangent Lie algebra of $G_x$ (see \cite{BCLX}). So, if $G_x$ is simply-connected and semi-simple, then we conclude that ${\Reduced}_{\mathrm{mult}}^k=0$.
Moreover, since Crainic \cite{Crainic2003}  proved  that   $\coHg ^1(G_x,\wedge^k \mathfrak{g}_x)=0$ for every proper groupoid, we have   ${\Reduced}_{\mathrm{mult}}^k=0$  for proper transitive groupoids.  

Without the transitive assumption, $\differentials^k(A)$, $\mathfrak{X}_{\mathrm{mult}}^k(\Gpd )$,  ${\Reduced}_{\mathrm{diff}}^k$, and ${\Reduced}_{\mathrm{mult}}^k$ are more complicated, and they are the  subject for investigation in   Sections \ref{Sec:2ndclassofdifferentials} and   \ref{Sec:multiplicativemultivectors}. As the starting point of our approach, we will show that in the ordinary case  $1\leqslant  k \leqslant  \Top $, each $k$-differential $\delta\in \differentials^k(A)$ is uniquely determined by a pair $( \chi , \pi)$, called the \textbf{characteristic pair} of $\delta$ (see Proposition \ref{Prop:kdifferentialinpair}
), where
$\chi:~ \jet A\to \wedge^k A$ is a $1$-cocycle with respect to the adjoint action of $\jet A$  on $\wedge^k A$ and $\pi\in \Gamma(TM\otimes (\wedge^{k-1} A))$  is a $\rho$-compatible $k$-tensor. Here $\jet A$ is the jet Lie algebroid of $A$  (see Appendix \ref{Appendix:jetLiealgebroid}).

The significance of  characteristic pair is that it provides a canonical decomposition of   $\delta\in \differentials^k(A)$ into two components $\chi$ and $\pi$. We then introduce  a map  $\kappa:~\differentials^k(A)  \to \coHg ^1(\jet A,\wedge^k A)$ that sends $
\delta $   to $[\chi ] $.
 Here $\coHg ^1(\jet A,\wedge^k A)$ is the degree $1$ Chevalley-Eilenberg cohomology of the jet Lie algebroid $\jet A$ with respect to its adjoint action on $\wedge^k A$.
We will prove that    $\kappa$ induces an embedding ${\Reduced}_{\mathrm{diff}}^k \stackrel{\kappa}{\hookrightarrow}\coHg ^1(\jet A,\wedge^k A)$, and  there is an embedding $\ReducedProper^k\stackrel{\mu}{\hookrightarrow} \coHg ^1(\core,\wedge^k A)$. Here $\ReducedProper^k$ is the   {reduced space of $\rho$-compatible $k$-tensors} (see Definition  \ref{def:reducedproperspace}) and $\core$ is the bundle of {isotropy} jet Lie algebras of   $\jet A$.
Our main result is that the reduced space of $k$-differentials ${\Reduced}_{\mathrm{diff}}^k$ fits into a pullback diagram
\begin{equation}\label{Digram:pullbackreduceddiffk}
\xymatrix{
	  \Reduced^k_{\mathrm{diff}} \ar@{^{(}->}[d]_{\kappa} \ar[r]^{i^!} & \ReducedProper^k \ar@{^{(}->}[d]^{\mu} \\
	  \coHg ^1(\jet A,\wedge^k A) \ar[r]^{i^*} & \coHg ^1(\core,\wedge^k A)  }
\end{equation}
(see Theorem \ref{Thm:Rdiffaspullback}).  The pullback diagram \eqref{Digram:pullbackreduceddiffk} tells us how $\Reduced^k_{\mathrm{diff}}$ is determined by $\ReducedProper^k$ and $\coHg ^1(\jet A,\wedge^k A)$  (when $1\leqslant  k \leqslant  \Top $).

One of the important contents of this article is to analyze the data of a multiplicative $k$-vector $\Pi\in \mathfrak{X}_{\mathrm{mult}}^k(\Gpd )$.
In the ordinary case $1\leqslant  k \leqslant  \Top $,
 our first observation is that  every $\Pi\in \mathfrak{X}_{\mathrm{mult}}^k(\Gpd )$ is uniquely determined by a pair $(c,\pi)$, where $c:~\jet \Gpd\to \wedge^k A$ is a $1$-cocycle with respect to the adjoint action of $\jet \Gpd$  (the jet Lie groupoid; see Appendix \ref{Appendix:jetLiegroupoid})  on $\wedge^k A$  and $\pi\in\Gamma(TM\otimes (\wedge^{k-1} A))$ is a $\rho$-compatible $k$-tensor.
Indeed, in Theorem \ref{Thm:mainexplicitformula} we show a canonical decomposition      of $\Pi$ in terms of  $(c,\pi)$   by the following formula
\begin{equation}\label{Eqt:PiGroupoidcase}\Pi_g=R_{g*}c([b_g])+L_{[b_g]} \frac{1-e^{-\rhopush}}{\rhopush} (\pi)_{s(g)},
\end{equation}
where $g\in \Gpd $,    $b_g$ is a local bisection through $g$, and $\frac{1-e^{-\rhopush}}{\rhopush} (\pi)$ is a section in $\wedge^k (TM\oplus A)$ determined by $\pi$ (see Equation   \eqref{Bpi}). This formula is due to Iglesias-Ponte,  Laurent-Gengoux, and  Xu because in an earlier version of their work \cite{ILX} (by private communication),  they  discovered an equivalent form of the formula. But it is not shown in
their published   article.
Our contribution is to present it in a neat form as in \eqref{Eqt:PiGroupoidcase} and give a detailed proof.

Equation \eqref{Eqt:PiGroupoidcase} gives the subtle relation between $\Pi$ and its components $c$ and $\pi$, and is the generalized form of Equation  \eqref{Eqt:PigLieGroupCase}. We will refer to $(c,\pi)$ as the \textbf{characteristic pair} of $\Pi$ (see Definition \ref{Defn:groupoidcharpair}).

It is natural to expect  that
the reduced space $\Reduced^k_{\mathrm{mult}}$ of multiplicative multivectors on the Lie groupoid $\Gpd$ has the analogous characterization of   the reduced space $\Reduced^k_{\mathrm{diff}}$  as in the pullback diagram \eqref{Digram:pullbackreduceddiffk}. The approach is similar. By means of   characteristic pairs, we   introduce  a map $K:~\mathfrak{X}_{\mathrm{mult}}^k(\Gpd ) \to \coHg ^1(\jet \Gpd,\wedge^k A)$ that sends   $\Pi$ with the characteristic pair $(c ,\pi)$  to $[c ]$.
Here $\coHg ^1(\jet \Gpd,\wedge^k A)$ is the degree $1$ groupoid cohomology of the jet groupoid $\jet \Gpd$ with respect to its adjoint action on $\wedge^k A$.
We will prove that   $K$ induces an embedding ${\Reduced}_{\mathrm{mult}}^k   \stackrel{K}{\hookrightarrow}\coHg ^1(\jet \Gpd,\wedge^k A)$. In the meantime, we   have an embedding  $\ReducedProper^k\stackrel{U}{\hookrightarrow} \coHg ^1(\heat,\wedge^k A)$  where $\heat$ is the bundle of {isotropy} jet Lie groups of   $\jet \Gpd$.
Our main result to characterize the reduced space  $\Reduced^k_{\mathrm{mult}}$  (for  $1\leqslant  k \leqslant  \Top $)  is analogous to the result of $\Reduced^k_{\mathrm{diff}}$. Indeed,  the space $\Reduced^k_{\mathrm{mult}}$ is determined by $\ReducedProper^k$ and $\coHg ^1(\jet \Gpd,\wedge^k A)$  via a pullback diagram
\begin{equation}\label{Digam:pullback2}
\xymatrix{  \Reduced^k_{\mathrm{mult}} \ar@{^{(}->}[d]_{K} \ar[r]^{I^!} & \ReducedProper^k \ar@{^{(}->}[d]^{U} \\
	  \coHg ^1(\jet \Gpd,\wedge^k A) \ar[r]^{I^*} & \coHg ^1(\heat,\wedge^k A)  }
\end{equation}
(see Theorem \ref{Thm:Rmultaspullback}).

We have Proposition \ref{Prop:3Ddiagram}, which presents the relationship between ${\Reduced}_{\mathrm{mult}}^\bullet$ and ${\Reduced}_{\mathrm{diff}}^\bullet$ as depicted by diagrams \eqref{Digram:pullbackreduceddiffk} and \eqref{Digam:pullback2}. There is a map $\bar{\delta}: $ ${\Reduced}_{\mathrm{mult}}^\bullet $ $\to $ ${\Reduced}_{\mathrm{diff}}^\bullet $, which is compatible with both pullbacks. In other words, ${\Reduced}_{\mathrm{diff}}^\bullet $ is the infinitesimal counterpart of ${\Reduced}_{\mathrm{mult}}^\bullet $ via $\bar{\delta}$. The kernel of this map is isomorphic to the kernel of the Van Est map $\coHg ^1(\Gpd,\wedge^k \ker\rho)\to \coHg ^1(A,\wedge^k \ker\rho)$.

Finally, we  handle the exceptional cases of $k=0$ and $k=\Top+1$ in Section~\ref{Sec:twoexceptionalcases} by giving detailed structures of Lie algebroid differentials and groupoid multiplicative multivectors. However, the above two pullback diagrams \eqref{Digram:pullbackreduceddiffk} and \eqref{Digam:pullback2} do not exist in these cases. Propositions~\ref{Prop:0differentialclass} and \ref{Prop:classifyk=0} present our results about $\Reduced^0_{\mathrm{diff}}$ and $\Reduced^0_{\mathrm{mult}}$. Moreover, we present our Theorems~\ref{Thm:top+1differentialpi} and \ref{Theorem:classifyk=top+1} for $\Reduced^{\Top +1}_{\mathrm{diff}}$ and $\Reduced^{\Top +1}_{\mathrm{mult}}$, respectively.


\section{Preliminaries}
 \label{Sec:pre}

\subsection{The deformation cochain complex}  Let $(A,[~,~],\rho)$ be a Lie algebroid over $M$.
The deformation complex of    $A$  is introduced in \cite{Cdef} by Crainic and Moerdijk. In \cite[Example 4.5]{AC2}, Arias Abad and Crainic  introduced a cochain complex stemming from the exterior power  of the adjoint representation of $A$   up to homotopy. In what follows we  generalize these notions.
\begin{definition}Let $ n\geqslant  -1$ and $p\geqslant  -1$ be integers.
	 A bidegree $(n,p)$ {\bf multiderivation} on the vector bundle $A\to M$   is an operator
	\[D:  \otimes^{n+1}_{\mathbb{R}}{\Apolysection{\bullet} } \to \Apolysection{\bullet} \] subject to the following conditions:
	\begin{itemize}
		\item[(1)] It is of total degree $(p-n)$, i.e., $D$ maps
		$ {\Apolysection{p_0}\otimes_{\mathbb{R}}\cdots \otimes_{\mathbb{R}} \Apolysection{p_n}} $ to $\Apolysection{p_0+p_1+\cdots+p_n +p-n}$;
		\item[(2)] It is
		skew-symmetric  with respect to  the   degree in $\Gamma(\wedge^{\bullet} A)[1]$:
		\begin{eqnarray*}
			D(\cdots,X,Y,\cdots)&=&-(-1)^{(|X|-1)(|Y|-1)}D(\cdots,Y,X,\cdots);
		\end{eqnarray*}
		\item[(3)] It is a
		derivation with respect to the $\wedge$-product in each argument:
		\begin{eqnarray*}D(X_0,\cdots,X_n\wedge X'_{n})&=& D(X_0,\cdots,X_n)\wedge X'_n\\ &&+(-1)^{|X_n|(|X_0|+\cdots+|X_{n-1}|+p-n )}X_n\wedge D(X_0,\cdots,X'_n).
		\end{eqnarray*}
	\end{itemize} 	
\end{definition}

We denote by $\Der^{n,p}(A)$ the space of bidegree $(n,p)$ multiderivations ($n,p\geqslant  -1$).
Note that   $\Der^{-1,p}(A) $ coincides with
$  \Gamma(\wedge^{p+1} A)$.

By the skew-symmetric property (2) and derivation property (3), a
multiderivation  $D\in \Der^{n,p}(A)$ is extended from a bunch of generating operations $D^{(i)}$, the restriction of $D$ on generating elements in $\Gamma(A)$ and $\CinfM$. To be specific, they are the following data:
\begin{itemize}
	\item[(0)] A skew-symmetric   map
	\[D^{(0)}:~\underbrace{\Gamma(A)\otimes_{\mathbb{R}}\cdots \otimes_{\mathbb{R}} \Gamma(A)}_{n+1}\to \Gamma(\wedge^{p+1} A);\]
	\item[(1)] A map (called  the first symbol of $D$, or simply symbol)
	\[ D^{(1)}:~ \underbrace{\Gamma(A)\otimes_{\mathbb{R}}\cdots \otimes_{\mathbb{R}} \Gamma(A)}_{n}\otimes_{\mathbb{R}} {\CinfM}\to \Gamma(\wedge^{p} A) \]
	which is skew-symmetric  in the first $n$ arguments and a derivation in the last one, i.e.,
	\begin{eqnarray}\nonumber
	D^{(1)}(u_0,\cdots,u_{n-1},fg)&=&f D^{(1)} (u_0,\cdots,u_{n-1},g)+g D^{(1)}(u_0,\cdots,u_{n-1},f),
	\end{eqnarray}
	for any $f,g\in {\CinfM}$ and $u_i\in \Gamma(A)$;
	\item[(2)] Analogous higher symbols of $D$, namely $D^{(2)}$, $D^{(3)}$, $\cdots$; The last one is $D^{(\ell)}$ ($\ell=\min{(n+1,p+1)}$):
	\[D^{(\ell)}:~ \underbrace{\Gamma(A)\otimes_{\mathbb{R}}\cdots \otimes_{\mathbb{R}} \Gamma(A)}_{n+1-\ell }\otimes_{\mathbb{R}}
	\underbrace{{\CinfM}\otimes_{\mathbb{R}}\cdots\otimes_{\mathbb{R}}{\CinfM}}_{\ell} \to \Gamma(\wedge^{p+1-\ell} A) \]
	which is skew-symmetric  in the first $(n+1-\ell)$ arguments,
	symmetric in the $\ell$ inputs of functions, and  a derivation in each of them.
\end{itemize}

The  maps $(D^{(0)},D^{(1)},D^{(2)},\cdots,D^{(\ell)})$ are compatible. For example, we have
\begin{eqnarray*}\label{Eqt:D0D1compatible}
D^{(0)}(u_0,\cdots,u_{n-1},fu_{n})&=&fD^{(0)}(u_0,\cdots,u_{n})+D^{(1)}(u_0,\cdots,u_{n-1},f)\wedge u_{n};\end{eqnarray*} and
\begin{eqnarray*}\nonumber
D^{(1)}(u_0,\cdots,g u_{n-1},f )&=&gD^{(1)}(u_0,\cdots,u_{n-1},f)-D^{(2)}(u_0,\cdots,u_{n-2},g,f)\wedge u_{n-1},\end{eqnarray*} where $g\in \CinfM$;
Similar conditions relating $D^{(i)}$ and $D^{(i+1)}$ are omitted. The last one says that $D^{(\ell)}$ is $\CinfM$-linear in the first $(n+1-\ell)$ arguments of $\Gamma(A)$.
We will use $D=(D^{(0)},D^{(1)},\cdots,D^{(\ell)})$ to denote such a multiderivation. Note that in \cite{AC2}, $(D^{(1)},\cdots,D^{(\ell)})$ is called the tail  of $D^{(0)}$.

 \begin{example}\label{Ex:LiestructuremofA}
Arising from the Lie algebroid structure on $A$, one has the standard Schouten bracket $m=[\cdot,\cdot]$ on $\Gamma(\wedge^\bullet A)$.  We can easily check  that $m$ is    a  bidegree $(1,0)$  multiderivation on $A$.  Indeed, it is extended from the Lie bracket  $m^{(0)}= [\cdot,\cdot]:~\Gamma(A)\otimes_{\mathbb{R}} \Gamma(A)\to \Gamma(A)$, and its   symbol, i.e. the anchor map $ m^{(1)}=\rho:~\Gamma(A)\otimes_{\mathbb{R}} {\CinfM}\to {\CinfM}$. \end{example}

If $p=0$, the space $\Der^{n,0}(A)$ we defined coincides with the space $\Der^{n}(A) $ introduced by Crainic and Moerdijk in \cite{Cdef}. They also defined a Lie bracket $[~,~]$ on $\Der^{\bullet}(A)$ which they call the Gerstenhaber bracket. We extend this bracket to the space  $\Der^{\bullet,\diamond}(A)$ of bigraded multiderivations.
For $D\in \Der^{n,p}(A)$ and $E\in \Der^{{n'},{p'}}(A)$, the \textbf{bigraded Gerstenhaber   bracket} of $D$ and $E$, denoted $[D,E]\in \Der^{n+{n'},p+{p'}}(A)$,  is an operation of degree $(p+{p'}-n-{n'})$    defined by the bigraded commutator:
\begin{equation}\label{Eq:[D,E]}
[D,E]=(-1)^{n{n'}}D{\littlecirc} E-(-1)^{p{p'}}E{\littlecirc} D
\end{equation}
{as a map } $\otimes_{\mathbb{R}}^{n+{n'}+1}\Gamma(\wedge^\bullet A)\to \Gamma(\wedge^{\bullet} A)$.
Here   the composition $D{\littlecirc} E:~\otimes_{\mathbb{R}}^{n+{n'}+1}\Gamma(\wedge^\bullet A)\to \Gamma(\wedge^{\bullet} A)$ is given by:
\[(D{\littlecirc} E)(X_0,\cdots,X_{n+{n'}})=\sum_\sigma (-1)^{|\sigma|} K(\sigma) D(E(X_{\sigma(0)},\cdots,X_{\sigma({n'})}),X_{\sigma({n'}+1)},\cdots X_{\sigma(n+{n'})}).\]
(The counterpart $E{\littlecirc} D$ is defined similarly.) The summation is taken over all $({n'}+1,n)$-shuffles, and $K(\sigma)$ denotes the Koszul sign produced by the permutation $\sigma$. The rule is simply that, if $X $ and $X'$ are adjacent and  swapped, then it produces a sign $\minuspower{(\abs{X}-1)(\abs{X'}-1)}$ as coefficient. Or, to find $K(\sigma)$, one can use the formula
$$
X_0\odot\cdots\odot X_{n+{n'}}=K(\sigma)
X_{\sigma(0)}\odot \cdots \odot X_{\sigma({n'})}\odot X_{\sigma({n'}+1)}\odot \cdots \odot X_{\sigma(n+{n'})},
$$
where $\odot$ denotes the product in the symmetric algebra $S(\Gamma(\wedge^\bullet A)[1])$.

\begin{proposition} Equipped with the bigraded Gerstenhaber  bracket, the  space $$\Der^{\bullet,\diamond}(A)=\bigoplus_{n\geqslant  -1, p\geqslant  -1}\Der^{n,p}(A)$$ is a bigraded Lie algebra. Namely, we have $
	[\Der^{n,p}(A),\Der^{{n'},{p'}}(A)] \subset \Der^{n+{n'},p+{p'}}(A) $, and
	\begin{eqnarray*}
		{[D,E]}&=&-(-1)^{n{n'}+p{p'}}[E,D],\\
		{[D,[E,F]]}&=&[[D,E],F]+(-1)^{n{n'}+p{p'}}[E,[D,F]],
	\end{eqnarray*}
	for all  $D\in \Der^{n,p}(A)$, $E\in \Der^{{n'},{p'}}(A)$ and $F\in \Der^{\bullet,\diamond}(A)$.
\end{proposition}
To prove this proposition, we need  to verify that $[D,E]$ defined in Equation  ~\eqref{Eq:[D,E]} satisfies the conditions of a bidegree $(n+{n'},p+{p'})$ multiderivation. This is a lengthy but straightforward verification, and thus omitted.

The following lemma is due to \cite{Cdef}.
\begin{lemma}
	Let $A\to M$ be a Lie algebroid. For the Schouten bracket structure $m:=[\cdot,\cdot]\in \Der^{1,0}(A)$ defined in Example \ref{Ex:LiestructuremofA}, we have $[m,m]=0$. Consequently, for all $p\geqslant  -1$,  $(\Der^{\bullet,p}(A),\partial:=[m,~\cdot~])$ is a cochain complex.
\end{lemma}

We call $(\Defcomplex^{\bullet, \diamond},\partial):=(\Der^{\bullet-1,\diamond}(A),\partial:=[m,~\cdot~])$ the (bigraded) {\bf deformation   cochain complex}. The  cohomology of this cochain complex is called the {\bf deformation cohomology of $A$} and is denoted $\coHg _{\mathrm{Def}}^{\bullet,\diamond}(A)$, which is a $(-1,0)$-bigraded Lie algebra, i.e. $[\coHg _{\mathrm{Def}}^{n,p}(A),\coHg _{\mathrm{Def}}^{n',p'}(A)]\subset \coHg _{\mathrm{Def}}^{n+n'-1 ,p+p' }(A)$.

\begin{remark}\label{Remk:specialdeformationcomplexes}The  complex  $(\Defcomplex^{\bullet, -1},\partial)$ coincides with the standard Chevalley-Eilenberg complex $(\OmegaA,\dA)$ of the Lie algebroid $A$, while  $(\Defcomplex^{\bullet, 0}, \partial )$ is the deformation complex of Lie algebroids introduced by Crainic and Moerdijk in \cite{Cdef}.  
	For general $k\geqslant  1$, the complex   $(\Defcomplex^{\bullet, k} ,\partial )$ is the same as    $(C^{\bullet}(A;\wedge^{k+1}\adjUTH),d)$ introduced in \cite[Example 4.5]{AC2}. See more explanations in the subsequent Section \ref{Sec:RUTH}.
\end{remark}

Recall that a {  $k$-differential} on a Lie algebroid $A$ is
a derivation    $\delta:~\Gamma(\wedge^\bullet A)\to \Gamma(\wedge^{\bullet+k-1} A) $
which is compatible with the Schouten bracket.
For a $0$-cochain $\tau\in \Defcomplex^{0, p}=\Der^{-1,p}(A) =\Gamma(\wedge^{p+1} A)$, its coboundary $\partial \tau \in \Der^{0,p}(A)$ is of the form of an exact $(p+1)$-differential: $$\partial \tau=-[\tau,~\cdot~],\qquad\mbox{ as a map }~ \Gamma(\wedge^\bullet A)\to \Gamma(\wedge^{\bullet+p} A).$$
A $1$-cochain $D\in \Defcomplex^{1, p}=\Der^{0,p}(A)$ is a (degree $p$) derivation of the graded algebra  $\Gamma(\wedge^\bullet A)$. The coboundary $\partial D\in \Der^{1,p}(A)$ is a (degree $(p-1)$) binary operator
$$\Gamma(\wedge^\bullet A)\otimes_{\mathbb{R}} \Gamma(\wedge^\bullet A)\to \Gamma(\wedge^\bullet A)$$
given by
$$
(\partial D)(X,Y)=[D(X),Y]+(-1)^{p(\abs{X}-1)}[X,D(Y)]-D[X,Y].
$$
So the following  statements  are apparent.

\begin{theorem}\cite{AC2}\label{Thm:gaugespacekdifferential1}~  Let $0\leqslant  k\leqslant  \Top +1$ be an integer.
	\begin{enumerate}
		\item[\rm{(1)}] There is a one-to-one correspondence between
		$k$-differentials $\delta\in \differentials^k(A)$ and $1$-cocycles $D\in \Der^{0,k-1}(A)$. Exact $k$-differentials correspond to coboundaries in $\Der^{0,k-1}(A)$.
		\item[\rm{(2)}] The  reduced  space of   differentials ${\Reduced}^k_{\mathrm{diff}}$ on $A$ and the degree $(1,k-1)$ Lie algebroid
		deformation cohomology
		$ \coHg _{\mathrm{Def}}^{1,k-1}(A)$ 	are one and the same.
	\end{enumerate} 	
\end{theorem}

\begin{remark}For multiplicative $k$-vectors on a Lie groupoid, we do not have a      theorem analogous to Theorem \ref{Thm:gaugespacekdifferential1} for Lie algebroid $k$-differentials. If so, there should exist a notion of (bigraded) \emph{deformation cohomology of a Lie groupoid} . In fact, in \cite{CMS},  the deformation complex of Lie groupoids is already introduced, which should be the $k=1$ case.  When $k=1$,      ${\Reduced}_{\mathrm{mult}}^1$ coincides with the deformation cohomology in degree $1$ of   Lie groupoids \cite[Proposition 4.3]{CMS}. 	
	It is also proved that this complex is isomorphic to the standard   complex of Lie groupoids with respect to the adjoint representation up to homotopy  {and the VB-groupoid complex of the tangent groupoid (\cite{GM})}. See  \cite{AC} for more about representations up to homotopy of Lie groupoids. 	
\end{remark}

\subsection{Reduced space of Lie algebroid differentials and representation up to homotopy}\label{Sec:RUTH}
 In this part we recall some   results in \cite{AC2} related to the notion of representation up to homotopy of a Lie algebroid. A  key  example is  the adjoint complex of a Lie algebroid $A$ over $M$,  a cochain complex of vector bundles concentrated in two degrees:
\begin{equation}\label{Seq:ATM}
\adjUTH:~\quad
A \stackrel{\rho}{\longrightarrow} TM[-1].
\end{equation}

In fact,  according to \cite{AC2}, one needs to choose a connection $\nabla$ on the vector bundle $A$, and the representation up to homotopy   structure operator reads $$ {\Dnabla}=\rho+ \basicconnection+\basiccurvature,$$
which is a differential on $\Omega^{\bullet}(A,\adjUTH):=\Gamma(\wedge^\bullet A^*\otimes (A \oplus TM[-1]))$. Indeed, it is a dg $(\OmegaA,\dA)$-module where $(\OmegaA,\dA)$  is the standard Chevalley-Eilenberg complex of the Lie algebroid $A$.

 The adjoint complex \eqref{Seq:ATM} together with $\Dnabla$ is denoted by $\adnabla$. The cohomology of $(\Omega^{\bullet}(A,\adjUTH),\Dnabla)$, denoted by $\coHg^\bullet(A,\adjUTH)$, is independent of $\nabla$. It is also shown in \cite[Proposition 4.6]{AC2} that the cochain complex
 $ (\Omega^{\bullet}(A,\wedge^{k } \adjUTH),\Dnabla) $ is isomorphic to $(\Der^{\bullet-1,k-1}(A),\partial )$ (the notation is  $(C^{\bullet}(A;\wedge^{k }\adjUTH),d)$ in \cite{AC2}). Therefore, according to Theorem \ref{Thm:gaugespacekdifferential1}, we get the following fact.
 \begin{proposition}For any Lie algebroid $A$,   we have $${\Reduced}^k_{\mathrm{diff}} =\coHg _{\mathrm{Def}}^{1,k-1}(A) \cong \coHg^1(A,\wedge^k \adjUTH) .$$
 	
 \end{proposition}

   We also notice a deep insight of representation up to homotopy due to \cite[Remark 2.8, Theorem 4.13, and Example 4.17]{AC2}.
   \begin{proposition}  Let $(A, [~,~], \rho) $ be a regular Lie algebroid (i.e. $\rho:A\to TM$ is a regular bundle map). There exists a deformation retract (a.k.a. homological contraction, see \cite{EML}) of dg $(\OmegaA,\dA)$-complexes
    	 \begin{equation}\label{Eq: contraction in AC2}
   	 	\begin{tikzcd}
   	 		(\Omega^{\bullet}(A,  \adjUTH),\Dnabla) \arrow[loop left, distance=2em, start anchor={[yshift=-1ex]west}, end anchor={[yshift=1ex]west}]{}{H} \arrow[r,yshift = 0.7ex, "\Phi"] & (\Omega^\bullet(A, {\ker \rho} \oplus N[-1]) , \dCE+\omega_2) \arrow[l,yshift = -0.7ex, "\Psi"].
   	 	\end{tikzcd}
   	 \end{equation}
    Here $ {\ker \rho} \subset A$ and $N=TM/\mathrm{img}\rho $ are both $A$-modules in natural manners, $\dCE$ denotes the standard Chevalley-Eilenberg differential on $ {\ker \rho} \oplus N[-1]$, and $\omega_2\in \Omega^2(A;\mathrm{Hom}(N, {\ker \rho}))$ depends on $\nabla$.
   	
   \end{proposition}
   For a thorough proof of this proposition, see \cite{XiangSolo2021}.
   From the deformation retract \eqref{Eq: contraction in AC2}, we get a new one:
   \begin{equation*}\label{Eq: wedge of contraction in AC2}
   	\begin{tikzcd}
   		\big(\Omega^{\bullet}(A,  \wedge^k \adjUTH),\Dnabla\big) \arrow[loop left, distance=2em, start anchor={[yshift=-1ex]west}, end anchor={[yshift=1ex]west}]{}{\tilde{H}} \arrow[r,yshift = 0.7ex, "\tilde{\Phi}"] & \Big(\Omega^\bullet\big(A,\wedge^k( {\ker \rho} \oplus N[-1])\big), \dCE+\omega_2\Big) \arrow[l,yshift = -0.7ex, "\tilde{\Psi}"].
   	\end{tikzcd}
 \end{equation*}
 	For  definitions of $\tilde{H},\tilde{\Phi}$ and $\tilde{\Psi}$, see \cite{Manetti}. 
Therefore, according to Theorem \ref{Thm:gaugespacekdifferential1} $\mathrm{(2)}$, we have two corollaries.
   \begin{corollary}\label{Cor:regularReducedDiff}If $(A, [~,~], \rho) $ is a regular Lie algebroid, then we have $${\Reduced}^k_{\mathrm{diff}} =\coHg _{\mathrm{Def}}^{1,k-1}(A) \cong \coHg^1(A,\wedge^k \adjUTH)\cong \coHg^1(A,\wedge^k ( {\ker \rho} \oplus N[-1])).$$
   	
   \end{corollary}
\begin{corollary}\label{Cor:transitivecasetorecover} If $A$ is transitive, i.e. $N$ is trivial, then we have ${\Reduced}^k_{\mathrm{diff}}\cong \coHg ^1(A,\wedge^k   {\ker \rho}  )$ (the usual Chevalley-Eilenberg cohomology).
	\end{corollary}
However, the isomorphism  in this corollary comes from the   quasi-isomorphism  $\tilde\Phi$, which is complicated to describe. In Section \ref{Sec:transtivedifferentials}, we will give another proof of this fact by decomposition of Lie algebroid differentials.

\subsection{Reduced space of $\rho$-compatible tensors} 

   A $k$-differential $\delta$ on $A$, as a multiderivation, can be written as $\delta=(\deltazero,\deltaone )$  where $\deltazero$ is an operator $\Gamma(A)\to \Gamma(\wedge^k A)$ and
\[\deltaone :{\CinfM}\to \Gamma(\wedge^{k-1} A) \] is the symbol of $\delta$. They are subject to the conditions
\begin{equation}\label{Eqt:delta0delta1derivatives} \deltaone (ff')=f'\deltaone(f)  +f\deltaone  (f') ,\qquad \deltazero (fu)=\deltaone ( f)\wedge u+f\deltazero ( u) ,\end{equation}
and
\begin{eqnarray}\label{Eqt:kdifferential2.5}
0&=& \ba{\deltaone  (f ) }{f' }+\minuspower{k-1}\ba{f
}{\deltaone ( f' )},\\
\label{Eqt:kdifferential3}
\deltaone \ba{u}{f}&=&[\deltazero(u),f]+[u,\deltaone(f)],\\
\label{Eqt:kdifferential4}
\deltazero [u,v]&=&[\deltazero(u),v]+[u,\deltazero (v)],
\end{eqnarray}
for all $f,f'\in {\CinfM}$, $u,v\in \Gamma(A)$.
Note that in the ordinary case ($1\leqslant  k\leqslant  \Top $),  Equations  ~\eqref{Eqt:delta0delta1derivatives} and   \eqref{Eqt:kdifferential4} imply Equations  ~ \eqref{Eqt:kdifferential2.5} and \eqref{Eqt:kdifferential3}.

The above operator $\deltaone$ corresponds to a tensor  $\pi$   $\in \Gamma(TM\otimes (\wedge^{k-1} A))$ such that $\deltaone(f)=\iota_{df}\pi$ for all $f\in \CinfM$. The condition \eqref{Eqt:kdifferential2.5} transforms to
 \begin{equation}\label{Eqt:fine-condition}
 \inserts_{\rho^*\xi}\iota_{\eta}\pi
 = -\inserts_{\rho^*\eta}\iota_{\xi}\pi,
 \qquad\forall~\xi,\eta\in\Omega^1(M).
 \end{equation}

\begin{definition}\label{Defn:properktensor}
	A \textbf{$\rho$-compatible $k$-tensor}  is a section $\pi$   $\in \Gamma(TM\otimes (\wedge^{k-1} A))$ that satisfies Equation   \eqref{Eqt:fine-condition}.
 	\end{definition}

When $k=1$, Equation  ~\eqref{Eqt:fine-condition} becomes  trivial and every $\pi\in \XX(M)$ is $\rho$-compatible. When $k=2$, Equation  ~\eqref{Eqt:fine-condition} simply  means
that the map $T^*M\to TM$ defined by $  \pi^\sharp{\littlecirc} \rho^*$ is skew-symmetric  where $\pi^\sharp: ~A^*\to TM$ is the contraction to the second factor of $\pi\in \Gamma(TM\otimes A )$.

Throughout this paper, we will constantly consider the direct sum $TM\oplus A$, which is identically the tangent space of the vector bundle $A$ along its base $M$. It is clear that
$ \wedge^n (TM\oplus A)  $ is the direct sum of  vector bundles $(\wedge^p TM)\otimes (\wedge^{n-p} A)$, for $p=0,1,\cdots,n$.

We need a particular operator $D_\rho$ which was   first studied in
\cite{CSX}. It is  a  degree $0$ derivation on $ \wedge^\bullet (TM\oplus A)$
determined by its $TM\oplus A\to TM\oplus A$ part:
\begin{eqnarray*}\label{D}
	D_\rho(X,u)=(\rho(u),0),\qquad \forall~ X\in TM, u\in A.
\end{eqnarray*}
Hence $D_\rho $ maps $(\wedge^pTM)\otimes(\wedge^q A)$ to $(\wedge^{p+1}TM)\otimes(\wedge^{q-1} A)$.
For
$\pi$   $\in   TM\otimes (\wedge^{k-1} A) $,
we introduce
\begin{equation}\label{Bpi}
B\pi= \frac{1-e^{-\rhopush}}{\rhopush} (\pi)
=
\pi-\tfrac{1}{2!}\rhopush\pi+\tfrac{1}{3!}\rhopush^2
\pi+\cdots+\tfrac{\minuspower{k-1}}{k!}\rhopush^{k-1}\pi\in \wedge^k (TM\oplus A) .
\end{equation}
Note that the term $\rhopush^{j}\pi\in  (\wedge^{j+1} TM)\otimes (\wedge^{k-1-j} A) $.

Let $\Gpd \toto M$ be a Lie groupoid and $A$   the associated tangent Lie algebroid. Along the base manifold $M$,  every $s$-fibre of $\Gpd$ is transversal to the identity section $M$, and thus  we have a canonical decomposition \begin{equation*}\label{Eqt:TG=TM+A} T \Gpd |_M = TM\oplus A. \end{equation*}
We adopt this identification  throughout the paper.
From this point of view, the tangent map of groupoid inverse
$\invstar : ~T_xM\oplus A_x \to T_xM\oplus A_x$  ($x\in M$)
is given by
\begin{equation}
\label{Eqt:firstinversestar}
(X,u)\mapsto (X+\rho(u),-u).
\end{equation} In the sequel,  for a general Lie algebroid $A$ with anchor map $\rho$, we use the same notation
$\invstar$ to denote the automorphism $  TM\oplus A\to TM\oplus A$ defined as above, though $A$ may not be integrable.

\begin{lemma}
	Extend the inverse map $\invstar$ defined in Equation   \eqref{Eqt:firstinversestar} to an automorphism of the exterior algebra $\wedge^\bullet (TM\oplus A)$. Then for all $W\in   (\wedge^pTM)\otimes(\wedge^q A) $, we have
	\begin{equation}\label{Eqt:inverstarinDrho}
	\invstar  W
	=\minuspower{q}(
	W- \rhopush W+\tfrac{1}{2 }\rhopush^2
	W+\cdots+\tfrac{\minuspower{q}}{q!}\rhopush^{q} W).
	\end{equation}
\end{lemma}This equality can be proved using the definition of $\invstar $. We omit the details.


\begin{example}Let
	$\tau\in\Gamma(\wedge^kA)$ be given. Then $\rhopush\tau$ is a $\rho$-compatible $k$-tensor. In fact, we have
	$$
	\inserts_{\rho^*\xi}\iota_{\xi}(\rhopush\tau)
	=\inserts_{\rho^*\xi}\inserts_{\rho^*\xi}\tau=0,\qquad \forall~   \xi \in\Omega^1(M).
	$$
\end{example}

\begin{definition}\label{def:reducedproperspace}
	We call a $\rho$-compatible $k$-tensor \textbf{exact} if it is of the form $\rhopush\tau$, for some $\tau\in\Gamma(\wedge^kA)$.  The space of $\rho$-compatible $k$-tensors modulo the exact ones, denoted by $\ReducedProper^k$,  is called the \textbf{reduced space of $\rho$-compatible  $k$-tensors}.
\end{definition}

If $k=1$, then we have $\ReducedProper^1=\mathfrak{X}(M)/\Gamma(\Img \rho)$.  For regular Lie algebroids,   $\ReducedProper^k$ can be identified with the section space of $(TM/\Img\rho)\otimes (\wedge^{k-1}A)$ (see Proposition \ref{Prop:ReducedProperkregular}).
The following lemma lists more  properties of   $\rho$-compatible tensors.

\begin{lemma}\label{Lemma:ixitoBpi}A $\rho$-compatible $k$-tensor
	$\pi$   $\in \Gamma(TM\otimes (\wedge^{k-1} A))$ satisfies the following equality
	\begin{equation}\label{Eqn:ixi}
	\iota_{\rho^* \xi} D_{\rho}^{j-1} \pi=D_{\rho}^j (\iota_\xi \pi)=\frac{1}{j+1}\iota_\xi D_{\rho}^j \pi,\qquad \forall~ \xi\in \Omega^1(M),~j=1,\cdots, {k-1}.
	\end{equation}
	Moreover, the tensor field $B\pi$ defined in Equation  ~\eqref{Bpi} satisfies the following relations:
	\begin{eqnarray}
	\label{Eqn:specialpropertypi2}
	\iota_{\xi}(B\pi) &=& (-1)^{k-1} \invstar  (\iota_\xi \pi),
	\\\label{Eqn:specialpropertypi1}
	\iota_{\xi}(B\pi)+\iota_{\rho^*\xi}(B\pi)&=&\iota_{\xi}\pi.
	\end{eqnarray}
\end{lemma}
\begin{proof}Relation \eqref{Eqn:ixi}  is proved in \cite[Lemma 2.5]{CSX}. By  Equation  ~\eqref{Eqn:ixi}  and the definition of $B\pi$, one easily gets Equation  ~\eqref{Eqn:specialpropertypi2}:
	\begin{eqnarray*}
		\iota_{\xi}(B\pi)&=&\iota_{\xi}\sum_{j=0}^{k-1}\frac{\minuspower{j}}{{(j+1)}!}\rhopush^{j}\pi
		=\sum_{j=0}^{k-1}\frac{\minuspower{j}}{{(j+1)}!} (\iota_{\xi}\rhopush^{j}\pi)
		\\&=& \sum_{j=0}^{k-1}\frac{\minuspower{j}}{{j }!} \rhopush^{j}(\iota_{\xi}\pi)=(-1)^{k-1} \invstar  (\iota_\xi \pi).
	\end{eqnarray*}
	The last step is due to Relation \eqref{Eqt:inverstarinDrho}.
	We also have
	\begin{eqnarray*}
		\iota_{\rho^*\xi}(B\pi)&=&\iota_{\rho^*\xi}\sum_{j=0}^{k-1}\frac{\minuspower{j}}{{(j+1)}!}\rhopush^{j}\pi
		=\sum_{j=0}^{k-2}\frac{\minuspower{j}}{{(j+1)}!} (\iota_{\rho^*\xi}\rhopush^{j}\pi)
		\\&=& \sum_{j=0}^{k-2}\frac{\minuspower{j}}{{(j+1) }!} \rhopush^{j+1}(\iota_{\xi}\pi)= \iota_{\xi}\pi+(-1)^{k} \invstar  (\iota_\xi \pi).
	\end{eqnarray*}
	One gets Equation  ~\eqref{Eqn:specialpropertypi1} immediately.  \end{proof}

\subsection{Expression of multiplicative  multivectors along $M$}\label{ULT} We need some basic facts about   multiplicative multivectors on  Lie groupoids. Let  $\Gpd $ be  a Lie groupoid over   $M$.
\begin{definition}\label{Defn:affine3conditions}\rm{(\cite{ILX})}
	A $k$-vector $\Pi$ on $\Gpd $ is said to be {\bf affine}
	if any of the following equivalent conditions holds:
	\begin{itemize}
		\item[1)] For all $u\in\Gamma(A)$, $\ba{\Pi}{\overrightarrow{u}}$
		is right-invariant;
		\item[2)] For all $u\in\Gamma(A)$,  $\ba{\Pi}{\overleftarrow{u}}$
		is left-invariant;
		\item[3)] For any $(g,{r})$ composable, $x=s(g)=t({r})$, one has
		\begin{equation}\label{aff}
		\Pi_{g{r}}=L_{[b_g ]}\Pi_r+R_{[b'_{r}]}\Pi_g-L_{[b_g ]}R_{[b'_{r}]}\Pi_{x},
		\end{equation}
		where $b_g$ and $b'_{r}$ are    bisections through, respectively, $g$ and ${r}$.
	\end{itemize}
\end{definition}

The following lemma gives a  useful characterization of  multiplicative $k$-vectors.

\begin{lemma}\rm{(\cite{ILX}*{Theorem~2.19})}
	\label{multi vector}
	A $k$-vector $\Pi$ is multiplicative if and only if the following three conditions hold:
	\begin{enumerate}
		\item \label{cun} The $k$-vector $\Pi$ is affine;
		\item \label{cdeux} With respect to $\Pi$, $M$ is a coisotropic submanifold of $\Gpd $;
		\item \label{ctrois} For any $\xi\in\Omega^1(M)$, $\inserts_{s^*(\xi)}\Pi$ is left-invariant (or $\inserts_{t^*(\xi)}\Pi$ is right-invariant).
	\end{enumerate}
\end{lemma}

We need the relation between multiplicative multivectors on   $\Gpd$ and differentials  on $A$.
A multiplicative $k$-vector $\Pi$ on the Lie groupoid $\Gpd $
corresponds to a $k$-differential $\delta_{\Pi}:~\Gamma(\wedge^\bullet A)\to \Gamma(\wedge^{\bullet+k-1} A)$,  called  the \textbf{infinitesimal} of $\Pi$.
For any $u\in \Gamma(\wedge^\bullet A)$,   $\delta_{\Pi}(u)$ is defined by the relation
$$\overrightarrow{\delta_{\Pi}(u)}=[\Pi,\overrightarrow{u}], \qquad u \in \Gamma(\wedge^\bullet A).
$$
So we obtain a map
\begin{equation}\label{Eqt:deltamap}\delta_{(\cdot)}\colon \mathfrak{X}_{\mathrm{mult}}^\bullet(\Gpd )\to \differentials^\bullet(A),\qquad \Pi\mapsto \delta_{\Pi}\end{equation}
which is morphism of graded Lie algebras.
Moreover, if $\Gpd $ is an $s$-connected and simply connected Lie groupoid, then  $\delta_{(\cdot)}$
is an isomorphism of graded Lie algebras (known as the {\bf universal lifting theorem}), see \cite{ILX}.

For a multiplicative $k$-vector $\Pi$ on $\Gpd $, since $M$ is coisotropic with respect to $\Pi$, $\Pi|_M$ has no component in $\wedge^k A$. So we have \[\Pi|_M\in \Gamma(TM\otimes (\wedge^{k-1} A))\oplus \Gamma((\wedge^2TM) \otimes (\wedge^{k-2} A))\oplus \cdots \oplus \Gamma(\wedge^k TM).\]

As shown in \cite{CSX}, $\Pi|_M$ is actually determined by its $\Gamma(TM\otimes (\wedge^{k-1} A))$-component.

\begin{proposition}\label{CSXproppi}\rm{(\cite{CSX})}
	Given  $\Pi\in\mathfrak{X}_{\mathrm{mult}}^k(\Gpd ) $,
	denote by $\pi$ the $\Gamma(TM\otimes (\wedge^{k-1} A))$-component of $\Pi$ restricting on $M$. Then $\pi$ is a $\rho$-compatible $k$-tensor, and we have
	$
	\Pi|_M = B\pi$ (defined in Equation  ~\eqref{Bpi}).
	Moreover, $\pi$ satisfies the following property:
	\begin{gather}\label{Eqt:Pitstarfbypi}
 {\delta_{\Pi}(f)}= 	\ba{\Pi}{t^*f}|_M=\minuspower{k-1}\inserts_{df}\pi,
	\quad\forall~ f\in\CIM .
	\end{gather}
	
\end{proposition}

For exact multiplicative multivectors, the $\rho$-compatible tensor $\pi$ is obtained by the following lemma.

\begin{lemma}\label{lexactmultiplicativepi}
	Given $\tau \in \Gamma(\wedge^k A)$,   let $\Pi=\overrightarrow{\tau }-\overleftarrow{\tau }$ be
	the associated exact multiplicative
	$k$-vector. Then 	
	the $\Gamma(TM\otimes (\wedge^{k-1} A))$-component of $\Pi$ restricting on $M$    is
	$ D_\rho \tau $.
\end{lemma}
\begin{proof}By Proposition \ref{CSXproppi}, we have \[B\pi=(\overrightarrow{\tau }-\overleftarrow{\tau })|_M=\tau -(-1)^k \invstar \tau.\]
	We then use the formula of $\invstar$ in Equation  ~\eqref{Eqt:inverstarinDrho}. It is clear that the $\Gamma(TM\otimes (\wedge^{k-1} A))$-component of $(\overrightarrow{\tau }-\overleftarrow{\tau })|_M$ is exactly $D_\rho \tau $.
\end{proof}

\section{The transitive case}\label{Sec:transtivedifferentials}

In this section, we   will show that differentials on transitive Lie algebroids and multiplicative multivectors on transitive Lie groupoids both
admit nice   decompositions. The idea  initiates from  an earlier work \cite{CL}.

A Lie algebroid $A$ is called \textbf{transitive} if its anchor map $\rho: A\to TM$ is surjective. In this case, there exists a bundle map $\lambda: TM\to A$ such that $\rho{\littlecirc} \lambda=\mathrm{id}$. The map $\lambda$ is called a {\bf connection} on $A$ (see \cite{Mackenzie}).
Via the connection $\lambda$, we are able to introduce a morphism of vector bundles $\Upsilon:~\wedge^\bullet(TM\oplus A)\to \wedge^\bullet A$ such that
\begin{eqnarray}\label{Ups}
\nonumber\Upsilon: && (\wedge^p TM)\otimes (\wedge^q A) ~\to~ \wedge^{p+q} A ,\\
&&X_1\wedge \cdots\wedge X_p\otimes Z\mapsto\lambda(X_1)\wedge\cdots \wedge\lambda(X_p)\wedge Z,
\end{eqnarray}
for $X_i\in TM $ and $Z\in   \wedge^q A $.  One can check the following relation:
\begin{equation}\label{Eqt:UpsilonEasyProperty}
\iota_{\rho^*\xi}\Upsilon(W)=\Upsilon(\iota_{\xi}W+\iota_{\rho^*\xi}W),\quad\forall~ \xi\in T^*M, W\in \wedge^\bullet(TM\oplus A).
\end{equation}


As the Lie algebroid $A$ is transitive,  the subbundle $\ker\rho\subset A$  is regular and a bundle of Lie algebras. Moreover, $A$ has adjoint actions on $\ker\rho$ and $\wedge^k \ker\rho$.

\subsection{Primary pairs of transitive Lie algebroids}
We introduce the notation  $$\mathcal{P}^k:=Z^1(A,\wedge^k \ker  \rho)\times \Gamma(\wedge^k A).$$
Two elements $(\Omega,\Lambda)$ and $(\Omega',\Lambda')$ in $\mathcal{P}^k$
are said to be {\bf equivalent}, written as $(\Omega,\Lambda)\sim(\Omega',\Lambda')$, if there exists some $\nu\in \Gamma(\wedge^k \ker  \rho)$ such that
\[\Omega'=\Omega-[\nu,~\cdot~],\qquad \Lambda'=\Lambda+\nu.\]
It is apparent that $\sim$ is an equivalence relation on $\mathcal{P}^k$.
\begin{definition}
	The equivalence class of   $(\Omega,\Lambda)\in \mathcal{P}^k$,  denoted by $[(\Omega,\Lambda)]$,
	is called a \textbf{$k$-primary pair} of the transitive Lie algebroid $A$.
	
\end{definition}

We will establish a one-to-one correspondence between $k$-differentials   on the transitive Lie algebroid $A$ and primary pairs in $\mathcal{P}^k/_\sim$.
Below is the detail of this construction. Given $(\Omega,\Lambda)\in \mathcal{P}^k$, define
\[\deltazero:\Gamma(A)\to \Gamma(\wedge^k A),\qquad u\mapsto  [\Lambda, u]+\Omega(u)\]
and
\[\deltaone: {\CinfM}\to \Gamma(\wedge^{k-1} A), \qquad f\mapsto [\Lambda,f].\]
Clearly,  $\deltazero$ together with $\deltaone$ extends to a $k$-differential on $A$, which is denoted by\[\delta= [\Lambda,~\cdot~]+\Omega\in \differentials^k(A).\]

It is easily seen that, if $(\Omega,\Lambda)\sim(\Omega',\Lambda')\in \mathcal{P}^k$, then they correspond  to the same differential
$$
\delta=[\Lambda,~\cdot~]+\Omega=[\Lambda',~\cdot~]+\Omega'.
$$
In summary, we obtain a   map from primary pairs to differentials on $A$:
\begin{align*}\Smap:~& ~~~\mathcal{P}^k/_{\sim} ~ ~~ \to \differentials^k(A),\\
&[(\Omega,\Lambda)]~~\mapsto ~~[\Lambda,~\cdot~]+\Omega.
\end{align*}

Our main result is the following theorem.

\begin{theorem}\label{Thm:s1-1}
	For a transitive Lie algebroid $A$, the map $\Smap$   defined above is a one-to-one  correspondence.
\end{theorem}

We need a lemma.
\begin{lemma}\label{lUpsilonBpi}Let $\lambda$ be a connection of $A$ and $\Upsilon$ be as defined in Equation  ~\eqref{Ups}. Let $\pi\in \Gamma(TM\otimes (\wedge^{k-1} A))$ be a $\rho$-compatible $k$-tensor.  Define \begin{eqnarray}\label{Ll}
	\Lambda:=\Upsilon(B\pi)\in \Gamma(\wedge^k A),
	\end{eqnarray}
	where $B\pi\in \Gamma(\wedge^k (TM\oplus A))$ is given by Equation  ~\eqref{Bpi}. Then $\Lambda$ is subject to the following property:
	\begin{equation}\label{Eqt:Lambdaf=delta0f} [\Lambda, f]=(-1)^{k-1} \iota_{df} \pi, \qquad \forall~ f\in {\CinfM}.
	\end{equation}
	\end{lemma}
\begin{proof}
	In fact, we have
	\begin{eqnarray*}
		[\Lambda, f]&=&(-1)^{k-1} \iota_{\rho^* df} \Upsilon(B\pi)\\ &=&(-1)^{k-1} \Upsilon(\iota_{df} B\pi+\iota_{\rho^*df} B\pi),
		\qquad\mbox{  by Equation  ~\eqref{Eqt:UpsilonEasyProperty};}\\
		&=&(-1)^{k-1} \Upsilon(\iota_{df} \pi),\qquad \mbox{by Equation  ~\eqref{Eqn:specialpropertypi1} ;}\\ &=&(-1)^{k-1} \iota_{df} \pi.
	\end{eqnarray*}
\end{proof}
One can equivalently rewrite Equation  \eqref{Eqt:Lambdaf=delta0f} as
\begin{equation}\label{Eqt:rhoLambdaispi}
\rhopush \Lambda=\pi.
\end{equation}
In other words, every $\rho$-compatible $k$-tensor is exact  and hence $\ReducedProper^k=0$ holds for transitive Lie algebroids.

We have another   consequence of this lemma.
\begin{proposition}\label{Prop:ReducedProperkregular} Let $A$ be a regular Lie algebroid over $M$.
	Let $N:=TM/\Img \rho$ be the normal bundle of $\Img\rho$.
	Then we have an isomorphism of vector spaces $\ReducedProper^k\cong \Gamma(N\otimes (\wedge^{k-1}A))$.
\end{proposition}
\begin{proof}As $\Img \rho$ is regular, one can find a vector bundle decomposition $TM\cong \Img \rho\oplus N$. Let $\pi\in \Gamma(TM\otimes (\wedge^{k-1}A))$ be a $\rho$-compatible $k$-tensor and suppose that it is decomposed as
	$$
	\pi=\pi_1+\pi_2,
	$$
	where $\pi_1\in \Gamma(\Img \rho\otimes (\wedge^{k-1}A))$, $\pi_2\in \Gamma(N \otimes (\wedge^{k-1}A))$. By construction, $\pi_2$ is  $\rho$-compatible  and    so is    $\pi_1$.  In spirit   of 	
	Lemma \ref{lUpsilonBpi}, one sees that $\pi_1$ is indeed exact. Hence $\pi\equiv \pi_2$ in $\ReducedProper^k$ and this proves the claim. We remark that the isomorphism $\ReducedProper^k\cong \Gamma(N\otimes (\wedge^{k-1}A))$ is not canonical as it depends on the choice of the decomposition $TM\cong \Img \rho\oplus N$.
\end{proof}
We are now ready to finish the proof of our main theorem in this section.
\begin{proof}[Proof of Theorem \ref{Thm:s1-1}]We first show that $\Smap$ is injective. It suffices to show that, if $$\Smap[(\Omega,\Lambda)]=\Smap[(\Omega',\Lambda')]=\delta,$$
	then $(\Omega,\Lambda)$ and $(\Omega',\Lambda')$ are equivalent. In fact, by
	$$
	\deltaone(f)=[\Lambda,f]=[\Lambda',f],\quad\forall~ f\in {\CinfM},
	$$
	we see that
	$\iota_{\rho^*df} (\Lambda'-\Lambda)=0$ and thus $\nu:=\Lambda'-\Lambda\in \Gamma(\wedge^k \ker  \rho)$. Moreover, by
	$$
	\deltazero(u)=[\Lambda,u]+\Omega(u)=[\Lambda',u]+\Omega'(u), \quad\forall~ u\in \Gamma(A),
	$$
	we get $\Omega'(u)=\Omega(u)-[\nu,u]$. This proves that $(\Omega,\Lambda)\sim (\Omega',\Lambda')$.
	
	We then show that $\Smap$ is surjective. Given $\delta=(\deltazero,\deltaone)\in\differentials^k(A)$, we find the corresponding   $\pi\in \Gamma(TM\otimes (\wedge^{k-1} A))$, a $\rho$-compatible $k$-tensor determined by the relation
	$$\deltaone f  =  (-1)^{k-1}\iota_{ df}\pi,\quad f\in  {\CinfM}.
	$$
Take a connection of $A$ and define $\Lambda =\Upsilon(B\pi)\in \Gamma(\wedge^k A)$ by Equation  ~\eqref{Ll}. Then by Equation  ~\eqref{Eqt:Lambdaf=delta0f}, we get
	\begin{equation}\label{Eqt:temp3}
	[\Lambda, f]=(-1)^{k-1} \iota_{df} \pi=\deltaone(f), \qquad \forall~ f\in {\CinfM}.
	\end{equation}
Now we define a map $\Omega: \Gamma(A)\mapsto \Gamma(\wedge^k A)$ by \[\Omega(u)=\deltazero(u)-[\Lambda,u],\qquad\forall~ u\in\Gamma(A).\] From the property \eqref{Eqt:temp3}, we easily verify that $\Omega$ is ${\CinfM}$-linear. So, $\Omega$ is a morphism of vector bundles. Moreover, we claim that $\Omega$ takes values in the subbundle $\wedge^k \ker  \rho$, which  amounts to show that $$\iota_{\rho^*df} (\Omega u)=0,\qquad\forall~ f\in {\CinfM}.$$ This is examined as follows:
	\begin{eqnarray*}
		(-1)^{k-1} \iota_{\rho^*df} \Omega u&=&[\deltazero(u),f]-[[\Lambda,u],f]\\ &=&\deltaone[u,f]-[u,\deltaone(f)]-[[\Lambda,u],f],\quad \mbox{by Equation  ~\eqref{Eqt:kdifferential3};}\\ &=&[\Lambda,[u,f]]-[u,[\Lambda,f]]-[[\Lambda,u],f],\qquad\mbox{by Equation  ~\eqref{Eqt:temp3};}\\ &=&0,\qquad\mbox{by Jocabi identity.}
	\end{eqnarray*}
	So we indeed  get a bundle map $\Omega:A\to \wedge^k\ker\rho$. Due to Equation  ~\eqref{Eqt:kdifferential4}, one easily verifies that $\Omega$ is a Lie algebroid $1$-cocycle. 	
	Clearly,  the primary pair $[(\Omega,\Lambda)]$ that we constructed is mapped to $\delta=(\deltazero,\deltaone)$ via $\Smap$.
\end{proof}
\subsection{Reduced space of differentials on transitive Lie algebroids}
The following lemma  describes  equivalence relations of differentials on $A$ in terms of  primary pairs.
\begin{lemma}\label{Thm:classbyH1transitivealgebroid}
	Let $\delta $ and $\delta'$ be in $\differentials^k(A)$. Let
	$[(\Omega,\Lambda)]$ and $[(\Omega',\Lambda')]$  be, respectively, the corresponding primary pairs in $\mathcal{P}^k/_{\sim}$. Then $\delta\gaugeequiv \delta'$ holds
	if and only if  $[\Omega]=[\Omega']$ holds in $\coHg ^1(A,\wedge^k\ker\rho)$.
\end{lemma}
\begin{proof}We show the implication  ``$[\Omega]=[\Omega']$''
	$\Rightarrow$   ``$\delta\gaugeequiv \delta'$''. Suppose that $\Omega'=\Omega+d_{A}\nu$, for some $\nu\in\Gamma(\wedge^k\ker \rho)$. Then we have
	$$
	\delta'=[\Lambda',\cdot]+\Omega'=[\Lambda',\cdot]+\Omega-[\nu,\cdot]
	=[\Lambda'-\nu,\cdot]+\Omega=\delta+[\tau,\cdot],
	$$
	where $\tau=\Lambda'-\Lambda-\nu\in\Gamma(\wedge^k A)$. 	
	The converse is also clear.
\end{proof}


Using the above facts, we can recover Corollary \ref{Cor:transitivecasetorecover} directly:
 \begin{proof}[Proof of Corollary   \ref{Cor:transitivecasetorecover}] Via Theorem \ref{Thm:s1-1}, we are able to built  a map $C\colon  \differentials^k(A) \to \coHg ^1(A,\wedge^k\ker\rho)$: for each $\delta\in \differentials^k(A)$ which corresponds to the primary pair $[(\Omega,\Lambda)]$, define $C(\delta)=[\Omega]$. By   Lemma \ref{Thm:classbyH1transitivealgebroid}, there is a   linear map $C\colon {\Reduced}^k_{\mathrm{diff}} \to   \coHg ^1(A,\wedge^k\ker\rho)$. It is a routine  work to verify that this  map is indeed an isomorphism.
 \end{proof}

\subsection{Multiplicative multivectors  on transitive Lie groupoids}\label{Sec:subsecTransitivegroupoids}

Now we consider a  \textbf{transitive} Lie groupoid $\Gpd$ over $M$. In other words, any two points in the base manifold are connected by some element in $\Gpd$.
The tangent Lie algebroid $A$ of $\Gpd$ is also transitive\footnote{This       follows from the fact that for   a transitive Lie groupoid $\Gpd$, the map $(s,t)$$:~ \Gpd\rightarrow M\times M$ is a surjective submersion. See Mackenzie's book \cite{Mackenzie}.}.
The groupoid $\Gpd$ acts naturally on the subbundle $\ker\rho\subset A$ via adjoint action:
$$\Ad_g v :=(R_{g^{-1}*}{\littlecirc} L_{g*}) v,\qquad \forall~ v\in \ker\rho|_{s(g)}.
$$

In what follows, $Z^1(\Gpd,\wedge^k \ker  \rho)$ denotes $1$-cocycles $\mathcal{F}:~\Gpd\to \wedge^k \ker  \rho$ with respect to the induced adjoint action of $\Gpd$ on $\wedge^k \ker\rho$. Note that $\mathcal{F}$ gives rise to a $k$-vector $\Pi_{\mathcal{F}}\in \mathfrak{X}^k(\Gpd )$  defined by $$\Pi_{\mathcal{F}}(g):=R_{g*} {\mathcal{F}(g)},\qquad\forall~ g\in \Gpd.$$
The fact that $\mathcal{F}$ is a $1$-cocycle can be rephrased in terms of $\Pi_{\mathcal{F}}$:
\begin{eqnarray*}\label{pif}
	\Pi_{\mathcal{F}}(gr)=L_{g*} \Pi_{\mathcal{F}}(r)+R_{r*}\Pi_{\mathcal{F}}(g),\qquad\forall~ (g,r)\in \Gpd ^{(2)}.
\end{eqnarray*}
By Lemma \ref{multi vector}, one easily verifies that $\Pi_{\mathcal{F}} \in \mathfrak{X}_{\mathrm{mult}}^k(\Gpd )$.

In analogous to the space $\mathcal{P}^k$, equivalence relation in $\mathcal{P}^k$ and the quotient space $\mathcal{P}^k/_\sim$ of primary pairs associated to transitive Lie algebroids, we  introduce the following notations.
\begin{itemize} \item A space $\mathcal{Q}^k$ defined by \[\mathcal{Q}^k:= Z^1(\Gpd ,\wedge^k \ker  \rho)\times \Gamma(\wedge^k A).\]
	\item An equivalence relation in $\mathcal{Q}^k$: two pairs $(\mathcal{F}, \Lambda)$ and $(\mathcal{F}', \Lambda')$ in $\mathcal{Q}^k$ are said to be {\bf equivalent}, denoted by $(\mathcal{F}, \Lambda)\sim (\mathcal{F}', \Lambda')$, if there exists some   $\nu\in \Gamma(\wedge^k \ker  \rho)$ such that
	\begin{eqnarray}
	\label{eqre3}\mathcal{F'}&=&\mathcal{F}+d_\Gpd\nu,\\ 
	\nonumber \Lambda'&=&\Lambda+\nu.
	\end{eqnarray}
\end{itemize}
It is direct to see that $\sim$ defines an equivalence relation in  $\mathcal{Q}^k$.
\begin{definition}
	An equivalence class  $[(\mathcal{F}, \Lambda)]\in \mathcal{Q}^k/_\sim$
	is called a {\bf $k$-primary pair} of the transitive Lie groupoid $\Gpd$.\end{definition}

Note that Equation  ~\eqref{eqre3} is saying that
\begin{equation}\label{eqre2}
\mathcal{F}'(g) = \mathcal{F}(g)+\Ad_g \nu _{s(g)}-\nu _{ t(g)},\quad\forall~ g\in \Gpd.
\end{equation}

Moreover, a Lie groupoid $1$-cocycle $\mathcal{F}\in Z^1(\Gpd ,\wedge^k\ker  \rho)$ differentiates to a Lie algebroid $1$-cocycle $\hat{\mathcal{F}}\in Z^1(A,\wedge^k\ker  \rho)$. Clearly, $(\mathcal{F}, \Lambda)\sim (\mathcal{F}', \Lambda')\in \mathcal{Q}^k$ implies that
$(\hat{\mathcal{F}},\Lambda)\sim (\hat{\mathcal{F}}',\Lambda')\in \mathcal{P}^k$.

Our goal is to establish a one-to-one correspondence between multiplicative $k$-vectors $\Pi\in \mathfrak{X}_{\mathrm{mult}}^k(\Gpd )$ and   $k$-primary pairs $[(\mathcal{F}, \Lambda)]\in\mathcal{Q}^k/_\sim$.
To show this fact, we  define a map
\begin{eqnarray*}\SSmap:~\mathcal{Q}^k&\to& \mathfrak{X}_{\mathrm{mult}} ^k(\Gpd ),\\
	(\mathcal{F}, \Lambda)&\mapsto& \overrightarrow{\Lambda}-\overleftarrow{\Lambda}+\Pi_{\mathcal{F}}.
\end{eqnarray*}
Here $\overrightarrow{\Lambda}-\overleftarrow{\Lambda}$ and $\Pi_{\mathcal{F}}$ are both multiplicative, and hence   $\SSmap$ takes values in $\mathfrak{X}_{\mathrm{mult}} ^k(\Gpd )$.
It is easy to show that,  if $(\mathcal{F}, \Lambda)\sim (\mathcal{F}', \Lambda')$, then we have $\SSmap(\mathcal{F}, \Lambda)= \SSmap(\mathcal{F}', \Lambda')$. So we actually obtain a map
$$\SSmap:~\mathcal{Q}^k/_\sim ~\to~ \mathfrak{X}_{\mathrm{mult}} ^k(\Gpd ).
$$
Here is our main result.
\begin{theorem}\label{Thm:S1-1}
	For a transitive Lie groupoid $\Gpd$, the map $\SSmap $   defined above is a one-to-one  correspondence.
\end{theorem}
\begin{proof}We first show that $\SSmap$ is injective.  If
	$$\SSmap(\mathcal{F}, \Lambda)= \SSmap(\mathcal{F}', \Lambda')=\Pi,$$
	then we have
	$$
	\iota_{t^*df}\Pi=\iota_{t^*df}\overrightarrow{\Lambda}=\iota_{t^*df}\overrightarrow{\Lambda'},\qquad\forall~ f\in {\CinfM}.
	$$
	It follows that $\nu:=\Lambda'-\Lambda\in \Gamma(\wedge^k \ker\rho)$, and then Relation \eqref{eqre2} is immediate. Thus $(\mathcal{F}, \Lambda)\sim (\mathcal{F}', \Lambda')$ in  $\mathcal{Q}^k$.
	
	We now prove that $\SSmap$ is surjective.  Given $\Pi\in \mathfrak{X}_{\mathrm{mult}} ^k(\Gpd )$,   one  has a $\rho$-compatible $k$-tensor $\pi\in \Gamma(TM\otimes (\wedge^{k-1} A))$ given by Proposition \ref{CSXproppi} such that $
	\Pi|_M = B\pi$. 	
	Choose a connection $\gamma$ on the transitive Lie algebroid $A$. Consider $\Lambda:=\Upsilon(B\pi)\in \Gamma(\wedge^k A)$ defined in Equation  \eqref{Ll}.
	We first show the following relation:
	\begin{equation}\label{Eqt:HowBpicomes}
	{\Lambda}-(-1)^k\invstar {\Lambda}=B\pi.
	\end{equation}
	In fact, by Equation  ~\eqref{Eqt:inverstarinDrho}, the left hand side of Equation  ~\eqref{Eqt:HowBpicomes} is
	\begin{eqnarray*}
		\Lambda-\sum_{j=0}^{k}\frac{\minuspower{j}}{j!}\rhopush^{j}\Lambda&=&
		\sum_{j=1}^{k}\frac{\minuspower{j-1}}{j!}\rhopush^{j}\Lambda
		=\sum_{j=1}^{k}\frac{\minuspower{j-1}}{j!}\rhopush^{j-1}\pi=\mbox{RHS of \eqref{Eqt:HowBpicomes}}.
	\end{eqnarray*}
	Here we  have used the fact $\rhopush{\Lambda}=\pi$ (see Equation  ~\eqref{Eqt:rhoLambdaispi}).
	
	Now we define
	$$
	\Xi:=\Pi-\overrightarrow{\Lambda}+\overleftarrow{\Lambda}.
	$$
	As $\Pi$ and $\overrightarrow{\Lambda}-\overleftarrow{\Lambda} $ are both multiplicative, so must be $\Xi$.
	We claim that $\Xi$
	is a $k$-vector on $\Gpd$ tangent to   $t$-fibres.
	In fact, for any $f\in {\CinfM}$, we have
	\begin{eqnarray*}
		\minuspower{k-1}\iota_{t^*f}\Xi&=&[\Pi,t^*f]-[\overrightarrow{\Lambda},t^*f]\\
		&=&\minuspower{k-1}\overrightarrow{\inserts_{df}\pi}-\overrightarrow{[\Lambda,f]}, \quad\mbox{by Equation  ~\eqref{Eqt:Pitstarfbypi};}\\
		&=&0,\quad \mbox{by Equation  ~\eqref{Eqt:Lambdaf=delta0f}.}
	\end{eqnarray*}
	
	As $\Xi$ is multiplicative, one has $ \Xi=(-1)^{k-1}\invstar\Xi$. From this relation and the above, we easily derive that
	$$\iota_{s^*f}\Xi=0,\qquad\forall~ f\in {\CinfM},$$
	and hence $\Xi$ is also  tangent to $s$-fibres.
		Moreover, by Equation  ~\eqref{Eqt:HowBpicomes}, we have $\Xi|_M=0$.
	By the affine property of $\Pi$, Equation  ~\eqref{aff}, it is easy to deduce that $\Xi$  satisfies the condition:
	$$
	\Xi(gr)=L_{g*} \Xi(r)+R_{r*} \Xi(g),\qquad\forall~ (g,r)\in \Gpd ^{(2)}.
	$$
	So $\Xi$ must be of the form $\Pi_{\mathcal{F}}$, for some $\mathcal{F}\in Z^1(\Gpd, \wedge^k\ker\rho)$.  At this point, the fact that $\Pi=\SSmap(\mathcal{F},\Lambda)$ is clear.
\end{proof}

\subsection{Reduced space of multiplicative multivectors on transitive Lie groupoids}

The following fact  describes equivalence relations of multiplicative multivectors on the transitive Lie groupoid $\Gpd$ in terms of the first cohomology   of $\Gpd$ with respect to its  adjoint action on $\ker\rho$.

\begin{lemma}\label{Thm:classbyH1transitivegroupoid}Let $\Pi $ and $\Pi'$ be in $\mathfrak{X}_{\mathrm{mult}}^k(\Gpd )$. Let
	$[(\mathcal{F},\Lambda)]$ and $[(\mathcal{F}',\Lambda')]$  be, respectively, the corresponding primary pairs in $\mathcal{Q}^k/_{\sim}$. Then $\Pi\gaugeequiv \Pi'$ holds
	if and only if  $[\mathcal{F}]=[\mathcal{F}']$ holds in $\coHg ^1(\Gpd,\wedge^k\ker\rho)$.
	
\end{lemma}
The proof of this lemma is easy and omitted. Together with Theorem \ref{Thm:S1-1}, we draw a direct conclusion:
\begin{theorem}
	\label{Thm:transitiveRmult}If a Lie groupoid $\Gpd$ is transitive, then we have ${\Reduced}^k_{\mathrm{mult}} \cong   \coHg ^1(\Gpd,\wedge^k\ker\rho)$.
	\end{theorem}

The following proposition is also easy to prove.
\begin{proposition}Let $\Gpd$ be transitive.
	Let $\Pi  $ be a multiplicative $k$-vector on $\Gpd$ and suppose that it corresponds to a primary pair $[(\mathcal{F},\Lambda)]$, i.e., $\Pi=\overrightarrow{\Lambda}-\overleftarrow{\Lambda}+\Pi_{\mathcal{F}}$. Then the infinitesimal of $\Pi$, $\delta_\Pi\in \differentials^k(A)$, corresponds to a primary pair $[(\hat{\mathcal{F}},\Lambda)]$, i.e.,
	$
	\delta_\Pi=[\Lambda,~\cdot~]+\hat{\mathcal{F}}
	$.
\end{proposition}
In summary, if the groupoid $\Gpd$ is transitive, we have the following relations:
\begin{equation*}
\xymatrix{
	\mathfrak{X}_{\mathrm{mult}}^k(\Gpd ) \ni& \Pi \ar[d]_{ } \ar[r]^{1:1\quad} & [(\mathcal{F},\Lambda)] \ar[d]_{ } \ar[r]^{ } & [\mathcal{F}]\ar[d]^{ }&\in \coHg ^1(\Gpd,\wedge^k \ker\rho) \cong {\Reduced}^k_{\mathrm{mult}}  \\
	\differentials^k(A)\ni& \delta_\Pi \ar[r]^{1:1\quad } & [(\hat{\mathcal{F}},\Lambda)] \ar[r]^{ } & [\hat{\mathcal{F}}]&\in  \coHg ^1(  A,\wedge^k \ker\rho)\cong {\Reduced}^k_{\mathrm{diff}}.  }
\end{equation*}
Here all vertical arrows refer to the method of taking  infinitesimal. The rightmost one coincides with the Van Est map $\mathrm{VE}_1$ in Diagram \eqref{Diag:fullpicturejetH1vanEast}.

\section{Reduced space of Lie algebroid differentials (ordinary case)} \label{Sec:2ndclassofdifferentials}

\textbf{Convention.} In this section, we only consider   ordinary cases of $k$, i.e.,
$1\leqslant  k\leqslant   \Top $.

\subsection{Embedding of $\ReducedProper^k$ into $\coHg ^1(\core,\wedge^k A)$}
We start by introducing a special type of $1$-cocycles associated to the  bundle of {isotropy} jet Lie algebras $\core=T^*M\otimes A=\mathrm{Hom}(TM,A)$ which admits a canonical adjoint action  on $\wedge^k A$.   In fact,   the Lie bracket in $\core$ reads
$$[df_1\otimes u_1,df_2\otimes u_2] = \rho(u_1)(f_2) df_1\otimes u_2-\rho(u_2)(f_1) df_2\otimes u_1\,
$$  (see  Equation  \eqref{Eqt:jetbracket3}). So, the action of $\core$   on   $ \wedge^k A $ is given by
$$\ad_{df\otimes u}w = -[w,f]\wedge u
$$  (see Equation   \eqref{Eqn:jetadjoint2}).
More details can be found in Appendix \ref{Appendix:jetLiegroupoid}.

\begin{lemma}\label{lspecailsubspaceinZ1core}~
	\begin{itemize}
		\item[(1)]
	Let $\pi\in \Gamma(TM\otimes(\wedge^{k-1}A))$ be a $\rho$-compatible $k$-tensor. Then the map $\mu_\pi:~\core\to \wedge^k A$ defined by
	\begin{equation}\label{Eqn:special1cocyclebypi0}
	\mu_\pi (H)= (H\otimes (\mathrm{id}_A)^{\otimes {k-1}})\pi,
	\end{equation}
where $H\in \core =\mathrm{Hom}(TM,A)$,  is a $1$-cocycle, i.e., $\mu_\pi\in Z^1(\core, \wedge^kA)$.
\item[(2)] Every coboundary $d_\core \tau\in B^1(\core,\wedge^kA)$, for some $\tau\in \Gamma(\wedge^kA)$, can be written in the form $d_\core\tau=-\mu_\pi$, where $\pi=\rhopush\tau$.
\end{itemize}
\end{lemma}

\begin{remark}Equation  ~\eqref{Eqn:special1cocyclebypi0} can be reformulated as
	\begin{eqnarray}\label{Eqn:special1cocyclebypi}
	\mu_\pi (\xi\otimes u)&=& \minuspower{k-1} (\iota_{ \xi}\pi) \wedge u,
	\end{eqnarray}
	where  $\xi \in \Omega^1(M)$ and $u\in\Gamma(A)$.
\end{remark}
\begin{proof}[Proof of Lemma \ref{lspecailsubspaceinZ1core}.] For (1), it amounts to prove the following identity:
	\begin{equation*}
	\mu_\pi([df_1\otimes u_1,df_2\otimes u_2])
=\ad_{ df_1\otimes u_1}\mu_\pi(df_2\otimes u_2)-
\ad_{ df_2\otimes u_2}\mu_\pi(df_1\otimes u_1).
	\end{equation*}
	The proof is a straightforward verification; one uses the definition of $\mu_\pi$ in Equation  ~\eqref{Eqn:special1cocyclebypi}, the explicit expressions of Lie bracket   in $\core$ and its adjoint action, and the $\rho$-compatible condition \eqref{Eqt:fine-condition} of $\pi$.   	
	The second statement can be verified by   direct computations as well.
\end{proof}

Recall the
reduced space of $\rho$-compatible  $k$-tensors $\ReducedProper^k$ defined in Definition \ref{def:reducedproperspace}. Indeed, it is part of a larger space:
\begin{lemma}\label{lembeddingReducedPropertoHjet}
	The map $\mu:~\pi\mapsto \mu_\pi$ induces an embedding of vector spaces
	$$\mu:~\ReducedProper^k\hookrightarrow \coHg ^1(\core,\wedge^k A).$$
\end{lemma}
The proof is easy and omitted.


\subsection{Characteristic pairs of  Lie algebroid differentials}
We need to consider the jet Lie algebroid $\jet A$ stemming from the given Lie algebroid $A$,  and the adjoint action of $\jet A$ on $A$ and $\wedge^k A$.
The basic knowledge related to these objects can be found in   Appendix
\ref{Appendix:jetLiealgebroid}. Let us denote by $d_{\jet A}\colon C^n(\jet A,\wedge^k A)\to C^{n+1}(\jet A,\wedge^k A) $ the standard differential associated to the $\jet A$-module   $\wedge^k A$.

\begin{definition}\label{k-cp}Let $A$ be a Lie algebroid over $M$. A \textbf{$k$-characteristic pair} on $A$
is a pair $(\chi , \pi)$, where \begin{enumerate}
	\item[(1)]  $\chi\in Z^1(\jet A,\wedge^k A)$ is  a Lie algebroid $1$-cocycle   with respect to the adjoint action of $\jet A$ on $\wedge^k A$,   \item[(2)] $\pi\in \Gamma(TM\otimes (\wedge^{k-1} A))$  is a $\rho$-compatible $k$-tensor, and
	\item[(3)] the restriction of $\chi$ on $\core\subset \jet A$ coincides with $ \mu_\pi$, i.e.,
	\begin{eqnarray}\label{Eqn:kappapicompatible}
	\chi (\xi\otimes u)=(-1)^{k-1}(\iota_{ \xi}\pi) \wedge u,
	\end{eqnarray}
	for all  $\xi \in \Omega^1(M)$, $u\in\Gamma(A)$.
\end{enumerate}

\end{definition}

 As we have assumed that $1\leqslant  k\leqslant  \Top $, the following fact can be easily derived from Equation  ~\eqref{Eqn:kappapicompatible}.
\begin{lemma}\label{luniquenessofpi}  If $(\chi ,\pi_1)$ and $(\chi ,\pi_2)$ are both $k$-characteristic pairs, then we have $\pi_1=\pi_2$.
\end{lemma}
So    the second data $\pi$ in a characteristic pair is actually   determined by the first data $\chi $.

 \begin{proposition}\label{Prop:kdifferentialinpair}~There is a one-to-one correspondence between
$k$-differentials $\delta=(\deltazero,\deltaone)\in \differentials^k(A)$ and $k$-characteristic  pairs $ (\chi , \pi) $ on $A$ such that
\begin{eqnarray}
\label{Eqn:delta1ukappa}
\deltazero u &=&  \chi (\liftingd u),\quad u\in \Gamma(A),\\
\label{Eqn:delta0fpi}
\deltaone f &=& (-1)^{k-1}\iota_{ df}\pi,\quad f\in  {\CinfM}.
\end{eqnarray}
Here   $\liftingd: \Gamma(A)\to \Gamma(\jet A)$ is the standard lifting map in Equation  ~\eqref{Eqt:liftingd}.

Moreover,  an exact $k$-differential $\delta=[\tau,\cdot]$, where $\tau\in\Gamma(\wedge^kA)$,  corresponds to the characteristic pair $(\chi ,\pi)$ where
$$ \chi = -d_{\jet A}\tau,\qquad
\pi=\rhopush \tau.
$$
\end{proposition}
\begin{proof} Given a characteristic pair  $(\chi ,\pi)$, one defines $\delta=(\deltazero,\deltaone)$ which is a derivation in $\Der^{0,k-1}(A)$. The $1$-cocycle condition of $\chi$ and the $\rho$-compatible condition of $\pi$ translate to Equations  \eqref{Eqt:kdifferential2.5} $\sim $ \eqref{Eqt:kdifferential4}, and we thus get a $k$-differential $\delta \in \differentials^k(A)$. The converse is  also natural.
\end{proof}
	This proposition is of fundamental importance in this section, and accordingly, we will say that $(\chi ,\pi)$ is the  {characteristic pair} of the differential $\delta$ if they are related by Equations  ~\eqref{Eqn:delta1ukappa} and \eqref{Eqn:delta0fpi}, and simply write $\delta=(\chi ,\pi)$.	


	\subsection{Embedding of ${\Reduced}^k_{\mathrm{diff}}$ into $\coHg ^1(\jet A,\wedge^k A) $}

 Proposition \ref{Prop:kdifferentialinpair} implies the following fact  which characterizes  equivalence relations of differentials on the Lie algebroid $A$ in terms of characteristic pairs on $A$.

\begin{proposition}\label{Prop:algebroidgaugeiff}Let $\delta=(\chi ,\pi) $ and $\delta'=(\chi ',\pi')$ be in $\differentials^k(A) $.  Then $\delta\gaugeequiv \delta'$ holds
if and only if $[\chi]=[\chi']$ holds in $\coHg ^1(\jet A,\wedge^k A)$.
\end{proposition}
\begin{proof}We first show the implication that ``$\delta\gaugeequiv \delta'$''
	$\Rightarrow$ ``$[\chi]=[\chi']$''. In fact,  if $\delta'= \delta+[\tau,\cdot~]$, for some $\tau\in \Gamma(\wedge^k A)$, then by Proposition \ref{Prop:kdifferentialinpair}, the corresponding characteristic pairs are related by
$		\chi ' = \chi -d_{\jet A}\tau$ and
		$\pi' = \pi+D_{\rho}\tau$.
Hence $[\chi]=[\chi']$.

Conversely, if $\chi'=\chi-d_{\jet A}\tau$ for some $\tau\in \Gamma(\wedge^k A)$, then the $k$-characteristic pairs $(\chi',\pi')$ and $(\chi-d_{\jet A}\tau,\pi+\rhopush\tau)$ share the same first entry. By Lemma \ref{luniquenessofpi}, we have $\pi'=\pi+\rhopush\tau$. It follows that $\delta'=\delta+[\tau,\cdot]$, as required.
\end{proof}

Before stating the main theorem of this section, we need to recall or introduce some  maps:
\begin{enumerate}
	\item Define a map $\kappa$: $\differentials^k(A)\to \coHg ^1(\jet A,\wedge^k A)$  by sending
	$$
	\delta = (\chi ,\pi) \quad\mapsto\quad [\chi ] .
	$$
	 	By Proposition \ref{Prop:algebroidgaugeiff}, the above $\kappa$ induces an embedding of vector spaces
	$$
	\kappa:~{\Reduced}^k_{\mathrm{diff}} \quad\hookrightarrow\quad \coHg ^1(\jet A,\wedge^k A).
	$$
	\item Define a map $i^!$: $\differentials^k(A)\to \ReducedProper^k$  by sending
	$$
	\delta= (\chi ,\pi) \quad\mapsto\quad [\pi ] .
	$$
 	 It is apparent that ``$\delta\gaugeequiv \delta'$'' implies    ``$[\pi]=[\pi']$''. Therefore, it induces a map  $i^!$: $\Reduced^k_{\mathrm{diff}}\to \ReducedProper^k$.
	\item The embedding $i: ~\core \to \jet A$ (see Sequence \eqref{Eqt:jetalgebroidexactseq}) of Lie algebroids induces a natural  map
 	$$i^*: ~\coHg ^1(\jet A,\wedge^k A)\to \coHg ^1(\core,\wedge^k A).$$

	\item
	Recall the embedding  map $\mu:~{\ReducedProper^k}\hookrightarrow \coHg ^1(\core,\wedge^k A)$ introduced in Lemma \ref{lembeddingReducedPropertoHjet}.   

\end{enumerate}Combining the four maps $\kappa$, $i^!$, $\mu$, and $i^*$  as above,   we are able to  state our   result that describes the reduced  space ${\Reduced}^k_{\mathrm{diff}}$.
\begin{theorem}\label{Thm:Rdiffaspullback}
	Let  $1\leqslant  k\leqslant  \Top $ be an integer.  The reduced space of Lie algebroid $k$-differentials $\Reduced^k_{\mathrm{diff}}$ is the pullback      of maps $i^*$ and $\mu$. In other words, the diagram
	\begin{equation*}
	\xymatrix{
		\Reduced^k_{\mathrm{diff}} \ar@{^{(}->}[d]_{\kappa} \ar[r]^{i^!} & \ReducedProper^k \ar@{^{(}->}[d]^{\mu} \\
		\coHg ^1(\jet A,\wedge^k A) \ar[r]^{i^*} & \coHg ^1(\core,\wedge^k A)  }
	\end{equation*}
	is commutative and the map
	\begin{eqnarray*}
	 {\beta}:~	\Reduced^k_{\mathrm{diff}} &  {\longrightarrow} &{{\coHg ^1(\jet A,\wedge^k A)}{ \ _{i^*}\times_\mu }\ReducedProper^k=\set{(x,y)\in \coHg ^1(\jet A,\wedge^k A)  \times   \ReducedProper^k; ~ i^* (x)=\mu (y)~ }} \\
	 	{[\delta=(\chi,\pi)]}  & \mapsto & (\kappa[\delta],i^![\delta])=([\chi],[\pi])
	\end{eqnarray*}
	is an isomorphism.
\end{theorem}
\begin{proof} The fact that the diagram is commutative is due to condition (3) in Definition \ref{k-cp} of characteristic pairs. The map $\beta$ is an injection because $\kappa$ is injective. It remains to show that $\beta$ is surjective, and it amounts to prove
	 the relation of inclusion $(i^*)^{-1}{\Img \mu}\subset \Img \kappa$. We suppose that $[\chi]\in (i^*)^{-1}{\Img \mu} $, where $\chi\in Z^1(\jet A,\wedge^k A)$. So, there exists a  $\rho$-compatible $k$-tensor $\pi$ such that $i^*(\chi)=\chi|_\core=\mu_\pi$. The pair $(\chi,\pi)$ becomes a $k$-characteristic pair as in Definition \ref{k-cp}, and it gives rise to a $k$-differential $\delta\in \differentials^k(A)$ by Proposition \ref{Prop:kdifferentialinpair}. We then have the desired relation $\kappa([\delta])=[\chi]\in \Img \kappa$.
\end{proof}

 By this theorem, we  see that   $\Reduced^k_{\mathrm{diff}} \cong \Img \kappa=(i^*)^{-1}\Img \mu$.
Also note   Lemma \ref{lleftHjet} in the appendix which gives the kernel of $i^*$, namely the space $\coHg ^1(  A,\wedge^k \ker\rho )$. Therefore, one is able to draw the following   commutative diagram:
\begin{equation*}\label{Diag:fullpictureRdiff}
\xymatrix{
	0   \ar[r]^{ } & \coHg ^1(  A,\wedge^k \ker\rho ) \ar[d]_{\mathrm{id}} \ar[r]^{p^!} & \Reduced^k_{\mathrm{diff}} \ar@{^{(}->}[d]_{\kappa} \ar[r]^{i^!} & \ReducedProper^k \ar@{^{(}->}[d]^{\mu} \\
	0 \ar[r]^{} & \coHg ^1(  A,\wedge^k \ker\rho ) \ar[r]^{p^*} & \coHg ^1(\jet A,\wedge^k A) \ar[r]^{i^*} & \coHg ^1(\core,\wedge^k A),  }
\end{equation*}
where the two horizontal sequences are exact,  and the vertical arrows are injective.  From this fact we find a decomposition $\Reduced^k_{\mathrm{diff}}\cong \coHg ^1(  A,\wedge^k \ker\rho )\oplus (\Img \mu\cap \Img i^*)$ although it  is not canonical.




\section{Reduced space of multiplicative  multivectors on Lie groupoids (ordinary case)}\label{Sec:multiplicativemultivectors}

Our method of finding the reduced space ${\Reduced}^k_{\mathrm{mult}}$ is parallel to the previous section. However, more efforts are needed before we reach our conclusion. In particular, we need to study in depth the precise composition of  multiplicative  multivectors on Lie groupoids.

\textbf{Convention.} In this section,  we again assume that
$1\leqslant  k\leqslant   \Top $.

\subsection{Embedding of $\ReducedProper^k$ into  $\coHg ^1(\heat,\wedge^k A)$}
In this part, we introduce a special type of $1$-cocycles associated to the  bundle of {isotropy} jet groups $\heat\cong\underline{\mathrm{Hom}}(TM,A)$ and its adjoint action  on $\wedge^k A$. The reader is referred to Appendix \ref{Appendix:jetLiegroupoid} for
an explicit description of the group structure of $\heat$ and its action on $\wedge^k A$.

The following lemma is parallel to Lemma \ref{lspecailsubspaceinZ1core}.
\begin{lemma}\label{lspecailsubspaceinZ1heat}~
	\begin{itemize}
		\item[(1)]
	Let $\pi\in \Gamma(TM\otimes(\wedge^{k-1}A))$ be a $\rho$-compatible $k$-tensor. Then the map ${U_\pi}:~\heat\to \wedge^k A$ defined by
	\begin{equation}\label{Eqn:special1cocyclebypi2}
	{U_\pi}([h_x])=(B\pi)_x-L_{[h_x]} (B\pi)_x,\qquad \forall~ [h_x]\in  \heat_x,~x\in M,
	\end{equation}
   is a $1$-cocycle, i.e., ${U_\pi}\in Z^1(\heat, \wedge^kA)$. (
Here $B\pi$ is defined in Equation  ~\eqref{Bpi}.)
\item[(2)] Every coboundary $d_\heat \tau\in B^1(\heat,\wedge^kA)$, for some $\tau\in \Gamma(\wedge^kA)$, can be written in the form $d_\heat\tau=-{U_\pi}$, where $\pi=\rhopush\tau$.
\end{itemize}
\end{lemma}
\begin{proof}
 A priori,  the right hand side of Equation  ~\eqref{Eqn:special1cocyclebypi2} belongs to $\Gamma(\wedge^k (TM\oplus A))$. So we need to explain why it only has the $\Gamma(\wedge^k   A) $-component, i.e.,  tangent to the $s$-fibres, or
\begin{equation}\label{Eqt:temp2}\iota_{s^* \xi}((B\pi)_x-L_{[h_x]} (B\pi)_x)
=0,\quad \forall~ \xi\in T^*_xM.\end{equation}
In fact, by Equation  ~\eqref{Eqn:specialpropertypi2} of Lemma \ref{Lemma:ixitoBpi},  \begin{equation}\label{Eqt:temp6}\iota_{s^* \xi} (B\pi) =\iota_{\xi} (B\pi) =(-1)^{k-1} \invstar  (\iota_\xi \pi) \end{equation}
is tangent to the $t$-fibres (because $\iota_\xi \pi$ is tangent to the $s$-fibres). Hence
\begin{eqnarray*}
 \iota_{s^* \xi}{\littlecirc} L_{[h_x]} (B\pi)_x =L_{[h_x]}{\littlecirc} \iota_{s^* \xi} (B\pi)_x= \iota_{s^* \xi} (B\pi)_x\,.
 \end{eqnarray*}
 This proves Equation  ~\eqref{Eqt:temp2}.

 Now we wish to show the equality
 \begin{equation}\label{Eqt:temp7}
 {U_\pi}([h_1][h_2])={U_\pi}([h_1])+\Ad_{[h_1]}{U_\pi}([h_2]), \quad \forall~ [h_1], [h_2]\in \heat.
 \end{equation}
 In fact, the right hand side of the above equation is ${U_\pi}([h_1])+L_{[h_1]}{U_\pi}([h_2])$ (by Equation  ~\eqref{Eqt:adjointofheatonwedgekA}). Then by substituting the definition of ${U_\pi}$, one easily gets the left hand side of Equation  ~\eqref{Eqt:temp7}.  So, ${U_\pi}$ is indeed a $1$-cocycle.

 The second claim is verified as follows. If $\pi=\rhopush \tau$, then one has
 \begin{eqnarray*}
 U_{\pi}([h_x])&=&(B\pi)_x-L_{[h_x]} (B\pi)_x\\
 &=& (\overrightarrow{\tau }-\overleftarrow{\tau })_x-L_{[h_x]}(\overrightarrow{\tau }-\overleftarrow{\tau })_x,\qquad\mbox{ by Lemma \ref{lexactmultiplicativepi};}\\
 &=&\tau_x-L_{[h_x]}\tau_x=\tau_x-\Ad_{[h_x]}\tau_x\\
 &=&-d_\heat \tau ([h_x]).
 \end{eqnarray*}
\end{proof}
Moreover, we have a basic fact:
\begin{proposition}\label{Prop:Mpitomupi}The infinitesimal of the group $1$-cocycle ${U_\pi}$ defined in Lemma \ref{lspecailsubspaceinZ1heat}, is the Lie algebra $1$-cocycle $\mu_\pi$ defined in Lemma \ref{lspecailsubspaceinZ1core}.
\end{proposition}
\begin{proof}It suffices to show the following identity:
\begin{equation}\label{Eqt:chatproperty}
\hat{U}_{\pi} (H)=(H\otimes (\mathrm{id}_A)^{\otimes {k-1}})\pi,
\end{equation}where $H \in \core=\mathrm{Hom}(TM,A)$.
 According to the definition of $\hat{c}$ in Equation  ~\eqref{Eqt:hatcdefinition}, we need to find a curve $\gamma(\epsilon)$ in $\heat$ with $\gamma'(0)=H$. We take $\gamma(\epsilon)=1+\epsilon H$, and then we have
 $$(\gamma(\epsilon))^{-1}=1-\epsilon H{\littlecirc} (1+\epsilon \rho{\littlecirc} H)^{-1}=1-\epsilon H + o(\epsilon).$$
 Hence, we are able to compute
 \begin{eqnarray*}
 \hat{U}_{\pi}(H)&=& \frac{d}{d\epsilon }|_{\epsilon =0}~ {{U_\pi}}(\gamma(\epsilon)^{-1})\\
 &=& \frac{d}{d\epsilon }|_{\epsilon =0}~( B\pi -L_{\gamma(\epsilon)^{-1}} (B\pi) ) \\
 &=& \frac{d}{d\epsilon }|_{\epsilon =0}~( B\pi -L_{1-\epsilon H} (B\pi) )\\
 &=& \frac{d}{d\epsilon }|_{\epsilon =0}~\mathrm{pr}_{\wedge^k A}( B\pi -L_{1-\epsilon H} (B\pi) ). \end{eqnarray*}
 Now using the formula of $L_{[h_x]}$ in Equation  ~\eqref{Eqn:Lhx}, the last step yields exactly the right hand side of Equation  ~\eqref{Eqt:chatproperty}.
\end{proof}

Also recall Lemma \ref{lembeddingReducedPropertoHjet} where we  find an embedding   $\mu:~{\ReducedProper^k}\hookrightarrow \coHg ^1(\core,\wedge^k A)$.   We have a similar fact.
\begin{lemma}\label{lUembedding}
		The map $U:~\pi\mapsto U_\pi$ induces an embedding of vector spaces
	$$U:~\ReducedProper^k\hookrightarrow \coHg ^1(\heat,\wedge^k A).$$
	
\end{lemma}



\subsection{Characteristic pairs of multiplicative multivectors}

Recall that $A$, the tangent Lie algebroid of $\Gpd$, is a module of the jet groupoid $\jet\Gpd$ via adjoint actions (see Definition \ref{Def:jetGpdadjoint}).  Therefore, $\jet\Gpd$ also acts on $\wedge^k A$.
Hence we have a coboundary operator
$d_{\jet \Gpd}: C^{n}(\jet\Gpd , \wedge^k A)\to C^{n+1}(\jet \Gpd ,\wedge^k A)$ associated to the    $\jet\Gpd $-module structure on $\wedge^k A$.
We denote by $Z^1(\jet\Gpd , \wedge^k A)$ the set of $1$-cocycles $c:\jet\Gpd\to \wedge^kA$ (see Equation  ~\eqref{Eqt:1cocyclecondition}).
\begin{definition}\label{Defn:groupoidcharpair}A  \textbf{$k$-characteristic pair} on $\Gpd$ is a pair
  $(c,\pi)  \in Z^1(\jet \Gpd , \wedge^k A)\times \Gamma(TM\otimes (\wedge^{k-1} A))$, where $\pi$ is a $\rho$-compatible $k$-tensor  and when $c$ is restricted to $\heat$, it coincides with ${U_\pi}$ defined in Lemma \ref{lspecailsubspaceinZ1heat}, i.e.,   \begin{eqnarray}
\label{formula1}c([h_x])&=&(B\pi)_x-L_{[h_x]} (B\pi)_x,\qquad \forall~ [h_x]\in  \heat_x,~x\in M.
\end{eqnarray}
\end{definition}

We need the following key fact.
\begin{theorem}\label{Thm:mainexplicitformula} There is a one-to-one correspondence between multiplicative $k$-vectors $\Pi$ on a Lie groupoid $\Gpd \rightrightarrows M$, and   $k$-characteristic pairs $(c,\pi)$ on $\Gpd$
such that
\begin{equation}\label{formula3}
\Pi_g=R_{g*}c([b_g])+L_{[b_g]}(B\pi)_{s(g)}
\end{equation}
holds for all $g\in \Gpd $ and   bisections $b_g$ through $g$, where $B\pi$ is given by Equation  ~\eqref{Bpi}.
\end{theorem}
By this theorem, we will say that $(c,\pi)$ is the characteristic pair of the multiplicative multivector $\Pi$ if they are related as in Equation  ~\eqref{formula3}, and simply write $\Pi=(c,\pi)$. As we stated in the introduction, this indeed gives a standard decomposition of multiplicative multivectors on a Lie groupoid.


\begin{remark}
	In a  recent work \cite{CMS1}     explicit formulas for multiplicative forms are established. 	
	Therefore, it is natural to ask whether multiplicative forms  and more generally, multiplicative  tensors  on Lie groupoids (see \cite{BD})   admit decompositions into jet groupoid cocycles and $\rho$-compatible tensors  ---   analogues of Equation \eqref{formula3} of
	multiplicative multivectors.    So far, we have  answers for multiplicative forms (see   \cite{CLL2022}). 

	

\end{remark}

We divide the proof of Theorem \ref{Thm:mainexplicitformula} into two lemmas.

\begin{lemma}\label{1}
Let $(c,\pi)$ be a  $k$-characteristic pair on $\Gpd$. Then $\Pi$ defined in  Equation  ~\eqref{formula3} is a
multiplicative $k$-vector.\end{lemma}

\begin{proof}
	We analyze  $\Pi$ step by step to prove     that it is multiplicative:

$\bullet$~ The $k$-vector $\Pi$ defined in Equation  ~\eqref{formula3} is well-defined, i.e.,  $\Pi_g$ is independent of the choice of $b_g$.

Let $b'_g$ be another   bisection through $g$.
 Then it must be of the form $b'_g=b_g\cdot h_x$, for some   bisection $h_x$ through $x=s(g)$. We have,
  \begin{eqnarray*}
R_{g*} c([b'_g])+L_{[b'_g]}(B\pi)_x&=&R_{g*}  c([b_g] \cdot [h_x])+L_{[b_g]} L_{[h_x]}  (B\pi)_x\\
&=&R_{g*}(c({[b_g]}) +  \Ad_{[b_g]}c([h_x]))+ L_{[b_g]} L_{[h_x]}  (B\pi)_x\\ &=&R_{g*}c([b_g])+L_{[b_g]}(c([h_x])+L_{[h_x]}  (B\pi)_x)\\
&=& R_{g*}c([b_g])+L_{[b_g]}(B\pi)_x\,.\end{eqnarray*}
 Here we have used the fact that $c$ is a $1$-cocycle,   and Relation \eqref{formula1}.

$\bullet$~ We claim that $\Pi$ satisfies the affine condition
\begin{eqnarray}\label{aff2}
\Pi_{gr}=L_{[b_g ]}\Pi_r+R_{[b'_{r}]}\Pi_g-L_{[b_g ]}R_{[b'_{r}]}\Pi_{x},\end{eqnarray}
for any $g,r\in \Gpd $ composable, $x=s(g)=t(r)$, and   bisections $b_g$ and $b'_r$.

In fact,  the left hand side  of Equation  ~\eqref{aff2} equals to
\begin{eqnarray*}
R_{r*}R_{g*} c([b_g][b'_r ])+L_{[b_g]} L_{[b'_r] }(B\pi)_{s(r)}.
\end{eqnarray*}Substituting the $1$-cocycle condition of $c$:
\[c([b_g][b'_r ])=c([b_g])+\Ad_{[b_g]} c([ b'_r ]), \]
into the above, and using the fact that $c([1+0_x])=0$ and $\Pi_x=(B\pi)_x$, one easily gets the right hand side of  Equation  ~\eqref{aff2}, where $1+0_x\in \heat_x$ is the identity element.

$\bullet$~ We show that $\iota_{s^* \xi} \Pi$ is left-invariant for any $\xi\in \Omega^1(M)$.

 In fact, it is obvious that $\iota_{s^* \xi} (R_{g*}c([b_g]))=0$. Moreover, as $s_*{\littlecirc} L_{[b_g]}=s_*$, we have
 $$\iota_{s^* \xi} \Pi_g=
 \iota_{s^* \xi}{\littlecirc} L_{[b_g]}(B\pi)_{s(g)}
=L_{[b_g]}{\littlecirc}\iota_{s^* \xi}(B\pi)_{s(g)}.
 $$
 By the identity in Equation  ~\eqref{Eqt:temp6}, the right hand side of  the above equation is in fact the left translation of $ \iota_{s^* \xi}(B\pi)$, a section of $\wedge^{k-1} \ker t_*$.

$\bullet$~Finally, $M\subset \Gpd $ is   coisotropic relative to $\Pi$. In fact, by definition, $
\Pi|_M= B\pi
$ has no component in $\wedge^k A$.

Combining the above facts and by Lemma \ref{multi vector}, we see that $\Pi$ defined in Equation  ~\eqref{formula3} is multiplicative.
\end{proof}

\begin{lemma}
 Given $\Pi\in \mathfrak{X}_{\mathrm{mult}}^k(\Gpd )$, let \[\pi:=\mathrm{pr}_{T M\otimes (\wedge^{k-1} A)} (\Pi|_M),\] i.e., the $\Gamma(T M\otimes (\wedge^{k-1} A))$-component of $\Pi$ when it is restricted to the base manifold $M$. Define a map $
c:~\jet \Gpd \to \wedge^k A$ by
\begin{eqnarray}\label{formulac}\qquad c([b_g]):=R_{g^{-1}*}(\Pi_g-L_{[b_g]} \Pi_{s(g)}).
\end{eqnarray}
Then $(c,\pi)$ is a   $k$-characteristic pair satisfying \eqref{formula3}.
\end{lemma}

\begin{proof}
The proof is divided into the following  steps:

$\bullet$~ By Proposition \ref{CSXproppi}, the tensor field $\pi$ derived from $\Pi$ is a $\rho$-compatible $k$-tensor.

$\bullet$~ We explain why $c$ is well-defined, i.e., the term $(\Pi_g-L_{[b_g]} \Pi_{s(g)})
$ in Equation  ~\eqref{formulac} is tangent to the $s$-fibre at $g$.
It suffices to show that
$$\iota_{s^* \xi}(\Pi_g-L_{[b_g]} \Pi_{s(g)})
=0.$$
In fact, by (3) of Lemma \ref{multi vector},  we have
\begin{eqnarray*}
\iota_{s^* \xi}(\Pi_g)&=& L_{g*} {\littlecirc} \iota_{s^* \xi} (\Pi_{s(g)})
 =
 L_{[b_g]} {\littlecirc} \iota_{s^* \xi} (\Pi_{s(g)})
=
 \iota_{s^* \xi} {\littlecirc} L_{[b_g]}   (\Pi_{s(g)}) .
\end{eqnarray*}

$\bullet$~ The map $c$ is a $1$-cocycle, i.e., for composable bisections $b_g$ and $b'_r$, we have
$$
c([b_g][b'_r])=c([b_g])+\Ad_{[b_g]} c([b'_r]).
$$
 In fact, using the affine property  \eqref{aff} and definition of $c$, one gets this relation. We omit the details.

$\bullet$~  We finally show Equation  ~ \eqref{formula1}. In fact, for  $[h_x]\in  \heat_x$,~$x\in M$, we have
\begin{eqnarray*}
 c([h_x])&=&\Pi_{x}-L_{[h_x]} \Pi_x=(B\pi)_x-L_{[h_x]} (B\pi)_x.
\end{eqnarray*}
Here we used the fact that $\Pi|_M=B\pi$, by Proposition \ref{CSXproppi}.
 \end{proof}

\begin{example}\label{Prop:exactmultiplicativecharpairs}
	Given $\tau \in \Gamma(\wedge^k A)$, the characteristic pair of the exact multiplicative
	$k$-vector   $(   \overrightarrow{\tau }-\overleftarrow{\tau })$   is
	$(c=-d_{\jet \Gpd}\tau ,\pi=D_\rho \tau)$.
	In fact, the relation $\pi=\rhopush \tau$ is due to Lemma \ref{lexactmultiplicativepi}. Moreover, by definition of $c$ in Equation  ~\eqref{formulac}, we have
	\begin{eqnarray*}
		c([b_g])&=&R_{g^{-1}*}((\overrightarrow{\tau }-\overleftarrow{\tau })_g -L_{[b_g]} (\overrightarrow{\tau }-\overleftarrow{\tau })_x ),\quad \mbox{ where }x=s(g), y=t(g)\\ &=&
		R_{g^{-1}*}(R_{g*}\tau_{y} - (-1)^k L_{g*}{\littlecirc} \invstar  (\tau_{x})
		- L_{[b_g]}(\tau_x-(-1)^k  {\littlecirc} \invstar  (\tau_{x})) )
		\\
		&=&
		\tau_y -\Ad_{[b_g]}\tau_x =-d_{\jet \Gpd}\tau ([b_g]).
\end{eqnarray*}\end{example}

\begin{example}
A  $1$-vector  $\Pi\in \mathfrak{X}_{\mathrm{mult}}(\Gpd )$ corresponds to a characteristic pair $(c,\pi)$ where $c\in Z^1(\jet \Gpd , A)$ and $\pi\in \mathfrak{X}(M)$. As $B\pi=\pi$, the compatibility condition is quite simple:
$$c([h_x])=\pi_x-L_{[h_x]}(\pi)_x=-H(\pi) , \quad~ \forall~ [h_x]=1+H \in \heat_x\,.$$
Moreover, we have the expression
$$
\Pi_g=R_{g*}{\littlecirc} c([b_g])+ b_{g*} (\pi_{s(g)}),
$$
for all $g\in \Gpd $ and   bisection $b_g$ passing through $g$.
 \end{example}

\begin{example}
A   $2$-vector $\Pi\in \mathfrak{X}_{\mathrm{mult}}^2(\Gpd )$ corresponds to a characteristic pair  $(c,\pi)$ where $c\in Z^1(\jet\Gpd , \wedge^2 A)$ and $\pi\in \Gamma(TM\otimes A)$ is a $\rho$-compatible $2$-tensor.  Moreover,
$$
B\pi=\pi-\frac{1}{2}D_\rho \pi.
$$
Therefore, we have
\[\Pi_g=R_{g*}{\littlecirc} c([b_g])+   L_{[b_g ]}(\pi-\frac{1}{2}D_\rho \pi
)_{s(g)}\,.\]

\end{example}

We turn now to  transitive Lie algebroids and groupoids, and give the characteristic pairs of Lie algebroid differentials  and   multiplicative multivectors  in spirit of Theorems \ref{Thm:s1-1} and \ref{Thm:S1-1}.

\begin{proposition}\label{prop:transitivealgebroidprimarypairtocharpair}Let $A$ be a transitive Lie algebroid and    $(\Omega,\Lambda)\in Z^1(A,\wedge^k \ker  \rho)\times \Gamma(\wedge^k A) $. Then we have the following two facts:
	\begin{itemize} \item[1)] The  characteristic pair of the   $k$-differential   $\delta=[\Lambda,~\cdot~]+\Omega$ is given by
		$$ \chi = -d_{\jet A}\Lambda+p^*\Omega,\qquad
		\pi=\rhopush \Lambda;
		$$Here $p^*:~Z^1(A,\wedge^k \ker\rho)\to Z^1(\jet A,\wedge^k A)$ is the pullback of the projection  $p: ~\jet A\to A$   (see Sequence
		\eqref{Eqt:jetalgebroidexactseq});
		\item[2)] The   map $\kappa:~\differentials^k(A)\to \coHg ^1(\jet A,\wedge^k A)$ sends   $\delta=[\Lambda,~\cdot~]+\Omega$ to $p^*[\Omega]$.  Here $p^*: \coHg ^1(A,\wedge^k \ker \rho)\to \coHg ^1(\jet A,
		\wedge^k A)$ is the embedding map in Lemma \ref{lleftHjet}.
	\end{itemize}
\end{proposition}
\begin{proposition}\label{prop:transitivegroupoidprimarypairtocharpair}Let $\Gpd$ be a transitive Lie groupoid and  $(\mathcal{F}, \Lambda)\in $ $Z^1(\Gpd ,\wedge^k \ker  \rho)\times \Gamma(\wedge^k A)$.
	Then the following statements are true:
	\begin{itemize}
		\item[1)] The  characteristic pair  of the multiplicative $k$-vector   $\Pi= \overrightarrow{\Lambda}-\overleftarrow{\Lambda}+\Pi_{\mathcal{F}}$ is given by
		$$  c = -d_{\jet\Gpd}\Lambda+P^*\mathcal{F},\qquad
		\pi=\rhopush \Lambda;
		$$Here $P^*:~Z^1(\Gpd,\wedge^k \ker\rho)\to Z^1(\jet \Gpd,\wedge^k A)$ is the pullback of the projection  $P: ~\jet \Gpd\to \Gpd$   (see Sequence
		\eqref{Eqt:jetgroupoidexactseq});
		
		\item[2)] The   map $K:~ \mathfrak{X}_{\mathrm{mult}}^k(\Gpd )\to \coHg ^1(\jet \Gpd,\wedge^k A)$ sends   $\Pi= \overrightarrow{\Lambda}-\overleftarrow{\Lambda}+\Pi_{\mathcal{F}}$ to $P^*[\mathcal{F}]$.  Here $P^*: \coHg ^1(\Gpd,\wedge^k \ker \rho)\to \coHg ^1(\jet \Gpd,
		\wedge^k A)$ is the embedding map in Lemma \ref{lleftHjet2}.
	\end{itemize}
\end{proposition}


The proof  of  these statements   is quite simple and omitted.


\subsection{Embedding of ${\Reduced}^k_{\mathrm{mult}}$ into $\coHg ^1(\jet \Gpd,\wedge^k A) $}

 Let $(c,\pi)$ be a  $k$-characteristic pair on $\Gpd$. By Proposition \ref{Prop:Mpitomupi}, we can examine that    $(\hat{c},\pi)$ is a $k$-characteristic pair on $A$, the tangent Lie algebroid of $\Gpd$. Here $\hat{c}\in Z^1(\jet A,\wedge^k A)$ is the infinitesimal of $c\in Z^1(\jet\Gpd,\wedge^k A)$.
By this fact, the assumption that $1\leqslant  k\leqslant  \Top $, and Lemma \ref{luniquenessofpi}, we immediately get
\begin{lemma}\label{Cor:uniquenessofpi2}If $(c,\pi)$ and $(c,\pi')$ are both $k$-characteristic pairs on $\Gpd$. Then we have $\pi=\pi'$.\end{lemma}

The following fact describes  equivalence relations of multiplicative multivectors of a groupoid in terms of   characteristic pairs.

\begin{proposition}\label{Prop:groupoidgaugeiff} Let $\Pi=(c,\pi) $ and $\Pi'=(c',\pi')$ be in $\mathfrak{X}_{\mathrm{mult}}^k(\Gpd )$.   Then $\Pi\gaugeequiv \Pi'$ holds if and only if $[c]=[c']$ holds in $\coHg ^1(\jet \Gpd,\wedge^k A)$.
\end{proposition}
\begin{proof}
The implication   ``$\Pi\gaugeequiv \Pi'$''
	$\Rightarrow$ ``$[c]=[c']$'' is easily derived from Example  \ref{Prop:exactmultiplicativecharpairs}.

We show the converse implication ``$[c]=[c']$'' $\Rightarrow$ ``$\Pi\gaugeequiv \Pi'$''. In fact, if $c'=c-d_{\jet \Gpd}\tau$ for some $\tau\in \Gamma(\wedge^k A)$, then the $k$-characteristic pairs $(c',\pi')$ and $(c-d_{\jet \Gpd}\tau,\pi+\rhopush\tau)$ share the same first entry. By Lemma  \ref{Cor:uniquenessofpi2}, we have $\pi'=\pi+\rhopush\tau$. It follows that $\Pi'=\Pi+  \overrightarrow{\tau }-\overleftarrow{\tau }$, as required.
\end{proof}

We now state the main result about the reduced space $ {\Reduced}^k_{\mathrm{mult}}$ of multiplicative multivectors on $\Gpd$, which is   parallel to the statements in Theorem  \ref{Thm:Rdiffaspullback} about
the reduced space $ {\Reduced}^k_{\mathrm{diff}}$ of Lie algebroid differentials.
Let us first introduce or recall some maps:
\begin{enumerate}
	\item Define:  $$K:~\mathfrak{X}_{\mathrm{mult}}^k(\Gpd ) \to \coHg ^1(\jet \Gpd,\wedge^k A),\qquad   \Pi = (c ,\pi) \mapsto [c ] .
	$$
 	By Proposition \ref{Prop:groupoidgaugeiff}, the map $K$ induces an embedding
	$$
	K:~ {\Reduced}^k_{\mathrm{mult}} \quad\hookrightarrow\quad \coHg ^1(\jet \Gpd,\wedge^k A).
	$$
	\item Define ${I^!}:~\mathfrak{X}_{\mathrm{mult}}^k(\Gpd )  \to   \ReducedProper^k$ by sending $\Pi = (c ,\pi) \mapsto [\pi ]$. It naturally induces a map ${I^!}:~\Reduced^k_{\mathrm{mult}}  \to   \ReducedProper^k$.
	\item The inclusion $I\colon \heat   {\longrightarrow}\jet\Gpd$ (see Sequence \eqref{Eqt:jetgroupoidexactseq}) naturally induces a map \[{I^*}\colon\coHg ^1(\jet \Gpd,\wedge^k A)   {\longrightarrow}  \coHg ^1(\heat,\wedge^k A).\]
	\item Recall Lemma \ref{lUembedding} where we defined an embedding   $U:~\ReducedProper^k\hookrightarrow \coHg ^1(\heat,\wedge^k A)$.
	
\end{enumerate}

\begin{theorem}\label{Thm:Rmultaspullback}Let  $1\leqslant  k\leqslant  \Top $ be given. 	
		  The reduced space of multiplicative  $k$-vectors $\Reduced^k_{\mathrm{mult}}$ on the Lie groupoid $\Gpd$ is the pullback      of maps $I^*$ and $U$. In other words, the diagram
	\begin{equation*}
\xymatrix{  \Reduced^k_{\mathrm{mult}} \ar@{^{(}->}[d]_{K} \ar[r]^{I^!} & \ReducedProper^k \ar@{^{(}->}[d]^{U} \\
	\coHg ^1(\jet \Gpd,\wedge^k A) \ar[r]^{I^*} & \coHg ^1(\heat,\wedge^k A)  }
	\end{equation*}
	is commutative and the map
	\begin{eqnarray*}
		{\Theta}:~	\Reduced^k_{\mathrm{mult}} &  {\longrightarrow} &{{\coHg^1(\jet \Gpd,\wedge^k A)}{ \ _{I^*}\times_U } \ReducedProper^k =  \set{(X,Y)\in \coHg ^1(\jet \Gpd,\wedge^k A)  \times   \ReducedProper^k; ~ I^* (X)=U (Y)~ }} \\
		{[\Pi=(c,\pi)]}  & \mapsto & (K[\Pi],I^![\Pi])=([c],[\pi])
	\end{eqnarray*}
	is an isomorphism.
\end{theorem}
This theorem  tells us that $\Reduced^k_{\mathrm{mult}}\cong \Img K=(I^*)^{-1}\Img U$.
We omit the proof as it is an easy adaptation of the arguments of Theorem  \ref{Thm:Rdiffaspullback}. Of course, one is able to draw a  bigger commutative diagram:
 \begin{equation*}\label{Diag:fullpictureRmult}
 \xymatrix{
 	0   \ar[r]^{ } & \coHg ^1(  \Gpd,\wedge^k \ker\rho ) \ar[d]_{\mathrm{id}} \ar[r]^{P^!} & \Reduced^k_{\mathrm{mult}} \ar@{^{(}->}[d]_{K} \ar[r]^{I^!} & \ReducedProper^k \ar@{^{(}->}[d]^{U} \\
 	0 \ar[r]^{} & \coHg ^1(  \Gpd,\wedge^k \ker\rho ) \ar[r]^{P^*} & \coHg ^1(\jet \Gpd,\wedge^k A) \ar[r]^{I^*} & \coHg ^1(\heat,\wedge^k A)  ,}
 \end{equation*}
 where the two horizontal sequences are exact, and the vertical arrows are injective.  Here the lower horizontal sequence is the one described by Lemma \ref{lleftHjet2} in the appendix. So, we can simply treat $\Reduced^k_{\mathrm{mult}}\cong \coHg ^1(  \Gpd,\wedge^k \ker\rho )\oplus (\Img U\cap \Img I^*)$ although it is not canonical.

\subsection{Connection between $\Reduced^\bullet_{\mathrm{mult}}$ and $\Reduced^\bullet_{\mathrm{diff}}$ }
From a very natural   point of view,   
  the space $\Reduced^\bullet_{\mathrm{diff}}$ should be the infinitesimal counterpart of $\Reduced^\bullet_{\mathrm{mult}}$.  Indeed, the map $\delta_{(\cdot)}\colon \mathfrak{X}_{\mathrm{mult}}^\bullet(\Gpd )\to \differentials^\bullet(A)$ defined by Equation \eqref{Eqt:deltamap} sends exact multiplicative multivectors to exact differentials, and thus induces a map,  denoted by $$\bar{\delta} :  \quad \Reduced^\bullet_{\mathrm{mult}} \to \Reduced^\bullet_{\mathrm{diff}} ,$$ which is also a morphism of graded Lie algebras.

 In this part, we make this point more explicit by
discussing the relation  between    characteristic pairs of Lie algebroid differentials on $A$ and characteristic pairs of multiplicative multivectors on   $\Gpd$. We also characterize the kernel of $\bar{\delta}$. 
Our  first observation is the   following proposition.
\begin{proposition}\label{Thm:PitodeltaPi}If $(c,\pi)$ is  the   $k$-characteristic pair of a multiplicative $k$-vector $\Pi\in \mathfrak{X}_{\mathrm{mult}}^k(\Gpd )$,  then the
	$k$-characteristic pair of the
	differential $\delta_\Pi\in \differentials^k(A)$  is given by
	$(\hat{c},\pi)$.
	Here $\hat{c}\in Z^1(\jet A , \wedge^k A) $ is the infinitesimal of the groupoid $1$-cocycle $c\in Z^1(\jet \Gpd , \wedge^k A) $.
\end{proposition}
\begin{proof}  By Proposition \ref{Prop:kdifferentialinpair}, it amounts to prove the following two identities:
	\begin{eqnarray*}\nonumber
	\deltazero_\Pi(u)&=&  \hat{c}(\liftingd {u}),\quad u\in \Gamma(A),
\\	\deltaone_\Pi(f)&=& (-1)^{k-1}\iota_{ df}\pi,\quad f\in  {\CinfM}.\end{eqnarray*}
		The first one is verified as follows.  Fix a point $x\in M$. Consider a curve
	$$
	\gamma(\epsilon)=[(\exp \epsilon u)_{(\exp \epsilon u)(x)}]
	$$
	in $\jet \Gpd$. It lands  in the $s$-fibre of $\jet \Gpd$ that passes through $x\in M$. Clearly, one has $\gamma'(0)=(\liftingd u)_x\in (\jet A)_x$. We then compute that
	\begin{eqnarray*}
		\deltazero_\Pi(u)_x&=&\overrightarrow{\deltazero_{\Pi}(u)}_x=-[\overrightarrow{u},\Pi]_x\\
		&=&-\frac{d}{d\epsilon}|_{\epsilon=0} L_{\exp (-\epsilon u)} \Pi|_{(\exp \epsilon u)(x)}\\
		&=&-\frac{d}{d\epsilon}|_{\epsilon=0} L_{(\gamma(\epsilon))^{-1}} (R_{(\exp \epsilon u)(x)*}c(\gamma(\epsilon))+L_{\gamma(\epsilon)} (B\pi)_x), \quad\mbox{by Equation  ~\eqref{formula3};}\\
		&=& -\frac{d}{d\epsilon}|_{\epsilon=0}\Ad_{(\gamma(\epsilon))^{-1}} c(\gamma(\epsilon))\\
		&=&\hat{c}(\gamma'(0))= \hat{c}(\liftingd {u})_x,
	\end{eqnarray*}
	as desired. The second one is easy, as
	$$\deltaone_\Pi(f)=\overrightarrow{\deltaone_{\Pi}(f)}|_M=[\Pi,t^*f]|_M =(-1)^{k-1}\iota_{ df}\pi.
	$$
	The last step is due to Equation  ~\eqref{Eqt:Pitstarfbypi}.
\end{proof}

Combining our previous results, we obtain the following diagram:
\begin{equation}\label{Diagram}
\xymatrix{
	\mathfrak{X}_{\mathrm{mult}}^k(\Gpd ) \ni & \Pi= (c,\pi) \ar[d]_{ }     \ar[r]^{\qquad K } & [c]\ar[d]^{\mathrm{VE}_2}&\in \coHg ^1(\jet\Gpd,\wedge^k A)  \\
	\differentials^k(A)\ni &\delta_\Pi  = (\hat{c},\pi) \ar[r]^{\qquad \kappa} & [\hat{c}]&\in  \coHg ^1(\jet A,\wedge^k A) .}
\end{equation}
Here the left vertical arrow  refers to the method of taking  infinitesimal, and the right one refers to the Van Est map (see Diagram \eqref{Diag:fullpicturejetH1vanEast}).


We finally state a summary diagram which connects our  results in Theorems \ref{Thm:Rdiffaspullback} and \ref{Thm:Rmultaspullback}. The proof is skipped as it is a straightforward verification.
\begin{proposition}\label{Prop:3Ddiagram} The following diagram is commutative:
	\begin{equation*}
		\xymatrixcolsep{5pc}\xymatrix@!0{
			&\coHg ^1(A,\wedge^k \ker\rho) \ar@{=}[dd]|\hole  \ar@{^{(}->}[rr]^{p^!}&  &     R_{\mathrm{diff}}^k \ar[rr]^{i^!} \ar@{^{(}->}[dd]|\hole ^<<<{\kappa}  &  & R_\rho^k \ar@{^{(}->}[dd]^<<<{\mu}      \\
		\coHg ^1(\Gpd,\wedge^k \ker\rho) \ar[ur]^<{\mathrm{VE}_1}\ar@{^{(}->}[rr]^{\qquad\quad P^!}\ar@{=}[dd] &  & R_{\mathrm{mult}}^k \ar[ur]^{\bar{\delta}}\ar[rr]^{\quad \quad \quad  \quad I^!}\ar@{^{(}->}[dd]^<<<{K}  &  & R_\rho^k \ar@{=}[ur] \ar@{^{(}->}[dd]^<<<{U} \\
			&  \coHg ^1(A,\wedge^k \ker\rho) \ar@{^{(}->}[rr]|\hole^{p^*\qquad} &  & \coHg^1(\jet A,\wedge^k A) \ar'[r][rr]^{i^*\quad\quad}&  & \coHg^1(\mathfrak{h}, \wedge^k A)  \\
			\coHg ^1(\Gpd,\wedge^k \ker\rho)  \ar@{^{(}->}[rr]^{P^*}  \ar[ur]^<{\mathrm{VE}_1} &   &	\coHg^1(\jet \Gpd,\wedge^k A) \ar[ur]_{}   \ar[rr]^{I^*}   \ar[ur]^<{\mathrm{VE}_2} &   &	\coHg^1(\mathcal{H}, \wedge^k A) \ar[ur]_{\cong}      }
	\end{equation*}
	  Consequently, we have 
	 $$
	 \ker\bar{\delta}\cong \ker\mathrm{VE}_2  \cong \ker\mathrm{VE}_1\,.
	 $$
	Here $\mathrm{VE}_1 $ and $\mathrm{VE}_2$ are the Van Est maps (cf. Lemma \ref{lleftHjet2}).
\end{proposition}

If $\mathcal{G}$ is  $s$-connected and simply connected, then $\mathrm{VE}_1$ is an isomorphism (see \cite{WX, Crainic2003}), and  so  is $\bar{\delta}$ (due to    the universal lifting theorem \cite{ILX} which we recalled in Section \ref{ULT} and the map $\delta_{(\cdot)}: \mathfrak{X}^\bullet_{\mathrm{mult}}(\Gpd)\to \differentials^\bullet(A)$ is an isomorphism of graded Lie algebras).  


\section{Exceptional cases}\label{Sec:twoexceptionalcases}
\subsection{The exceptional case   $k=0$}\label{Sec:exceptionalscases0}
 By Theorem \ref{Thm:gaugespacekdifferential1}, the  reduced  space of   $0$-differentials ${\Reduced}^0_{\mathrm{diff}}$ on $A$ coincides with the degree $(1,-1)$ Lie algebroid
deformation cohomology
$ \coHg _{\mathrm{Def}}^{1, -1}(A)$.
Also recall  that $(\Der^{\bullet,-1}(A),\partial)$ coincides with the standard Chevalley-Eilenberg complex $(\OmegaA,\dA)$, as already mentioned in Remark \ref{Remk:specialdeformationcomplexes}. Therefore, we have the following fact.
\begin{proposition}\label{Prop:0differentialclass}
	The reduced  space ${\Reduced}^0_{\mathrm{diff}}$ of $0$-differentials  on   $A$ is isomorphic to $\coHg ^1(A,\mathbb{R})$, the degree $1$ Chevalley-Eilenberg cohomology of the Lie algebroid $A$.
\end{proposition}

On the Lie groupoid level, by definition, a {\bf multiplicative $0$-vector} on   $\Gpd $ is a multiplicative function $F\in C^\infty(\Gpd )$, i.e.,
\[F(gr)=F(g)+F(r),\qquad \forall~ (g,r)\in \Gpd ^{(2)}.\]
 In other words, the map $F:~\Gpd \to M\times \mathbb{R},~~ g\mapsto (t(g),F(g)) $
is a groupoid $1$-cocycle (with respect to the trivial action of $\Gpd $ on $M\times \mathbb{R}$).
An exact multiplicative $0$-vector is a multiplicative function of the form $$F=\overrightarrow{f }-
\overleftarrow{f }=t^*f-s^*f=-d_\Gpd f,$$ for some $f\in {\CinfM}$. So the following fact is clear.
\begin{proposition}\label{Prop:classifyk=0}The reduced  space ${\Reduced}^0_{\mathrm{mult}}$ of multiplicative $0$-vectors on a Lie groupoid $\Gpd $ is isomorphic to $\coHg ^1(\Gpd ,  \mathbb{R})$, the degree $1$ cohomology of $\Gpd $ with respect to the trivial action of $\Gpd$ on $M\times \mathbb{R}$.
\end{proposition}

 \subsection{The exceptional case  $k=\Top +1$}\label{Sec:exceptionalscasestop}

 \subsubsection{$(\Top  +1)$-differentials on $A$}
First,
 we  define two subsets in $M$:
\begin{equation}\label{Eqt:B}
B=\{x\in M |\rho_x= 0\}\qquad\mbox{and}\quad B^c= \{x\in M |\rho_x\neq  0\},
\end{equation}
and we observe the following basic fact.
\begin{lemma}\label{pi=0}
	Let  $\pi$ be in $\Gamma(TM\otimes (\wedge^{\Top} A))$. Then it is a $\rho$-compatible $(\Top+1)$-tensor if and only if $\pi$ vanishes at those $x\in B^c$.\end{lemma}

\begin{proof} It suffices to prove that         $\pi_x=0$ for all $x\in B^c$ if $\pi$ is $\rho$-compatible. In fact, as  $ \rho_x\neq 0$, we have 	  $p:=\mathrm{dim} (\Img \rho_x)\geqslant  1$. Then we are able to find  a  basis $\{u_1,\cdots,u_{\Top}\}$ of $ A_x$ such that $\{\rho (u_1),\cdots \rho (u_p)\}$ spans $\Img \rho_x$ and $\rho (u_{p+1})=\cdots=\rho (u_{\Top})=0$.
	Suppose further
	that
	$T_x M$ is spanned by a basis $\{ \rho (u_1),\cdots,\rho (u_p),N_1,\cdots,N_q\}$, where  $p+q=\mathrm{dim} M$. Let $\{ \xi_1,\cdots,\xi_p,\eta_1,\cdots,\eta_q\}$ be the associated dual basis of $T_x^*M$.
	Then we can write $\pi_x$ in the following form:
	\[\pi_x=(\sum_{i=1}^p a^i \rho (u_i)+\sum_{j=1}^qb^j N_j)\otimes (u_1\wedge\cdots\wedge u_{\Top}),\qquad \mbox{for some}\quad a^i,b^j\in\mathbb{R}.\]
	
	Now we examine the $\rho$-compatible condition of  $\pi_x$ at $x$:
	\begin{equation}
	\label{Eqt:atxproper}\iota_{\rho^*\eta}\iota_\xi \pi_x=-\iota_{\rho^*\xi}\iota_\eta \pi_x,\qquad \forall~ \xi,\eta\in T^*_x M .
	\end{equation}
	Substituting  $\xi=\eta=\xi_i\in T_x^*M$ ($1\leqslant  i\leqslant  p$) in Equation   \eqref{Eqt:atxproper}, we have
	\[0=\iota_{\rho^*\xi_i}\iota_{\xi_i}\pi=(-1)^{i-1}a^i u_1\wedge \cdots\wedge\hat{u_i}\wedge\cdots \wedge u_p\wedge\cdots\wedge u_{\Top}.\]
	We thus get $a^i=0$, for all  $i=1$, $\cdots$, $p$. Consequently, we have
	\[\iota_{\eta_j} \pi_x=b^j u_1\wedge\cdots\wedge u_{\Top},\]
	and
	\[\iota_{\rho^*\xi_i}\iota_{\eta_j} \pi_x=(-1)^{i-1}b^j u_1\wedge \cdots\hat{u_i}\wedge\cdots \wedge u_p\wedge\cdots\wedge u_{\Top}=-\iota_{\rho^*\eta_j}\iota_{\xi_i} \pi_x=0.\]
	Thus $b^j =0$, for all  $j=1$, $\cdots$, $q$. This proves that $\pi_x=0$.
\end{proof}

Now let us study a $(\Top+1)$-differential $\delta=(\deltazero,\deltaone)$ on $A$. Certainly $  \deltazero=0$ and $$\deltaone(f)=(-1)^{\Top}\iota_{df} \pi,$$ for some $\rho$-compatible tensor $\pi\in \Gamma(TM\otimes (\wedge^{\Top} A))$. Moreover, Conditions  ~\eqref{Eqt:kdifferential2.5} $\sim$ \eqref{Eqt:kdifferential4} become a single one equation
\begin{equation}\label{Eqt:delta0single}
\deltaone[u,f]=[u,\deltaone(f)],\qquad \forall~ u\in \Gamma(A ),f\in {\CinfM}.
\end{equation}

 By Lemma \ref{pi=0}, if $\deltaone_x$ is   nontrivial, then  $x$ must be in $B\subset M$.
Note that $A|_B$ is a bundle of Lie algebras (not necessarily of the same   type of Lie algebras). So,   Equation   \eqref{Eqt:delta0single} can be characterized by
\begin{equation}\label{Eqt:delta0single2}
0=[u,\deltaone(f)],\qquad  \forall~  u\in  A_x, ~x\in B,~f\in {\CinfM}.
\end{equation}

\begin{lemma}\label{Lemma:adLu} If $\deltaone_x\neq 0$ on some $x\in B$, then Equation  ~\eqref{Eqt:delta0single2} holds if and only if $\mathrm{tr} \ad_u=0$, for all $u\in A_x$.
	 \end{lemma}
 \begin{proof} Choose a basis $\{u_1,\cdots,u_{\Top}\}$ of $ A_x$  and assume that
 	$\pi_x=X\otimes (u_1\wedge\cdots\wedge u_{\Top})$, {for some nonzero } $X\in T_xM$.
 	The statement of this lemma is justified by calculating  the right hand side of Equation  ~\eqref{Eqt:delta0single2}:
 	\begin{eqnarray*}
 		[u,\deltaone(f)]&=&(-1)^{\Top}[u,\iota_{df} \pi_x]\\ &=&(-1)^{\Top}X(f)[u,u_1\wedge  \cdots\wedge u_{\Top} ]\\ &=&
 		(-1)^{\Top}X(f)\sum_{i=1}^{\Top} u_1\wedge\cdots\wedge [u,u_i]\wedge \cdots \wedge u_{\Top}\\ &=& (-1)^{\Top}X(f)
 		\mathrm{tr} (\ad_u)u_1\wedge\cdots \wedge u_{\Top}.
 	\end{eqnarray*}
 \end{proof}

\begin{remark}
	For the Lie algebra $A_x$ ($x\in B$), the map $A_x\to \mathbb{R}$, $u\mapsto \mathrm{tr} (\ad_u)$ is called the {\bf adjoint character} of $A_x$ (\cite{EL}).
\end{remark}

Combining Lemmas \ref{pi=0} and \ref{Lemma:adLu}, we have a conclusion:
\begin{theorem}\label{Thm:top+1differentialpi}   Every $  (\Top+1)$-differential     $\delta=(\deltazero=0,\deltaone)$ on $A$ corresponds to  a section $\pi\in \Gamma(TM\otimes (\wedge^{\Top} A))$ which vanishes on  the sets  $B^c$ and
	$$  \{x\in B|  \mbox{ the adjoint character of } A_x \mbox{ is nontrivial}\}, $$
	such that $\deltaone(f)=(-1)^{\Top}\iota_{df} \pi$, for all $f\in\CinfM$. Conversely, any  such $\pi$ determines a  $(\Top+1)$-differential $\delta$ in this way.
\end{theorem}
	Clearly, exact $(\Top+1)$-differentials  on $A$ are trivial. Hence, the reduced  space of $(\Top+1)$-differentials can be identified with the set of $\pi$ described in the above theorem.

\subsubsection{Multiplicative $(\Top +1)$-vectors on Lie groupoids}

Let $\Pi$ be a multiplicative $(\Top+1)$-vector  on a Lie groupoid $\Gpd $.
Denote by $\pi $ the $\Gamma(TM\otimes (\wedge^\Top A))$-component of $\Pi|_M$. Then $\pi$ is a $\rho$-compatible tensor by Proposition \ref{CSXproppi}, and determines the $(\Top+1)$-differential  $$\delta_{\Pi}=(\delta_{\Pi}^{(0)}=0,\delta_{\Pi}^{(1)}).$$   Moreover, $\pi$ is subject to the condition in Theorem  \ref{Thm:top+1differentialpi}.

\begin{lemma}\label{lPig=01}
	 If  $g\in \Gpd$ satisfies $t(g)\in B^c$, where $B^c$ is defined by Equation \eqref{Eqt:B}, then we have $\Pi_g=0$. Similarly, if  $g\in \Gpd$ satisfies $s(g)\in B^c$, then we have $\Pi_g=0$.
\end{lemma}
\begin{proof}	
	By Proposition \ref{CSXproppi}, for any $f\in \CinfM$, we have
	\begin{eqnarray*}
	\iota_{t^*(df)} \Pi_g &=& \overrightarrow{\iota_{df} \pi}|_g
	 =  R_{g*}(\iota_{df} \pi|_{t(g)}).
	\end{eqnarray*}
	If $t(g)\in B^c$, then one has $\iota_{t^*(df)} \Pi_g=0$, and hence $\Pi_g\in \wedge^{\Top+1} \ker t_{*g}$. By dimensional reasons, we get $\Pi_g=0$.
\end{proof}

In what follows, we assume that the Lie groupoid $\Gpd$ is $s$-connected. It is easy to prove that, for $g\in\Gpd$, the following three conditions are equivalent:
	
\qquad	 (1)~\quad    $s(g)= t(g)$; \qquad (2)~\quad   $s(g)\in B$; \qquad   (3)~\quad    $t(g)\in B$.

Let us denote by $\Gpd|_B$ the collection of $g\in \Gpd$ subject to any of the three properties described above. Clearly, $\Gpd|_B$ is a bundle of Lie groups over $B$. Denote by $$\mathrm{pr}:~\Gpd|_B\to B$$ the projection, which coincides with $s$ and $t$. We also denote by $G_x$ the fibre of $\Gpd|_B$ at $x\in B$.

The Lie algebroid of $\Gpd|_B$ is identically $A|_B$, a bundle of Lie algebras.  By Theorem \ref{Thm:top+1differentialpi} and   Lemma \ref{lPig=01}, one has
\begin{lemma}\label{lPig=02}
	If $x\in B$ and the adjoint character of   $A_x$  { is nontrivial}, then we have $\Pi_g=0$ for all $g\in G_x$.
\end{lemma}

 Of course, we concern where $\Pi$ could be nontrivial. This is answered by the following assertion.
 \begin{proposition}
    If $\Pi_g\neq 0$,  then $g$ must be in $\Gpd|_B$ and $\mathrm{det}(\Ad_g)=1$, where $\Ad_g$ is considered as an automorphism on $A_x$, $x=\mathrm{pr}(g)  \in B$;
 	Moreover, if $\pi_x=X\otimes W$ for $X\in T_xM$, $W\in \wedge^\Top A_x$, then one has
 		\begin{equation}
 		\label{Eqt:Pigandpix}
 		\Pi_g= \tilde{X}\wedge \overrightarrow{ W}_g=\tilde{X}\wedge \overleftarrow{ W}_g\,,
 		\end{equation}
 		where $\tilde{X}\in T_g\Gpd$ satisfies $\mathrm{pr}_*(\tilde{X})=X$, and $\overrightarrow{W}$ is the right-invariant $\Top$-vector (on $G_x$), which coincides with the left-invariant $\Top$-vector  $\overleftarrow{W}$. 	
\end{proposition}
\begin{proof} By Lemmas \ref{lPig=01} and \ref{lPig=02}, the $\Pi_g\neq 0$ situation only happens at those $g\in G_x$ with $\mathrm{tr} \ad_u=0$, for all $u\in A_x$. As the fibre $G_x$ of $G|_B$ is connected, we have $\mathrm{det}(\Ad_g)=1$, for all $g\in G_x$.
	This implies that $\overrightarrow{W}=\overleftarrow{W}$, for any $W\in  \wedge^\Top A_x$.
	
	Also, $\iota_{\mathrm{pr}^*(\xi)} \Pi   $ should be a left-invariant (and right-invariant as well) $\Top$-vector on $G_x$, for $\xi\in T_xM$, and hence one is able to write $$
	\Pi_g= \tilde{X}\wedge \overrightarrow{ W}_g=\tilde{X}\wedge \overleftarrow{ W}_g\,,
	$$
	for some $\tilde{X}\in T_g\Gpd$. Then, by
	\begin{eqnarray*}
	(\mathrm{pr}_*\tilde{X})(f)\overrightarrow{ W}_g	=\iota_{t^*(df)} \Pi_g &=& \overrightarrow{\iota_{df} \pi}|_g \\
		&=&  R_{g*}(\iota_{df} \pi|_{t(g)})=X(f)\overrightarrow{ W}_g\,,
	\end{eqnarray*}
we see that $\mathrm{pr}_*(\tilde{X})=X$.	
\end{proof}

We have a converse proposition:
\begin{proposition}If   $\pi\in \Gamma(TM\otimes (\wedge^{\Top} A))$ satisfies the same condition as in Theorem  \ref{Thm:top+1differentialpi}, then via Equation  ~\eqref{Eqt:Pigandpix}, it determines a 	
	multiplicative $(\Top+1)$-vector $\Pi$ on $\Gpd $ which vanishes on any $g\in (\Gpd|_B)^c$ and $g\in \Gpd|_B$ with $\mathrm{det}(\Ad_g)\neq 1$.
	
\end{proposition}
\begin{proof}
	One should examine the conditions declared by Lemma \ref{multi vector}. In fact, conditions (2) and (3) are easy, while the affine condition (1) can be directly verified by any of the three equivalent conditions in Definition \ref{Defn:affine3conditions}.
\end{proof}

In summary, we get the description of $\Reduced_{\mathrm{mult}}^{\Top+1}=\mathfrak{X}_{\mathrm{mult}}^{\Top+1}(\Gpd )$:
\begin{theorem}\label{Theorem:classifyk=top+1}
	Multiplicative $(\Top+1)$-vectors  $\Pi$ on  the Lie groupoid $\Gpd $ are in one-to-one correspondence to elements $\pi\in \Gamma(TM\otimes (\wedge^{\Top} A))$ that satisfy the same condition as in Theorem  \ref{Thm:top+1differentialpi}.
\end{theorem}

\appendix

\section{Lie algebroid and groupoid modules}\label{Sec:LALGPbasic}

We recall some  basic knowledge   of Lie groupoids and Lie algebroids. For more details on this subject, see \cite{Mackenzie}.	A  {Lie  algebroid} is a $\mathbb{R}$-vector bundle $A \rightarrow M$, whose space of sections is endowed with a Lie bracket $[\cdot,\cdot]$, together with a bundle map $\rho:  ~ A \rightarrow T M $ called {\bf anchor} such that
	$\rho:~~ \Gamma(A) \rightarrow \XX(M) $ is a morphism of Lie algebras and
	\[
	[u, \, fv]=f[u, \,v]+(\rho(u)f)v,
	\]
	holds for all $u, v \in \Gamma(A)$ and $f\in  C^\infty({M})$.

By  an {\bf $A$-module}, we mean a vector bundle $E\to M$ which is endowed with an $A$-connection:
$$
\nabla:~\Gamma(A)\times  \Gamma(E)\to \Gamma(E)
$$
which is flat:
$$
\nabla_{[u,v]}e=\nabla_u\nabla_v e-\nabla_v\nabla_u e,\qquad\forall~ u,v\in \Gamma(A), e\in \Gamma(E).
$$

Given an $A$-module $E$, we have the standard Chevalley-Eilenberg complex $(C^\bullet(A,E),d_A)$, where
$$
C^\bullet (A,E):=\Gamma(\Hom(\wedge^\bullet  A , E )), \quad 0\leqslant  \bullet \leqslant  \Top
$$
and the coboundary operator  $d_A: ~C^{n} (A,E)\to C^{n+1} (A,E)$ is given by
\begin{eqnarray*}
	&&(d_A\lambda)(u_0,u_1,\cdots, u_{n})\\
	&=& \sum_{i}\minuspower{i}\nabla_{u_i} \lambda(\cdots, \hat{u}_i, \cdots)
	+\sum_{i<j}\minuspower{i+j}\lambda([u_i,u_j],\cdots, \hat{u}_i, \cdots, \hat{u}_j, \cdots),
\end{eqnarray*}
for all $u_0,\cdots, u_n\in\Gamma(A)$.  

By
a {\bf Lie algebroid $1$-cocycle} valued in the $A$-module $E$, we mean a $d_A$-closed element $\lambda\in C^1(A,E)$, namely a vector bundle map $\lambda:~ A\to E$  satisfying the following condition:
\begin{equation*}\label{Eqt:A1cocyclecondition}
\lambda[u,v]=\nabla_{u} \lambda(v)-\nabla_{v} \lambda(u),\quad \forall~ u,v\in \Gamma(A).
\end{equation*}
The collection of such Lie algebroid $1$-cocycles is denoted by $Z^1(A,E)$. The subspace of coboundaries $ B^1(A,E)\subset  Z^1(A,E) $ consists of elements of the form $d_A\nu$ where $\nu\in C^0(A,E)=\Gamma(E)$, i.e.,
\begin{equation}\label{Eqt:dnuforalgebroid}
(d_A\nu)(u)= \nabla_u \nu,\qquad\forall~ u\in \Gamma(A).
\end{equation}
We denote by $\coHg ^1(A,E)=Z^1(A,E)/B^1(A,E)$ the first cohomology of $A$ valued in the $A$-module $E$.

   Let $\Gpd $  be a Lie groupoid over a smooth manifold $M$ whose   source and target maps
are denoted by $s$ and  $t$, respectively. We treat the set of identities $M\hookrightarrow \Gpd $   as a submanifold of $\Gpd $. The groupoid multiplication of two elements $g$ and ${r}$ is denoted by $g{r}$, if $s(g)=t({r})$. The collection of such pairs   $(g,{r})$,  called composable pairs, is denoted by $\Gpd ^{(2)}$.
The groupoid inverse map of $\Gpd $ is denoted by $\mathrm{inv}: ~\Gpd \rightarrow \Gpd $. For $g\in \Gpd $, its inverse $\mathrm{inv}(g)$ is also denoted by $g^{-1}$.

A {\bf bisection} of $\Gpd $ is a smooth splitting $b: M\to \Gpd $ of the source map $s$ (i.e., $s{\littlecirc} b=\mathrm{id}_M$) such  that $$\phi_{b}:=t{\littlecirc} b:\quad M\to M$$ is a diffeomorphism. The set of bisections of $\Gpd $ forms a group which we denote by $\mathrm{Bis}(\Gpd )$ whose  identity element is   $\mathrm{id}_M$.  The multiplication $bb'$ of $b$ and $b'\in  \mathrm{Bis}(\Gpd )$ is given by
$$ (b  b')(x):=b(\phi_{b'}(x))b'(x),\qquad \forall~ x\in M.$$ The inverse $b^{-1}$ of $b$ is given by
$$ b^{-1}(x):=b(\phi_b^{-1}(x))^{-1},\qquad \forall~ x\in M.$$

A bisection $b$ defines a diffeomorphism of $\Gpd $ by left multiplication:
$$
L_b(g):=b(t(g))g,\qquad  \forall~ g\in \Gpd .
$$  We call $L_b$ the left translation by $b$.
Similarly, $b$ defines the right translation:
$$
R_b(g):=g b(\phi_b^{-1}s(g)),\qquad  \forall~ g\in \Gpd .
$$

Bisections that are exponentials of sections of $A$ are particularly useful. Given $u\in \Gamma({A})$, the corresponding right-invariant vector $\overrightarrow{u}\in \mathfrak{X}(\Gpd )$ generates a flow $\lambda_\epsilon:~\Gpd \to \Gpd $.  The exponential $\exp{\epsilon u}\in \mathrm{Bis}(\Gpd )$ (for $|\epsilon|$ sufficiently small) denotes  a bisection of $\Gpd$ such that
$$
\lambda_\epsilon (g)= L_{\exp{\epsilon u}} g,\qquad  \forall~ g\in \Gpd .
$$

We also need the notion of local bisections. A local bisection is a smooth splitting $b: U\to s^{-1}(U)$ of $s$, where $U\subset M$ is open, such that $$\phi_b=t{\littlecirc} b:\quad U\to \phi_b(U)\subset M$$ is a local diffeomorphism.

A local bisection $b$ passing through the point $b(x)=g$ on $\Gpd $, where $x=s(g)\in U$, is also denoted by $b_g$, to emphasise the particular point $g$.\emph{ In the sequel, by saying a   bisection through $g$, we mean a local bisection that passes through $g$.
}

A (left) {\bf $\Gpd $-module}   is a vector bundle $E\rightarrow M$  together with a smooth assignment: $g\mapsto \Phi_g$, where $g\in \Gpd $ and $\Phi_g $ is an isomorphism of vector spaces from $E_{s(g)}$ to $E_{t(g)}$
satisfying
\begin{enumerate}
	\item $\Phi_x=\mathrm{id}_{E_x}$, for all $x\in M$;
	\item $\Phi_{gr}= { \Phi_g{\littlecirc} \Phi_r}$, for all composable pairs $(g,r)$.
\end{enumerate}

Given a bisection $b\in \mathrm{Bis}(\Gpd )$, it induces an isomorphism of vector bundles $\Phi_b:~E\to E$ (over $\phi_b: ~M\to M$) defined by
$$
\Phi_b(e):=\Phi_{b(x)}e,\qquad\forall~ x\in M, e\in E_x.
$$

 We recall the tangent Lie algebroid $(A, [~,~],\rho) $ of the Lie groupoid $\Gpd $. The vector bundle $A$ consists of  tangent vectors on $\Gpd $ that are attached to the base manifold $M$ and tangent to $s$-fibres, i.e.,
$A=\ker s_*|_M$.   The anchor map $\rho :~A\rightarrow TM$ is simply $t_*$. For  $u,v\in \Gamma({A})$,
the Lie bracket $[u,v]$ is determined by
$$\overrightarrow{[u,v]}=[\overrightarrow{u},\overrightarrow{v}].$$
Here $\overrightarrow{u}$ denotes the right-invariant vector on $\Gpd $ corresponding to $u$. In the mean time, the left-invariant vector corresponding to $u$, denoted by $\overleftarrow{u}$, is related to $\overrightarrow{u}$ via
$$
\overleftarrow{u}= - \invstar  (\overrightarrow{u})=-\overleftarrow{\invstar  u}.
$$

A $\Gpd $-module $E$ is also an $A$-module, i.e., there is an $A$-action on $E$,
$$
\nabla:~\Gamma(A)\times  \Gamma(E)\to \Gamma(E)
$$
defined by
$$
\nabla_{u}e=  \frac{d}{d\epsilon }|_{\epsilon =0} \Phi_{\exp{(-\epsilon u)}}e,\qquad \forall~ u\in \Gamma(A), e\in \Gamma(E).
$$

An {\bf $n$-cochain} on $\Gpd $ valued in the $\Gpd $-module $E$ is a smooth map $c: \Gpd ^{(n)}\to E$ such that $c(g_1,\cdots,g_n)\in E_{t(g_1)}$, where $\Gpd ^{(n)}$ is the space of $n$-composable elements, i.e., $n$-arrays $(g_1,\cdots, g_n)$ satisfying $$s(g_1)=t(g_2),~s(g_2)=t(g_3),\cdots, ~s(g_{n-1})=t(g_n).$$

Denote by $C^n(\Gpd , E)$  the space of $n$-cochains.
The coboundary operator
\[ d_\Gpd: C^n(\Gpd , E)\to C^{n+1}(\Gpd , E),\]
is  standard:
\begin{itemize}
	\item[(1)] For $n=0$, $\nu\in C^0(\Gpd,E)=\Gamma(E)$, define
	$$
	(d_\Gpd \nu)(g)=\Phi_g \nu_{s(g)} -\nu_{t(g)},\qquad\forall~ g\in \Gpd;
	$$ 	\item[(2)] For  $c\in C^n(\Gpd , E)$ $(n\geqslant  1)$ on $\Gpd$, define
	\begin{eqnarray*}
		(d_\Gpd c)(g_0,g_1,\cdots,g_{n })&=&\Phi_{g_0} c(g_1,\cdots,g_{n })\\ &&+\sum_{i=0}^{n-1} (-1)^{i-1} c(g_0,\cdots,g_ig_{i+1},\cdots,g_{n})+(-1)^{n+1} c(g_0,\cdots,g_{n-1}).
	\end{eqnarray*}
\end{itemize}
In particular, a {\bf $1$-cocycle} is a map $c: \Gpd \to E$ such that $c(g)\in E_{t(g)}$ and \begin{equation}\label{Eqt:1cocyclecondition}
c(g{r})=c(g)+\Phi_{g} c({r}),\end{equation} for any composable pair $(g,{r})$.  The collection of such groupoid $1$-cocycles is denoted by $Z^1(\Gpd,E)$.  The subspace of coboundaries $ B^1(\Gpd,E)\subset  Z^1(\Gpd,E) $ consists of those of the form $d_\Gpd \nu$. We denote by $\coHg ^1(\Gpd,E)=Z^1(\Gpd,E)/B^1(\Gpd,E)$ the first cohomology of $\Gpd$ with coefficients in the $\Gpd$-module $E$.

A groupoid $1$-cocycle $c$ induces a Lie algebroid $1$-cocycle $\hat{c}: A\to E$   by the following formula.  For each $u\in A_x$, choose a smooth curve $\gamma(\epsilon)$  in   the $s$-fibre $s^{-1}(x)$ such that $\gamma'(0)=u$. Then   $\hat{c}(u)\in E_x$ is defined by
\begin{equation}\label{Eqt:hatcdefinition}
\hat{c}(u):=  -\frac{d}{d\epsilon }|_{\epsilon =0}~ \Phi^{-1}_{\gamma(\epsilon)} c(\gamma(\epsilon))= \frac{d}{d\epsilon }|_{\epsilon =0}~ c(\gamma(\epsilon)^{-1}).
\end{equation}
We call $\hat{c}$ the infinitesimal of $c$.
Note that our convention is slightly different from that in \cite{VE} (up to a minus sign).
Also, it is easily verified that, for $\nu\in \Gamma(E)$, we have
$$
( {d_\Gpd\nu}){\hat{}}=d_A\nu.
$$
Here on the left hand side, $d_\Gpd\nu$ belongs to $ B^1(\Gpd, E)$, while on the right hand side, $d_A\nu$ belongs to $B^1(A,E)$ (see Equation  ~\eqref{Eqt:dnuforalgebroid}).
 So, the map $c\mapsto \hat{c}$ induces a morphism $\coHg ^1(\Gpd,E)\stackrel{\mathrm{VE}}{\longrightarrow} \coHg ^1(A,E)$
 known as the Van Est map \cite{WX, Crainic2003}.



\section{The jet Lie algebroid}\label{Appendix:jetLiealgebroid} In this part and the next one, we recall the notions of jet Lie algebroid and jet Lie groupoid. More systematic knowledge can be found in    \cite{C}.
A (first) jet of the vector bundle $A\to M$ at $u_x\in A_x$ ($x\in M$), is a subspace in $T_{u_x} A$ which is isomorphic to $T_xM$ under the bundle projection.
The jet space of $A$ is denoted by $\jet A$. It is also a vector bundle over $M$ whose fibre at $x\in M$ consists of such jets.

Another approach  is to introduce an equivalence relation for sections $u_1$, $u_2\in \Gamma(A)$:
$$
u_1 \sim_x u_2  \iff u_1(x)=u_2(x) \mbox{ and } d_x\langle u_{1 },\zeta\rangle=d_x\langle u_{2 },\zeta\rangle,\quad \forall~ \zeta\in \Gamma(A^*).
$$
The jet of $u\in \Gamma(A)$ at $x$ is denoted by $[u_x]$, the equivalence class of the $\sim_x$ relation.
We note that $\jet A$ is also a Lie algebroid over $M$. 
Indeed, one has an exact sequence of Lie algebroids:
\begin{equation}\label{Eqt:jetalgebroidexactseq}
0\to\core  \stackrel{i}{\longrightarrow}\jet A\stackrel{p}{\longrightarrow} A\to 0.
\end{equation}
Here $\core=T^*M\otimes A=\mathrm{Hom}(TM,A)$ is a bundle of Lie algebras. At $x\in M$, the Lie bracket of $\lambda$ and $\mu\in \core_x=\mathrm{Hom}(T_xM,A_x)$, is given by
$$
[\lambda,\mu]=\mu{\littlecirc} \rho{\littlecirc}\lambda-\lambda{\littlecirc} \rho{\littlecirc} \mu.
$$
We will call $\core$ the bundle of \textbf{{isotropy} jet Lie algebras}.

We have a natural lifting map $\liftingd:\Gamma(A)\to \Gamma({\jet A})$ that sends $u\in \Gamma(A)$ to its jet $\liftingd u\in \Gamma({\jet A})$:
\begin{equation}\label{Eqt:liftingd}
(\liftingd u)_x=[u]_x,\qquad\forall~ x\in M.
\end{equation}
It can be easily seen that
$$
\liftingd(fu)=df\otimes u+ f \liftingd u,\quad\forall~ f\in \CIM.
$$
There is a decomposition of $\mathbb{R}$-vector spaces
$$
\Gamma({\jet A})=\liftingd\Gamma(A)~\oplus~\Gamma({T^*M\otimes A}).
$$

Moreover, the Lie bracket  of $\Gamma(\jet A)$ is determined by the following relations:
\begin{eqnarray}\label{Eqt:jetbracket1}
[\liftingd u_1,\liftingd u_2]&=&\liftingd [u_1,u_2],\\
\label{Eqt:jetbracket2}
[\liftingd u_1,df_2\otimes u_2]&=&d(\rho(u_1)f_2)\otimes u_2+df_2\otimes [u_1,u_2],\\
\label{Eqt:jetbracket3}
[df_1\otimes u_1,df_2\otimes u_2]&=&\rho(u_1)(f_2) df_1\otimes u_2-\rho(u_2)(f_1) df_2\otimes u_1\,.
\end{eqnarray}

It is simply to see that, Equation  ~\eqref{Eqt:jetbracket1} implies Equations  ~\eqref{Eqt:jetbracket2} and \eqref{Eqt:jetbracket3}. Note also that Equation  ~\eqref{Eqt:jetbracket3} coincides with the Lie bracket in $\core$.

The anchor of $\jet A$ is simply: \[\rho_{\jet A}(\liftingd u_1+df_2\otimes u_2)=\rho(u_1).\]

The vector bundle $A$  is a module of the jet algebroid $\jet A$ via adjoint actions:
\begin{eqnarray*}
	\ad_{\liftingd u}v&=&[u,v];\\
	\ad_{df\otimes u}v&=&-\rho(v)(f)u,
\end{eqnarray*}
where $u,v\in\Gamma(A)$, $f\in {\CinfM}$.
By standard derivation, $\jet A$ also acts on $\wedge^k A$, for $1\leqslant  k\leqslant  \Top $. Indeed, it can be verified that the adjoint action of $\jet A$ on $\wedge^k A$ is given by
\begin{equation}
\label{Eqn:jetadjoint2}
\ad_{\liftingd u}w=[u,w]\quad\mbox{ and }\quad
\ad_{df\otimes u}w = -[w,f]\wedge u,
\end{equation}
for all $w\in \Gamma(\wedge^k A)$.


The embedding $i: ~\core \to \jet A$ (in Sequence \eqref{Eqt:jetalgebroidexactseq}) of Lie algebroids induces   obvious  maps
$$i^*: ~C^1(\jet A,\wedge^k A)\to C^1(\core,\wedge^k A)$$
and
$$i^*: ~\coHg ^1(\jet A,\wedge^k A)\to \coHg ^1(\core,\wedge^k A).$$
The second $i^*$ is not necessarily surjective, and we will find its kernel. To this aim, consider $\ker\rho\subset A$, a bundle of Lie algebras (which might be singular). We have the adjoint action of $A$  on $\wedge^k \ker\rho$:~
$$
\ad_u \nu=[u,\nu],\qquad \forall~ u\in\Gamma(A),\nu\in\Gamma(\wedge^k \ker\rho).
$$
The projection map  $p: ~\jet A\to A$    (in Sequence
\eqref{Eqt:jetalgebroidexactseq}) also induces a map
\begin{eqnarray*}
	 C^1(A,\wedge^k \ker\rho)&\stackrel{p^*}{\longrightarrow}& C^1(\jet A,\wedge^k A) ,\\ \Omega&\mapsto& p^*\Omega \quad \mbox{ s.t. }\
p^*\Omega (U)=\Omega(p(U)),\qquad\forall~ U\in \Gamma(\jet A).
\end{eqnarray*}
It gives rise to $p^*:~\coHg ^1(  A,\wedge^k \ker\rho )\rightarrow \coHg ^1(\jet A,\wedge^k A)$. Moreover, the kernel of $i^*$ matches with the image of $p^*$:
\begin{lemma}\label{lleftHjet}For all $1\leqslant  k\leqslant  \Top $,
	the following sequence is   exact:~
	\begin{equation*}
	0\to\coHg ^1(  A,\wedge^k \ker\rho )\stackrel{p^*}{\longrightarrow} \coHg ^1(\jet A,\wedge^k A) \stackrel{i^*}{\longrightarrow} \coHg ^1(\core,\wedge^k A).
	\end{equation*}
\end{lemma}
\begin{proof}We first show that $\ker i^*=\Img p^*$.
	Assume that $c\in Z^1(\jet A,\wedge^k A)$ satisfies $i^*[c]= 0 $, i.e., $i^*c=d_\core \tau$ for $\tau\in \Gamma(\wedge^k A)$. We claim that there exists $\Omega\in Z^1(A,\wedge^k \ker \rho)$ such that $p^* \Omega=c$. In fact, by   $(c-d_{\jet A} \tau)|_\core=0$, we are able to write  $c-d_{\jet A} \tau=p^*\Omega$,
	for some $\Omega:~ A\to \wedge^k A$, as the sequence    \eqref{Eqt:jetalgebroidexactseq} is exact. We now show that $\Omega$ takes values in  $\wedge^k \ker \rho$, and it is a $1$-cocycle.
	
	In fact, by applying $d_{\jet A}(p^*\Omega)=0$ to elements $\liftingd u_1$ and $df_2\otimes u_2\in \Gamma(\jet A)$, we get
	\begin{eqnarray*}
		p^*\Omega([\liftingd u_1,df_2\otimes u_2])&=&\ad_{\liftingd u_1}(p^*\Omega(df_2\otimes u_2))-\ad_{df_2\otimes u_2} (p^*\Omega(\liftingd u_1))\\ &=&0-\ad_{df_2\otimes u_2}\Omega(u_1)\\ &=&[\Omega(u_1),f_2]\wedge u_2.
	\end{eqnarray*}
	As $[\liftingd u_1,df_2\otimes u_2]\in \Gamma(\core)$, we have $p^*\Omega([\liftingd u_1,df_2\otimes u_2])=0$. Therefore, the last term in the above equation must vanish, and it follows that  $[\Omega(u_1),f_2]=0$, i.e., $\Omega(u_1)\in \Gamma(\wedge^k \ker \rho)$.
	
	Applying $d_{\jet A}(p^*\Omega)=0$ to $\liftingd u_1$ and $\liftingd u_2\in \Gamma(\jet A)$, we get
	\[p^*\Omega([\liftingd u_1,\liftingd u_2])=\ad_{\liftingd u_1} p^*\Omega(\liftingd u_2)-\ad_{\liftingd u_2} p^*\Omega(\liftingd u_1)
	,\]or
	\[\Omega([u_1,u_2])=[u_1,\Omega(u_2)]-[u_2,\Omega(u_1)].\]
	This proves that  $\Omega\in Z^1(A,\wedge^k \ker \rho)$, as desired. We get that $\ker i^*\subset\Img p^*$. While $\Img p^*\subset \ker i^*$ is obvious, this proves the claim $\ker i^*=\Img p^*$.

	Second, we need to  prove that the arrow $\coHg ^1(  A,\wedge^k \ker\rho )\stackrel{p^*}{\longrightarrow} \coHg ^1(\jet A,\wedge^k A)$ is an injection.  This is easy and omitted.
\end{proof}

\section{The jet Lie groupoid}\label{Appendix:jetLiegroupoid}

Let $b_g,b'_{g}$ be two  bisections through $g\in \Gpd $. They are said to be equivalent at $g$ if $b_{*x}=b'_{*x}:~T_xM\to T_g\Gpd $, where $x=s(g)$. The equivalence class, denoted by $[b]_x$, or $[b_g]$, is called the first jet of   $b_g$.
The  {\bf (first) jet groupoid} $\jet \Gpd $ of a Lie groupoid $\Gpd $, consisting of such  jets  $[b_g]$, is a Lie groupoid over $M$.
The source and target  maps of $\jet \Gpd $ are given by \[s([b_g])=s(g), \qquad t([b_g] )=t(g).\] The multiplication is given by
\[
[b_{g}][b'_{r}] =[(b   b')_{g{r}}],\]
for two   bisections $b_g$ and $b'_r$ through $g$ and $r$ respectively. The inverse map is
\[ [b_g]^{-1}=[(b^{-1})_{g^{-1}}].\]

Given a   bisection $b_g$ through $x=s(g)$, its first jet induces a left translation $$L_{[b_g]}:~ T \Gpd |_{t^{-1}(x)} \to T \Gpd |_{t^{-1}(y)},$$ where $y=\phi_b(x)=t(g)$. Similarly we have a right translation $$R_{[b_g]}: T \Gpd |_{s^{-1}(y)} \to T \Gpd |_{s^{-1}(x)}.$$ Note that when  $L_{[b_g]}$ is restricted on $(\ker t_*)|_{t^{-1}(x)}$, it is exactly the left translation $L_{g*}$. Similarly, when  $R_{[b_g]}$ is restricted on $(\ker s_*)|_{s^{-1}(y)}$, it is exactly the right translation $R_{g*}$.
Moreover, the restriction of $L_{[b_g]}$ to $T_x M$  is exactly the tangent map $b_{*x}: T_xM\to T_g\Gpd$, i.e., $L_{[b_g]}(X)=b_{g*}(X)$ for $X\in T_x M$.

A bisection $b$ of $\Gpd $ acts on $\Gpd $ by conjugation
\[\mathrm{AD}_b (g)=R_{b^{-1}}{\littlecirc} L_b(g)=b(t(g)) g \bigl(b(s(g))\bigr)^{-1},\]
which maps units to units and $s$-fibres to $s$-fibres ($t$-fibres to $t$-fibres as well). There is an action of the jet groupoid $\jet\Gpd $ on $A$ and $TM$.
\begin{definition}\label{Def:jetGpdadjoint} 
	The {\bf adjoint action} of $\jet \Gpd$    on $A$:
	$$\Ad_{[b_g]}:\quad A_x\to A_{\phi_b(x)}, \quad\mbox{ where } ~b_g\in \mathrm{Bis}(\Gpd), ~x=s(g)$$
	is defined by$$
	\Ad_{[b_g]} u=R_{g^{-1}*}{\littlecirc} L_{[b_g]} (u)=R_{g^{-1}*}\big(L_{g*}(u-\rho(u))+b_{*}(\rho(u))\big),$$
	for all $u\in A_x$.
	
	Similarly, the adjoint action of $\jet \Gpd$    on
	$TM$:
	\[\Ad_{[b_g]}:\quad T_x M \to T_{\phi_b(x)} M,\]
	is given by \[ \Ad_{[b_g]} X=R_{[b_g]^{-1}}{\littlecirc} L_{[b_g]}(X)=\phi_{b*} (X),\quad  X\in T_x M.\]
\end{definition}
We note that  the anchor map $\rho: A\to TM$ is  equivariant with respect to the adjoint action:
\begin{eqnarray*}\label{compatible}
	\Ad_{[b_g]} {\littlecirc} \rho =\rho{\littlecirc} \Ad_{[b_g]},\quad\mbox{ as a map }A_x\to T_{\phi_b(x)}M.
\end{eqnarray*}


The following exact sequence of groupoids can be easily established:
\begin{equation}\label{Eqt:jetgroupoidexactseq}
1\to\heat  \stackrel{I}{\longrightarrow}\jet\Gpd \stackrel{P}{\longrightarrow} \Gpd \to 1.
\end{equation}
In fact, the map $[b_g]\mapsto g$ gives the projection $P:~\jet \Gpd \to \Gpd $. The space  $\heat$ consists of isotropic  jets, i.e., those of the form $[h_x]$, where   $x\in M$
and $h$ is a bisection through $x$. Let us call $\heat$ the bundle of \textbf{{isotropy} jet groups}.

We now describe $\heat$ from an alternative point of view.
For $x\in M$, with the standard decomposition
$T_x\Gpd \cong T_xM\oplus A_x$,
an  $[h_x]\in \heat_x$ can be identified  with its tangent map  $T_xM\to T_x\Gpd $ which yields identity map of $T_xM$ after composing with $s_*$, and an isomorphism of $T_xM$ after composing with $t_*$. Thus there exists $H: T_x M\rightarrow A_x$ such that
$$
h_{*}(X)=(X,H(X)),\qquad X\in T_x M,
$$and \[\phi_{h*}=t_*{\littlecirc} h_{*}=1+\rho{\littlecirc} H
\in \mathrm{GL}(T_x M).\]

Let us introduce  \[{\underline{\mathrm{Hom}}(TM,A)}:=\{H\in \mathrm{Hom}(TM,A);1+\rho{\littlecirc} H\in \mathrm{GL}(T M)\}.\]
As explained above, we can identify $\heat$ with the space ${\underline{\mathrm{Hom}}(TM,A)}$. In the sequel, for $H\in {\underline{\mathrm{Hom}}(T_xM,A_x)}$, we  write $[h_x]=1+H$ for the corresponding element in $\heat_x$.

The bundle of Lie groups structure on $\heat\cong{\underline{\mathrm{Hom}}(TM,A)}$ is thus induced:
\begin{itemize}\item[1)]The identity element in $\heat_x$ is of the form $1+0_x$, where $0_x\in {\underline{\Hom}(T_x M,A_x)}$ is the zero map; \item[2)]The multiplication of $1+H_1$ and $ 1+H_2$, where  $H_1,H_2\in  {\underline{\Hom}(T_x M,A_x)}$  is given by
	\[(1+H_1)\cdot (1+H_2)=1+H_1+H_2+H_1{\littlecirc}\rho{\littlecirc} H_2;\]
	\item[3)] The inverse is given by
	\[(1+H)^{-1}=1-H{\littlecirc} (1+\rho{\littlecirc} H)^{-1}.\]
\end{itemize}
More details can be found in \cite{C,LL}. The following lemma is about   the translation  of $\heat$ on $T\Gpd |_M=TM\oplus A$.
\begin{lemma}\label{ad for h}
	For all $[h_x]=1+H \in \heat_x$, where $H\in {\underline{\mathrm{Hom}}(T_xM,A_x)}$,   the left and right translation maps
	\[~L_{[h_x] },\quad R_{[h_x] }: T_x M\oplus A_x\to T_x M\oplus A_x\] are given by
	\begin{eqnarray}\label{Eqn:Lhx}
	L_{[h_x]}(X,u)&=&\bigl( X,H(X)+u+H \rho(u)\bigr),\\\label{Eqn:Rhx}
	R_{[h_x]}(X,u)&=&\bigl((1+\rho H)^{-1} X,H(1+\rho H)^{-1} (X)+u\bigr),
	\end{eqnarray}
	where   $X\in T_x M$ and $u\in A_x$.
\end{lemma}
\begin{proof}
	It is by definition that
	\[L_{[h_x]}(X,0)=  h_*(X)=(X,H(X)).\]
	It follows that
	\begin{eqnarray*}
		L_{[h_x]}(0,u)&=&L_{[h_x]}\bigl((-\rho(u),u)+(\rho(u),0)\bigr)=L_{[h_x]}(-\rho(u),u)
		+L_{[h_x]}(\rho(u),0)\\ &=&(-\rho(u),u)+(\rho(u),H\rho(u))=(0,u+H\rho(u)).
	\end{eqnarray*}
	This proves Equation  ~\eqref{Eqn:Lhx}.	
	Using the expression \eqref{Eqt:firstinversestar}
	of $\invstar $,  $[h_x]^{-1}=1-H{\littlecirc} (1+\rho{\littlecirc} H)^{-1}$, and Equation  ~\eqref{Eqn:Lhx} which we already proved, one can compute
	\begin{eqnarray*}
		R_{[h_x]}(X,u)&=& \invstar {\littlecirc} L_{[h_x]^{-1}}{\littlecirc} \invstar  (X,u),
	\end{eqnarray*} yielding exactly the right hand side of Equation  ~\eqref{Eqn:Rhx}.
\end{proof}

\begin{corollary}The adjoint action $\Ad_{[h_x]}:~ T_x M\oplus A_x\to T_x M\oplus A_x $ is given by
	\begin{equation*}\label{Eqt:adhXu}
	\Ad_{[h_x]}(X, u)=((1+\rho H)(X),(1+H\rho)(u)),
	\end{equation*}
	where $[h_x]=1+H \in \heat_x $.
\end{corollary}
In particular, for $1\leqslant  k\leqslant  \Top $, the induced adjoint action of $\heat$ on $\wedge^k A$ is given by
\begin{equation}\label{Eqt:adjointofheatonwedgekA}
\Ad_{[h_x]} w=L_{[h_x]}w= (1+H\rho)^{\otimes k} (w),\qquad\forall~ w\in  \wedge^k A_x.
\end{equation}

The following fact is also standard.
\begin{proposition}Let $\Gpd \toto M$ be a Lie groupoid and $A$ the associated tangent Lie algebroid. Then the  jet Lie algebroid $\jet A$ is the tangent Lie algebroid of the jet groupoid $\jet \Gpd $: $
	\jet A=\Lie (\jet \Gpd )$.
	In particular,  the bundle of {isotropy} jet Lie algebras is the tangent Lie algebroid of the bundle of {isotropy} jet groups, i.e.  $\core=\Lie\heat$.
\end{proposition}


Finally, we   have a lemma parallel to Lemma \ref{lleftHjet}.

\begin{lemma}\label{lleftHjet2}For all $1\leqslant  k\leqslant  \Top $,    the exact sequence \eqref{Eqt:jetgroupoidexactseq} of Lie groupoids gives rise to an   exact sequence of vector spaces:~
	\begin{equation*}
	0\to\coHg ^1(  \Gpd,\wedge^k \ker\rho )\stackrel{P^*}{\longrightarrow} \coHg ^1(\jet \Gpd,\wedge^k A) \stackrel{I^*}{\longrightarrow} \coHg ^1(\heat,\wedge^k A).
	\end{equation*}
	
	Moreover, we have a commutative diagram
	\begin{equation}\label{Diag:fullpicturejetH1vanEast}
	\xymatrix{
		0   \ar[r]^{ } & \coHg ^1(  \Gpd,\wedge^k \ker\rho ) \ar[d]_{\mathrm{VE}_1} \ar[r]^{P^*} & \coHg ^1(\jet \Gpd,\wedge^k A) \ar[d]_{\mathrm{VE}_2} \ar[r]^{I^*} & \coHg ^1(\heat,\wedge^k A) \ar[d]_{\mathrm{VE}_3}^{\cong} \\
		0 \ar[r]^{} & \coHg ^1(  A,\wedge^k \ker\rho ) \ar[r]^{p^*} & \coHg ^1(\jet A,\wedge^k A) \ar[r]^{i^*} & \coHg ^1(\core,\wedge^k A)  }
	\end{equation}
	where the three vertical arrows are Van Est maps.
\end{lemma}

Here, as  $\heat$ is connected and simply connected, the rightmost Van Est map $\mathrm{VE}_3\colon \coHg ^1(\heat,\wedge^k A)\to \coHg ^1(\core,\wedge^k A)$ is an isomorphism (see \cite{WX, Crainic2003}). So we can see that $\ker\mathrm{VE}_2=P^*(\ker\mathrm{VE}_1)\cong \ker\mathrm{VE}_1$. In particular, if source fibers of $\Gpd$ are connected, then $\mathrm{VE}_1$ is   injective and hence $\mathrm{VE}_2$ is also injective.


\begin{thebibliography}{10}
	
	
	
	\bibitem{VE} C. Arias Abad and M. Crainic, The Weil algebra and the Van Est isomorphism, {\it Ann. Inst. Fourier. (Grenoble)} 61 (2011), no. 3, 927-970.
	\bibitem{AC2} C. Arias Abad and M. Crainic, Representations up to homotopy of Lie algebroids, {\it J. Reine Angew. Math.} 663 (2012), 91-126.
	\bibitem{AC} C. Arias Abad and M. Crainic, Representations up to homotopy and Bott's spectral sequence for Lie groupoids, {\it Adv. Math.} 248 (2013), 416-452.
	
	\bibitem{BLlie2algebra} D. Berwick-Evans and E. Lerman, Lie 2-algebras of vector fields,   {\it Pacific J. Math.}   309 (2020), 1-34.
	
	
	\bibitem{BCLX} F. Bonechi, N. Ciccoli, C. Laurent-Gengoux and P. Xu, Shifted Poisson structures on differential stacks,
	 {\it Int. Math. Res. Not. IMRN} (2020), rnaa293, https://doi.org/10.1093/imrn/rnaa293.
	
	
	
	\bibitem{BC} H. Bursztyn and A. Cabrera, Multiplicative forms at the infinitesimal level, {\it Math. Ann.}  353 (2012), 663-705.
	\bibitem{BCO} H. Bursztyn, A. Cabrera and C. Ortiz, Linear and multiplicative 2-forms, {\it Lett. Math. Phys.} 90 (2009), no. 1-3, 59-83.
	\bibitem{BCWZ} H. Bursztyn, M. Crainic, A. Weinstein and C. Zhu, Integration of twisted Dirac brackets, {\it Duke Math. J.} 123 (2004), no. 3, 549-607.	
	\bibitem{BD} H. Bursztyn and T. Drummond, Lie theory of multiplicative tensors, {\it Math. Ann.} (2019), no. 3-4, 1489-1554.	
	
	\bibitem{CMS1} A. Cabrera, I. Marcut and M. A. Salazar, Local formulas for multiplicative forms, {\it Transformation Groups} (2020), DOI: 10.1007/S00031-020-09607-y.
	
	\bibitem{CLL2022}  Z. Chen, H. Lang and Z. Liu, Multiplicative forms on Poisson groupoids, arXiv: 2201.06242.
	\bibitem{CL}  Z. Chen and Z. Liu, On transitive Lie bialgebroids and Poisson groupoids, {\it Differential Geom. Appl.} 22 (2005), 253-274.
	\bibitem{CSX} Z. Chen, M. Sti\'enon and P. Xu, Poisson 2-groups, {\it J. Differential Geom.}  94 (2013), no. 2, 209-240.
	
	\bibitem{Crainic2003} M. Crainic,  Differentiable and algebroid cohomology, Van Est isomorphisms, and characteristic classes, {\it Comment. Math. Helv.} 78 (2003),  no. 4, 681-721.
	
	
	
	\bibitem{CMS} M. Crainic, J. Mestre and I. Struchiner, Deformations of Lie groupoids, {\it Int. Math. Res. Not. IMRN} (2020), no. 21, 7662-7746. 
	
	\bibitem{Cdef} M. Crainic and I. Moerdijk, Deformations of Lie brackets: cohomological aspects, {\it J. Eur. Math. Soc. (JEMS)} 10 (2008), no. 4, 1037-1059.
	
	\bibitem{C} M. Crainic, M. A. Salazar and I. Struchiner, Multiplicative forms and Spencer operators, {\it Math. Z.}  279 (2015), no. 3-4, 939-979.
	\bibitem{CZ} M. Crainic and C. Zhu, Integrability of Jacobi and Poisson structures,  {\it Ann. Inst. Fourier (Grenoble)} 57 (2007), 1181-1216.
	\bibitem{Drinfeld} V. Drinfeld, Hamiltonian structures on Lie groups, Lie bialgebras and the geometric meaning
of the classical Yang-Baxter equations, {\it Soviet Math. Dokl.} 27 (1983), 68-71.	

\bibitem{DJO} T. Drummond, M. Jotz Lean and C.  Ortiz,  VB-algebroid morphisms and representations up to
homotopy, {\it Differential Geom. Appl.} 40 (2015), 332-357.

\bibitem{EML} S. Eilenberg and S. Mac Lane, On the groups $H(\Pi;n)$. I, {\it Ann. of Math.} (2) 58 (1953), 55-106.
	
	\bibitem{ETV} C. Esposito, A. Tortorella and L. Vitagliano,  Infinitesimal automorphisms of VB-groupoids and
algebroids, {\it Q. J. Math.} 70 (2019), no. 3, 1039-1089.
	\bibitem{EL} S. Evens, J.-H. Lu and A. Weinstein, Transverse measures, the modular class, and a cohomology pairing for Lie algebroids, {\it Quart. J. Math. Oxford Ser. (2)} 50 (1999), no. 200, 417-436.

 \bibitem{GM} A. Gracia-Saz and R. A. Mehta, VB-groupoids and representation theory of Lie groupoids, {\it J. Symplectic Geom.} 15 (2017), no. 3, 741-783.



	\bibitem{H} E. Hawkins,  A groupoid approach to quantization, {\it J. Symplectic Geom.} 6 (2008), no. 1, 61-125.	
	
	
	\bibitem{ILX} D. Iglesias-Ponte, C. Laurent-Gengoux and P. Xu,  Universal lifting theorem and quasi-Poisson groupoids,
	{\it J. Eur. Math. Soc. (JEMS)}  14 (2012), no. 3, 681-731.
	\bibitem{IM} D. Iglesias-Ponte and J. C. Marrero,  Jacobi groupoids and generalized Lie bialgebroids, {\it J. Geom.
Phys.} 48 (2003), 385-425.	
\bibitem{JL} M. Jotz Lean, Dirac groupoids and Dirac bialgebroids, {\it J. Symplectic Geom.} 17 (2019), no. 1, 179-238.
\bibitem{JO} M. Jotz Lean and C. Ortiz, Foliated groupoids and infinitesimal ideal systems, {\it Indag. Math.
(N.S.)} 25 (2014),  no. 5, 1019-1053.
\bibitem{JSX} M. Jotz Lean, M. Sti\'enon and P. Xu, Glanon groupoids, {\it Math. Ann.} 364 (2016), no. 1-2, 485-518.


		\bibitem{Karasev} M.  Karasev, Analogues of objects of the theory of Lie groups for nonlinear Poisson brackets,
{\it Math. USSR-Izv.} 28 (1987), no. 3, 497-527.		
	\bibitem{Kosmann} Y. Kosmann-Schwarzbach, Multiplicativity, from Lie groups to generalized geometry, Geometry of jets and fields, 131-166, {\it Banach Center Publ.}, 110, Polish Acad. Sci. Inst. Math., Warsaw, 2016.
	

	\bibitem{LL} H. Lang and Z. Liu, Coadjoint orbits of Lie groupoids, {\it J. Geom. Phys.} 129 (2018), 217-232.
\bibitem{LSX1} C. Laurent-Gengoux, M. Sti\'enon and P. Xu, Holomorphic Poisson manifolds and Lie algebroids, {\it Int. Math. Res. Not. IMRN} 2008, Art. ID rnn 088, 46 pp.
	\bibitem{XuCamile} C. Laurent-Gengoux, M. Sti\'enon and P. Xu, Integration of holomorphic Lie algebroids, {\it Math. Ann.} 345 (2009), 895-923.
	\bibitem{LS} D. Li-Bland and P. Severa, Quasi-Hamiltonian groupoids and multiplicative Manin pairs, {\it Int. Math. Res. Not. IMRN} 2011, no. 10, 2295-2350.	
	\bibitem{LuThesis} J.-H. Lu, Multiplicative and affine Poisson structures on Lie groups, Ph. D thesis, UC Berkeley, 1990.	
	\bibitem{LuW} J.-H. Lu and A. Weinstein, Poisson Lie groups, dressing transformations, and Bruhat decompositions, {\it J. Differential Geom.} 31 (1990), 501-526.
	
	
	
		
	
		
	\bibitem{Mackenzie}K. Mackenzie, General Theory of Lie Groupoids and Lie Algebroids, London Mathematical Society Lecture Note Series, vol. 213, Cambridge University Press.
	\bibitem{MX1} K. Mackenzie and P. Xu,  Lie bialgebroids and Poisson groupoids, {\it Duke Math. J.} 73 (1994), 415-452.
       \bibitem{MX2} K. Mackenzie and P. Xu, Classical lifting processes and multiplicative vector fields, {\it Quart. J. Math. Oxford Ser. (2)} 49 (1998), 59-85.	
	\bibitem{Manetti}M. Manetti, {A relative version of the ordinary perturbation lemma},
		 {\it Rend. Mat. Appl. }(7) 30 (2010) no.2, 221-238.
	\bibitem{O} C. Ortiz,  Multiplicative Dirac structures, {\it Pacific J. Math.} 266 (2013), 329-365.		
	
	\bibitem{OrtizWaldron17} C. Ortiz and J. Waldron, On the Lie 2-algebra of sections of an $LA$-groupoid, {\it J. Geom. Phys.} 145 (2019), 103474, 34pp.
	
	
	

	
	
	
\bibitem{W2} A. Weinstein, Symplectic groupoids and Poisson manifolds, {\it Bull. Amer. Math. Soc. (N.S.)} 16
(1987), 101-104.

\bibitem{W1} A. Weinstein,  Coisotropic calculus and Poisson groupoids, {\it J. Math. Soc. Japan} 40 (1988),
705-727.









\bibitem{Waffine} A. Weinstein, Affine Poisson structures, {\it Internat. J. Math.} 1 (1990), 343-360.	
	
 \bibitem{WX} A. Weinstein and P. Xu, Extensions of symplectic groupoids and quantization, {\it J.
Reine Angew. Math.} 417 (1991), 159-189.	
	
	\bibitem{XiangSolo2021}  M. Xiang,    Atiyah and Todd classes of regular Lie algebroids,   arXiv: 2110.04720.
	
	
\end{thebibliography}
\end{document}